\documentclass[11pt]{amsart}

\usepackage{amsthm,amsmath,amssymb,amsrefs}
\RequirePackage{hyperref}
\usepackage[mathscr]{eucal}
\usepackage{enumerate}
\usepackage{verbatim}

\usepackage{epstopdf}
\usepackage{subcaption}
\usepackage{graphicx}
\usepackage{verbatim}
\usepackage{amsmath,amsthm}
\usepackage{graphics}
\usepackage{color}
\usepackage{epsfig}
\usepackage{amssymb,amsmath}
\usepackage[mathscr]{eucal}
\usepackage{ upgreek }
\usepackage{tensor}
\usepackage{relsize}

\usepackage{appendix}
\usepackage{amssymb,amsmath}
\usepackage[mathscr]{eucal}
\usepackage{color}
\usepackage{enumitem}
\usepackage[utf8]{inputenc}
\usepackage[leftcaption]{sidecap}
\usepackage{tikz, pgfplots}
 \usepackage{tikz-cd}
 \usepackage{cancel}
 \usepackage{stmaryrd}
\usetikzlibrary{matrix,arrows,decorations.pathmorphing}
\usetikzlibrary{positioning}
\usetikzlibrary{matrix,arrows,decorations.pathmorphing, shapes.geometric}
\tikzset{>=latex}

\usepackage{psfrag}

\DeclareFontFamily{U}{mathx}{\hyphenchar\font45}
\DeclareFontShape{U}{mathx}{m}{n}{
      <5> <6> <7> <8> <9> <10>
      <10.95> <12> <14.4> <17.28> <20.74> <24.88>
      mathx10
      }{}
\DeclareSymbolFont{mathx}{U}{mathx}{m}{n}
\DeclareFontSubstitution{U}{mathx}{m}{n}
\DeclareMathAccent{\widecheck}{0}{mathx}{"71}
\usepackage{psfrag}

\newtheorem{theorem}[equation]{Theorem}
\newtheorem{lemma}[equation]{Lemma}
\newtheorem{prop}[equation]{Proposition}
\newtheorem{cor}[equation]{Corollary}
\newtheorem{corollary}[equation]{Corollary}
\newtheorem{conjecture}[equation]{Conjecture}
\newtheorem{question}[equation]{Question}

\newtheorem{definition}[equation]{Definition}
\newtheorem{example}[equation]{Example}

\newtheorem{notation}[equation]{Notation}
\newtheorem{assumption}[equation]{Assumption}

\theoremstyle{remark}
\newtheorem{remark}[equation]{Remark}
\newtheorem{convention}[equation]{Convention}

\numberwithin{equation}{section}

\newcommand{\R}{\mathbb{R}}
\newcommand{\C}{\mathbb{C}}
\newcommand{\F}{\mathbb{F}}
\newcommand{\Z}{\mathbb{Z}}
\newcommand{\D}{\mathbb{D}}
\newcommand{\B}{{\mathbb{B}}}
\newcommand{\N}{\mathbb{N}}
\newcommand{\Sph}{\mathbb{S}}
\newcommand{\T}{\mathbb{T}}
\newcommand{\K}{\mathbb{K}}
\newcommand{\Kbb}{{\mathbb{K}}}

\newcommand{\Ric}{\text{Ric}}

\renewcommand{\Re}{\operatorname{Re}}
\renewcommand{\Im}{\operatorname{Im}}
\newcommand{\ind}{\operatorname{Ind}}
\newcommand{\nul}{\operatorname{Null}}
\newcommand{\Int}{\operatorname{Int}}
\newcommand{\area}{\operatorname{Area}}
\newcommand{\genus}{\operatorname{genus}}
\newcommand{\supp}{\operatorname{supp}}

\newcommand{\dbold}{{\mathbf{d}}}
\newcommand{\stab}{\operatorname{Stab}}
\newcommand{\Grp}{\mathscr{G}} 
\newcommand{\group}{{\mathscr{G}  }}
\newcommand{\grouptilde}{\widetilde{\mathscr{G}  }}
\newcommand{\bfSigma}{{\mathbf{\Sigma}}} 
\newcommand{\bfSigmap}{{\bfSigma'}} 
\newcommand{\GrpSigma}{\mathscr{G}^{{\bfSigma}}} 
\newcommand{\GrpSigmam}{\mathscr{G}^{{\bfSigma}[m]}} 
\newcommand{\GrpSigmap}{\mathscr{G}^{{\bfSigmap}}} 
 
\newcommand{\GrpSigmapm}{\mathscr{G}^{{\bfSigmap}[m]}}

\newcommand{\Gsym}{\mathscr{G}_{sym}}

\newcommand{\Lcal}{\mathcal{L}}

\newcommand{\Rcal}{\mathcal{R}}
\newcommand{\Ecal}{\mathcal{E}}

\newcommand{\CC}{{{{\mathscr{C}}}}}
\newcommand{\cC}{{{\CC}}}

\newcommand{\CCperp}{{{\CC}^\perp}} 

\newcommand{\fhat}{{\hat{f}}}
\newcommand{\ftilde}{{\tilde{f}}}
\newcommand{\ftildep}{{\tilde{f}}_+}

\newcommand{\phicheck}{{\widecheck{\phi}}}

\newcommand{\psihat}{{\hat{\psi}}}

\newcommand{\mubreve}{{\check{\mu}}}

\newcommand{\osc}{{{\mathrm{osc}}}}
\newcommand{\avg}{{{\mathrm{avg}}}}

\newcommand{\ii}{\ensuremath{\mathrm{i}}}

\newcommand{\xx}{\ensuremath{\mathrm{x}}}
\newcommand{\yy}{\ensuremath{\mathrm{y}}}
\newcommand{\zz}{\ensuremath{\mathrm{z}}}
\newcommand{\xxx}{{\ensuremath{\mathring{\mathrm{x}}}}}
\newcommand{\yyy}{{\ensuremath{\mathring{\mathrm{y}}}}}
\newcommand{\zzz}{{\ensuremath{\mathring{\mathrm{z}}}}}
\newcommand{\YY}{\ensuremath{\mathrm{\Phi}}}
\newcommand{\rr}{\ensuremath{\mathrm{r}}}

\newcommand{\rp}{\ensuremath{\mathrm{r}_p}}
\newcommand{\rA}{\ensuremath{\mathrm{r}_A}}
\newcommand{\sss}{\ensuremath{\mathrm{s}}}
\newcommand{\ssst}{\ensuremath{{\mathrm{s}}}}

\newcommand{\sroot}{{\sss_{\mathrm{root}}}}
\newcommand{\xxmax}{{\xx_{\mathrm{max}}}}
\newcommand{\Cnodal}{\CC_{\mathrm{nodal}}}

\newcommand{\sbar}{{\underline{\sss}}}
\newcommand{\sbars}{{\sbar_*}}
\newcommand{\stil}{{\tilde{\sss}}}

\newcommand{\rbar}{{\underline{\rr}}}
\newcommand{\cbar}{{\underline{c}}}
\newcommand{\cbarp}{{\underline{c}'}}

\newcommand{\sbarroot}{{\sbar_{\mathrm{root}}}}
\newcommand{\fp}{f'}
\newcommand{\foY}{f_{1,\YY_0}}
\newcommand{\fo}{f_{1}}
\newcommand{\ftY}{f_{2,\YY_0}}
\newcommand{\fhattY}{\fhat_{2,\YY_0}}
\newcommand{\ft}{f_{2}}
\newcommand{\fC}{f_\CC}

\newcommand{\up}{u'}

\newcommand{\jbold}{\boldsymbol{j}}
\newcommand{\mbold}{\boldsymbol{m}}

\newcommand{\ssshat}{{{\widehat{\sss}}}}

\newcommand{\thetatilde}{{\widetilde{\theta}}}

\newcommand{\Mbreve}{{\bar{\bar{M}}}}
\newcommand{\Mbr}{{\bar{\bar{M}}}}
\newcommand{\chibar}{{\bar{\bar{\chi}}}}
\newcommand{\Mhat}{\Mbreve}

\newcommand{\LcalM}{\Lcal_\Mbreve}
\newcommand{\LcalS}{\Lcal_\Sigma}

\newcommand{\chiK}{{{\chi}}}
\newcommand{\chiM}{{{\chi_\Mbreve}}}
\newcommand{\hK}{{\bar{h}_{\Kbb}}}
\newcommand{\hM}{{h_{\Mbreve}}}

\newcommand{\Rhat}{{\widehat{\Rcal}}}

\newcommand{\pp}{\mathrm{p}}

\newcommand{\Sigmau}{\underline{\Sigma}}
 
\newcommand{\GrpSigmamu}{\mathscr{G}^{\mathbf{\Sigmau}[m]}} 
\newcommand{\ptti}{\pp_{ {i\pi/{m}} }} 
\newcommand{\pttj}{\pp^{ {j\pi/{2}} }}

\newcommand{\rot}{\mathsf{R}}
\newcommand{\refl}{{\underline{\mathsf{R}}}}
\newcommand{\homot}{\mathsf{H}}

\newcommand{\Om}{\Omega} 
\newcommand{\Udom}{U} 
\newcommand{\Usq}{U_{J}} 
\newcommand{\Ustar}{U_{*}} 
\newcommand{\fstar}{f_{*}} 
\newcommand{\psistar}{\psi_{*}} 
\newcommand{\Ustarp}{U'_{*}} 
\newcommand{\fstarp}{f'_{*}} 
\newcommand{\UC}{U_\CC} 
\newcommand{\Omu}{{\underline{\Omega}}} 
\newcommand{\Abold}{\mathbf{A}} 
\newcommand{\Ombold}{\mathbf{\Omega}} 
\newcommand{\Omboldu}{\underline{\mathbf{\Omega}}} 
 
\newcommand{\gu}{{\underline{g}}} 
\newcommand{\Ut}{{{{\widetilde{U}}}}} 
\newcommand{\Utp}{{{{\widetilde{U}}_1}}} 
\newcommand{\gt}{{\widetilde{g}}} 

\newcommand{\Lcalh}{{{\Lcal}_h}}

\newcommand{\LcalhN}{{{\Lcal}_h^N}}

\newcommand{\LchiK}{{\mathcal{L}_{\chiK}}} 
\newcommand{\LchiM}{{\mathcal{L}_\chiM}}

\newcommand{\LhM}{{\mathcal{L}_\hM}}
\newcommand{\LhK}{{\mathcal{L}_\hK}}
 
\newcommand{\LcalhD}{{{\Lcal}_h^D}} 
\newcommand{\LhMD}{{\mathcal{L}_\hM^D}}
\newcommand{\Bcal}{\mathcal{B}}

\newcommand{\Lint}{{{\Lcal}^\circ}} 
 
\newcommand{\BbouR}{{{\Bcal}^R}} 
\newcommand{\BbouD}{{{\Bcal}^D}} 
\newcommand{\BbouN}{{{\Bcal}^N}}

\newcommand{\Lcalu}{{\underline{\Lcal}}}

\newcommand{\Lcalt}{{\widetilde{\Lcal}}} 
\newcommand{\LcaltD}{{{\widetilde{\Lcal}}^D}} 
\newcommand{\LcaltN}{{{\widetilde{\Lcal}}^N}} 
\newcommand{\Lintt}{{{{{\Lcalt}^\circ}}}} 
\newcommand{\Linttp}{{{\left.{{\Lcalt}^\circ}\right|_{\Ut'}}}} 
\newcommand{\Bbout}{{\widetilde{\Bcal}}} 
\newcommand{\BboutD}{{{\widetilde{\Bcal}}^D}} 
\newcommand{\BboutN}{{{\widetilde{\Bcal}}^N}}

\newcommand{\Vt}{{\widetilde{V}}}

\newcommand{\cone}{\mbox{$\times \hspace*{-0.244cm} \times$}}
\newcommand{\Span}{\operatorname{Span}}

\newcommand{\CCslash}{{\CC\mspace{.8mu}\!\!\!\!\boldsymbol{/}\,}}
\newcommand{\CCbackslash}{{\CC\mspace{.8mu}\!\!\!\!\boldsymbol{\backslash}\,}}

\newcommand{\Lmer}{{L_{mer}}}
\newcommand{\Lpar}{{L_{par}}}
\newcommand{\Lpol}{{L_{2}}}
\newcommand{\Leq}{{L_{0}}}

\newcommand{\kcir}{k_{\mathrm{cir}}}
\newcommand{\kpol}{k_{\mathrm{pol}}}
\newcommand{\kexc}{k_{\mathrm{excess}}}

\newcommand{\xL}{\boldsymbol{{\xi}}}   
\newcommand{\xbreve}{\bar{\bar{\boldsymbol{\xi}}}}   
\newcommand{\xibreve}{\xbreve}   
\newcommand{\xring}{\boldsymbol{{\xi}}^{\T}}   
\newcommand{\xbrevering}{\xibreve^{\T}}   
\newcommand{\xiomo}{\xbreve_{(1,m,1)}} 
\newcommand{\xibreveomo}{\xiomo} 
\newcommand{\xibrevemm}{\xbreve_{(m,m)}} 
\newcommand{\xbreveum}{{\xbreve}_{(m)}^\D}  
\newcommand{\xbreveutwo}{{\xbreve}_{(2)}^\D}  
\newcommand{\xbreveump}{{\xbreve}_{(1,m)}^\D}  
\newcommand{\Kbreve}{\bar{\bar{K}}}
   
\newcommand{\Abreve}{\bar{\bar{A}}}
\newcommand{\Kbr}{\Kbreve}
\newcommand{\Xbreve}{\bar{\bar{X}}}
\newcommand{\Xbr}{\Xbreve}
\newcommand{\ubreve}{\bar{\bar{u}}}
\newcommand{\ubreveM}{\ubreve}

\newcommand{\Sigmabreve}{\bar{\bar{\Sigma}}} 
\newcommand{\sym}{\mathrm{sym}}
\newcommand{\gr}{{\mathrm{gr}}}

\newcommand{\SigmaU}{{{\Sigmabreve}_\gr}} 
\newcommand{\ML}{\xL_{m-1,1}}

\newcommand{\PP}{\underline{P}}
\newcommand{\PV}{{P}}
\newcommand{\mbar}{{{m}\mspace{.8mu}\!\!\!\!\!\boldsymbol{/}\,}}
\newcommand{\mtwo}{{\boldsymbol{\underline{\mu}}}}
\newcommand{\mo}{{\boldsymbol{\overline{\mu}}}}
\newcommand{\mi}{{\mu}}
\newcommand{\mh}{{(m/2)}}
\newcommand{\Oplus}{{{\textstyle{\bigoplus}}}}

\newcommand{\disjun}{\textstyle\bigsqcup}
\newcommand{\Thetahat}{\widehat{\Theta}}
\newcommand{\Thetaring}{\mathring{\Theta}}
\newcommand{\Thetaringhat}{\widehat{{\Thetaring}}}
\newcommand{\Thetacyl}{\Theta_{\cyl}}
\newcommand{\cyl}{\ensuremath{\mathrm{Cyl}}}
\newcommand{\Spheq}{\mathbb{S}^2_{\mathrm{eq}}} 
\newcommand{\Sz}{\Sigma^0}           
\newcommand{\ThetaSphcyl}{\Theta^{\cyl}_{\Sz}}
\newcommand{\sech}{\operatorname{sech}}
\newcommand{\csch}{\operatorname{csch}}
\newcommand{\PiSph}{\Pi_{\Sz   }}

\newcommand{\tildecat}{{\mathbb{K}}}
\newcommand{\Kbt}{{\mathbb{K}[\tau]}}

\newcommand{\XKt}{{X_{\Kbb[\tau]}}}
\newcommand{\gS}{{g_{\Sigma^0}}}
\newcommand{\gSigma}{{g_{\Sigma}}}
\newcommand{\gSph}{{g_{\Sph^2}}}
\newcommand{\gnuK}{{\nu_\K^* \gSph}}
\newcommand{\gK}{{g_{\Kbb}}}
\newcommand{\gM}{{g_{\Mbreve}}}
\newcommand{\AK}{{A_{\Kbb}}}
\newcommand{\nuK}{{\nu_{\Kbb}}}
\newcommand{\phicat}{\varphi_{\mathrm{cat}}}
\newcommand{\arccosh}{\operatorname{arccosh}}
\newcommand{\OL}{{{U}_{L}}}
\newcommand{\Sbreve}{{\bar{\bar{S}}}} 
\newcommand{\SbreveS}{\Sbreve_\Sigma}
\newcommand{\SbreveL}{\Sbreve_L}
\newcommand{\KbrL}{\Kbr_L}

\newcommand{\Opar}{{{\Sigmabreve}_{par}}}
\newcommand{\ev}{{{\mathrm{ev}}}}
\newcommand{\od}{{{\mathrm{od}}}}
\newcommand{\en}{{{\mathrm{end}}}}
\newcommand{\sn}{{{\mathrm{sng}}}}
\newcommand{\phie}{{\phi_{\ev}}}
\newcommand{\phio}{{\phi_{\od}}}
\newcommand{\phiN}{{\phi_{N}}}
\newcommand{\phiD}{{\phi_{D}}}
\newcommand{\phien}{{\phi_{\en}}}
\newcommand{\phisn}{{\phi_{\sn}}}

\newcommand{\phitp}{{\phi_{\mathrm{ext}}}}
\newcommand{\phiext}{{\phitp}}

\newcommand{\aunder}{{\underline{a}}}
\newcommand{\llcup}{{ \, {{\mathlarger{\mathlarger{\cup}}}} \, }} 
\newcommand{\llcap}{{ \, {{\mathlarger{\mathlarger{\cap}}}} \, }} 
\newcommand{\lllcup}{{ \, {\mathlarger{\mathlarger{\mathlarger{\cup}}}} \, }} 

\newcommand{\inter}{\operatorname{Interior}}

\newcommand{\vecv}{{{\vec{v}}}}  
\newcommand{\vecf}{{{\vec{f}}}}  
\newcommand{\vecfp}{{{\vec{f'}}}}  
\newcommand{\vecfpp}{{{\vec{f''}}}}  
  
\newcommand{\vecK}{{{\vec{K}}}}  
  
\newcommand{\Cperp}{{{C}^\perp}} 
 
\newcommand{\Sim}{\displaystyle\operatornamewithlimits{{\scalebox{1.296}{{$\, \sim \, $}}}}}  
\newcommand{\disint}{{{\displaystyle\int}}}  
\newcommand{\psicut}{{\psi_{\mathrm{cut}}}}
\newcommand{\Psibold}{{\boldsymbol{\Psi}}}

\newcommand{\UY}{{\bar{\bar{\YY}}    }}

\newcommand{\taubreve}{\bar{\bar{\tau}}}

\newcommand{\Omegahat}{\Sigmabreve}
\newcommand{\Dbreve}{\bar{\bar{D}}}

\newcommand{\taubold}{{\boldsymbol{\tau}}}
\newcommand{\con}{{{(con)}}}
\newcommand{\Sl}{{{(sl)}}}
\newcommand{\subker}{\mathscr{K}_{{\mathrm{sub}}}}  

\newcommand{\und}{\underline{\phantom{i}}}

\title[Index and nullity]{Index and nullity \\ of minimal surface doublings, I} 

\author[N.~Kapouleas]{Nikolaos~Kapouleas}
\address{Department of Mathematics, Brown University, Providence, RI 02912}
\email{nicolaos\_kapouleas@brown.edu}

\author[J.~Zou]{Jiahua~Zou} 
\address{Department of Mathematics, Rutgers University, Pistacaway, NJ 08854} 
\email{jiahua.zou@rutgers.edu}

\date{\today}

\begin{document}

\begin{abstract}
We prove that 
for any large enough $m \in\mathbb{N}$, 
the genus $\gamma=m+1$  
equator-poles minimal surface doubling 
%$\boldsymbol{\bar{\bar{\xi}}}_{(1,m,1)}$ 
of the equatorial two-sphere $\Sigma^0 = \mathbb{S}^2_{\mathrm{eq}}$ in the round three-sphere $\mathbb{S}^3$, 
which has two catenoidal bridges at the poles and  $m$ bridges equidistributed along the equatorial circle $\mathscr{C}$ of $\Sigma^0                  $
and was discovered in earlier work of Kapouleas, 
has index $2\gamma+5=2m+7$ and nullity $6$, 
and so it has no exceptional Jacobi fields and is $C^1$-isolated.   
\end{abstract}

\maketitle

\section{Introduction}
\label{S:intro}

\paragraph{\bf{General framework}}
$\phantom{ab}$ 
$\vspace{.072cm}$ 

Determining the index and nullity of complete  or closed minimal surfaces is an important but difficult problem 
which has been fully solved only in a few cases; 
see for example \cites{nayatani1992,nayatani1993, morabito} 
and surveys \cite{brendle:survey,urbano:2025:book,kapouleas:nisyros}. 
In the case of closed embedded minimal surfaces of genus $\gamma\ge2$ the only surfaces of known index are the Lawson surfaces $\xL_{\gamma,1}$ \cite{Lawson}: 
their index was recently determined \cite{LindexI} to be $2\gamma+3$ and 
their nullity to be $6$, and so they have no exceptional Jacobi fields and are $C^1$-isolated. 
The purpose of this article and its sequel \cite{IIDindex} is to determine the index and nullity of various \emph{minimal surface doublings of gluing type} in $\Sph^3$ 
and \emph{free boundary minimal (FBM) surface doublings} in the unit Euclidean ball $\B^3$. 
We briefly recall  now the minimal surface doublings we are interested in. 

Minimal surface doublings are \emph{smooth} minimal surfaces which are unions of two (sometimes more) graphs over (smooth closed domains of) a given minimal surface $\Sigma$ 
called the \emph{base surface $\Sigma$} 
(see for example \cite[Definition 1.1]{gLD}). 
When constructed by PDE gluing methods they are of \emph{gluing type} and resemble two or more copies of $\Sigma$ joined by small \emph{catenoidal bridges} 
\cite{kapouleas:yang,wiygul:jdg2020,SdI,SdII,gLD,IIgLD} or small \emph{half-catenoidal bridges} (when at the boundary in the FBM case) \cite{FPzolotareva,tripling}.  
They should not be confused with \emph{minimal surface desingularizations} constructed by PDE gluing methods, 
as for example in \cite{CCD,choe:soret,kapouleas:wiygul:toridesingularization,k1,kapouleas:general}.  

The simplest high-genus doublings of a great two-sphere $\Sigma=\Sigma^0=\Spheq$ in $\Sph^3$ are the \emph{two-circle doublings} \    $\xibrevemm $  (with $\kcir=2$, $\kpol=0$) \     
and the \emph{equator-poles doublings} \    $\xibreveomo$ (with $\kcir=1$, $\kpol=2$), \   
discovered for large \   $m\in\N$ \   by Kapouleas in \cite{SdI}.  
Generalizing the construction in \cite{SdI} to any     $\kcir\in\N$,   Kapouleas-McGrath discovered in \cite{SdII} 
minimal surface doublings \     $\xbreve_{(m:\kcir)}$ (with $\kcir>2$, $\kpol=0$) \     and \     $\xbreve_{(1,m:\kcir,1)}$ (with $\kcir>1$, $\kpol=2$) \   
for any \   $m\in \N$ \   large enough in terms of \   $\kcir$. \ 
(Here we use $\xibreve$ to denote minimal surface doublings in analogy with     Lawson's notation in \cite{Lawson}; $  m:\kcir $   simply means $m$ repeated $\kcir$ times.)  
We recall that the above doublings have genus \ $\gamma  = m\kcir+\kpol-1$ \ and 
contain \ $m\kcir$ \ catenoidal bridges located at the intersections of $m$ symmetrically arranged meridians with $\kcir$ symmetrically arranged parallel circles (see \ref{EL}),   
and when $\kpol=2$ two extra bridges at the poles. 
Their     symmetry group \   $\GrpSigmam\simeq\Z_2\times \Z_2\times D_{2m}$ \   (see \ref{Egroup}) is an index two subgroup of the symmetry group of the Lawson surface \   $\xL_{m-1,1}$.

In \cite{gLD,IIgLD} many more minimal surface doublings of $\Sigma=\Spheq$ have been discovered 
where the number and alignment of bridges concentrating along the $\kcir$ parallel circles can vary 
and there are $\kpol\in\{0,1,2\}$ bridges at the poles: 
we denote these by $\xbreve_{\mbold}  $, 
where $\mbold\in \N^{\kcir+\kpol}$ determines the number and alignment of bridges at the circles and poles as in        \cite[9.18]{gLD}.  
Of interest to us are the doublings      \ $\xbreve_{(m:\kcir,1)}$,    \  constructed by PDE gluing methods in \cite{IIgLD} for any $m$ large enough in terms of any given $\kcir\in\N$,   
with \ $m\kcir$ \    bridges at the intersections of $m$ symmetrically arranged meridians with $\kcir$ parallel circles (not symmetrically arranged) and an extra bridge at one of the poles. 
The   symmetry group of \ $\xbreve_{(m:\kcir,1)}$     \  
is \     $\GrpSigmapm\simeq \Z_2\times D_{2m}$, \     an index two subgroup of  $\GrpSigmam$  (see \ref{Egroupp}).

For completeness we mention also that a large class of new minimal surface doublings of $\Sigma=\Spheq$ has been found recently by a variational method \cite{KKMS};  
these doublings maximize their normalized first eigenvalue subject to various symmetry and topological constraints.  
Many doublings of the Clifford torus $\T$ have been discovered also 
\cite{kapouleas:yang,wiygul:jdg2020,gLD,IIgLD,ketoverMN:2020,k35}.  
In particular we denote by $\xbrevering_{(m:k)     }$ a Clifford torus minimal doubling 
containing $km$ bridges centered at the points of a $k\times m$ lattice ($k,m\in\N$) in $\T$;  
such a doubling was first discovered in \cite{kapouleas:yang} for $k=m$ large,
in \cite{wiygul:jdg2020} for $m$ large enough in terms of a bound on $|m/k|$,     
and in \cite{gLD} for $m$ large enough in terms of any $k\in [3,m]$. 
We mention in passing that desingularizations of Clifford tori have also been discovered \cite{choe:soret,kapouleas:wiygul:toridesingularization,KW2024}, 
the simplest of which (desingularizing $k$ tori) we denote by $\xring_{k,m}$ (with $n_1=n_2=\sigma=1$ in the notation of \cite{kapouleas:wiygul:toridesingularization}).        

Folha-Pacard-Zolotareva \cite{FPzolotareva}, following great progress    in the subject of FBM surfaces 
as for example in \cite{fraser-schoen:2016:invent,fraser-schoen:2011:advances,fraser-schoen:2015:IMRN,fraser-li:2014:JDG,chen-fraser-pang:2015:index:TAMS},  
and the doubling constructions in \cite{kapouleas:yang,SdI}, 
constructed for any   large \ $m\in\N$ \   
FBM surface doublings of the equatorial disc  \       $\D\subset \B^3$,  \       which we call \       $\xbreveum$  \       and  \       $\xbreveump$. \ 
They     both have symmetry group  \       $\GrpSigmamu$  \       isomorphic to  \       $\Z_2\times D_{2m}$,    \        
and both have $m$ boundary components and contain $m$ half-catenoidal bridges equidistributed along $\partial\D$.  
$\xbreveum$  \       is of genus zero and  \       $\xbreveump$ \ of genus one with an extra catenoidal bridge at the center of $\D$.

\begin{remark}[The notation $\xibreve$]   
\label{rem} 
Note that the doublings in $\Sph^3$ or $\B^3$ discussed above are strongly expected to be unique by their descriptions 
and it is common to refer to them as if this was known to be so \cite[Remark 5.23]{gLD}. 
They are also expected or known by other methods to exist in cases where 
the bridges are not small perturbations of truncated catenoids 
(they are not of gluing type  in the sense of \ref{dgtSS}),  
for example (assuming uniqueness and graphical property even if not known yet) we can identify \   $\xibreve_{(m)}$ \   with Lawson's \   $\xL_{m-1   ,1}$ \   in \cite{Lawson}, or  
\  $\xbreve_{(1,3)}$, \  $\xbreve_{(4,4)}$, \  $\xbreve_{(1,4,1)}$,  \  $\xbreve_{(1,5,-5,1)}$,  \  and \  $\xbreve_{(5,5,-5,-5)}$ \  
with the closed minimal surfaces of genus $3$, $7$, $5$, $11$, and $19$ of Karcher--Pinkall--Sterling in \cite{KPS},   
or 
\ $\xbreve_{(m:\kcir)}$, \ $\xbreve_{(1,m:\kcir,1)}$, \ and \ $\xbreve_{(m:\kcir,1)}$  \    (including cases of low $m,\kcir$) 
with surfaces discovered or rediscovered in \cite{KKMS}. 
Finally, 
\   $\xL_{1   ,1}$ \   
in Lawson's notation 
can be identified with the Clifford torus \cite{Lawson} 
(with the Clifford torus placed at distance $\pi/4$ from $\CC$ or $\CCperp$), \     
and in our notation with $\xibreve_{(2)}$ or $\xibreve_{(1,1)}$.  
Similarly \ $\xbreveutwo$ \ can be identified with the critical catenoid with its axis placed on $\D$. 
\end{remark}

\bigbreak 
\paragraph{\bf{Brief discussion of results}}
$\phantom{ab}$ 
$\vspace{.072cm}$

In this article we concentrate on the equator-poles doublings $\xibreveomo$ of genus $\gamma=m+1$ for which we have the following.     

\begin{theorem}[Main Theorem \ref{Tmain}] 
\label{TA} 
For \  $m\in\N$ \  large enough, the index of \  $\xibreveomo$ \  is \quad  $2\gamma+5=2m+7$ \quad  and its nullity is \ $6$, \  
and so it has no exceptional Jacobi fields and is $C^1$ isolated. 
\end{theorem}

Moreover in ongoing work \cite{IIDindex} we prove that the FBM doublings $\xbreveum$ satisfy 
\quad $\ind(\xbreveum)=2m$ \quad 
for any large enough $m\in\N$ (see also \cite{carlotto-schulz-wiygul-2308}), 
and we also expect to generalize Theorem \ref{TA} by proving that 
\begin{equation} 
\label{E:i} 
\ind(\xbreve)= 2 \, \genus(\xbreve) \, + \, 2 \, \kcir \, + \, \kpol \, + \, 1   
\end{equation} 
holds for $\xbreve$ any of 
\ $\xbreve_{(m:\kcir)}$, \ $\xbreve_{(1,m:\kcir,1)}$, \ and \ $\xbreve_{(m:\kcir,1)}$,    \  with any  $\kcir\in\N$ and $m$ large enough in terms of $\kcir$.  
(Note that \eqref{E:i} holds as well for $\kcir=1$, $\kpol=0$ by \ref{rem} and \cite{LindexI}.) 

\eqref{E:i} and ongoing work on torus doublings motivate us to formulate the following conjecture which would generalizate Urbano's result \cite{urbano} to any genus. 
Note that the lower bound proposed by the conjecture is much stronger than the best known lower bound \cite{savo}.

\begin{conjecture} 
\label{conj1} 
The index of a genus $\gamma\in\N$ closed embedded minimal surface in $\Sph^3$ is at least $2\gamma+3$, 
with equality if and only if the surface is congruent to the Lawson surface $\xL_{\gamma,1}$. 
Moreover the number of closed embedded minimal surfaces in the round three-sphere $\Sph^3$ of genus $\gamma\in\N$ 
and index $\le 2\gamma+3+ \kexc$ is bounded above independently of $\gamma$ for any given $\kexc\in\N$. 
\end{conjecture} 

It is natural then to ask the following related question whose answer is known only for $\gamma=0,1$ by 
the uniqueness of the sphere \cite{almgren} and the Clifford torus \cite{brendle}. 

\begin{question} 
\label{ques1} 
Given (small) $\kexc\in\N$ and $\gamma\in\N$ determine all 
closed embedded minimal surfaces in $\Sph^3$ of genus $\gamma$ and index $\le 2\gamma+3+ \kexc$.  
\end{question} 

For example it may be that the only closed embedded minimal surface in $\Sph^3$ of large genus $\gamma$ and index $2\gamma+4$ is the 
\emph{single-circle-pole doubling $\xibreve_{\gamma,1}$}. 
In analogy with \ref{conj1} and \ref{ques1}, it is also natural to ask the following question,  
where $F_{\gamma,m}$ denotes the class of compact embedded FBM surfaces in $\B^3$ 
of genus $\gamma$ and $m$ boundary connected components. 

\begin{question} 
\label{ques2} 
For each $\gamma\in\N$ and $m\in\N$ determine the surfaces of smallest index in $F_{\gamma,m}$. 
\end{question} 

$\xbreveum$ seems a possible candidate for index minimizer in $F_{0,m}$.  
Note that this is known for $m=2$ (which is the critical catenoid case by \ref{rem}) by \cite{tran-CCindex}.

\subsection*{Outline of strategy and main ideas}
$\vspace{.126cm}$ 
\nopagebreak

The determination of the index and nullity in \cite{LindexI} is based on the observation that the symmetries can be used to decompose the function spaces (see \ref{EIVpm}) 
so that each summand in the decomposition contains exactly one Jacobi field (up to scaling).  
This is supplemented by a careful analysis of the graphical properties of the Lawson surfaces using the maximum principle.     

The results in this article are also based on the same decomposition and observation. 
The smaller symmetry group however makes it much harder to close the argument. 
Our approach here is based on                  a careful study    of the eigenfunctions using PDE gluing methods 
instead of the graphical properties of the surfaces. 
This, unlike in \cite{LindexI}, restricts our results to the high genus (high $m$) case. 
We outline now the main features of the gluing approach. 

We use a conformal modification of the metric and the Jacobi operator (see \ref{ELh}) to simplify the geometry and the spectrum.  
This is based on an idea of R. Schoen which played a fundamental role in \cite{kapouleas:1990} and has been heavily used since (for example in 
\cite{kapouleas:1995,kapouleas:yang,CCD,kapouleas:general}, and which is as follows. 
We introduce a metric $h$ which is conformal to the induced metric $g$. 
The catenoidal bridges with the $h$ metric are approximately isometric to round spheres attached via small necks to the graphical region on which $h$ approximates $g$, 
and the Jacobi operator is conformally modified to the operator $ \Lcalh= \Delta_h + 2$ whose low eigenvalues cluster close to $-2$ and $0$. 
The  corresponding eigenfunctions are induced  from \emph{constants} or \emph{slidings} on the bridges, or from eigenfunctions on the base surface $\Sigma$.   
The latter ones restrict to \emph{scalings} on the bridges.        
The eigenfunctions whose eigenvalues are close to $0$ span what is called \emph{approximate kernel} which plays an important role in PDE gluing constructions 
(see for example \cite{schoen1988,kapouleas:1990,kapouleas:1995,kapouleas:survey,ALM20,kapouleas:general}).  

It remains to determine the sign of the eigenvalues close to \  $0$.\  
The ones induced from constants are easy to handle because their $\Lcalh$ eigenvalue is close to \    $-2$. \     
The ones induced from slidings are much harder because they can have small positive or small negative eigenvalue. 
Our approach is based on a refinement of the decomposition into spaces $V^{\pm\pm}_\pm$ used in \cite{LindexI}, 
by using spaces of $\C$-valued functions  \   $W^{\pm\pm}_{\mu\pm}$, \    
where  \   $\mu\in [0,m/2] \cap \Z$  \   controls the \emph{frequency} with respect to rotational symmetries along the equatorial circle $\CC$. 
We also use the Jacobi fields (of eigenvalue $0$) and the normal components (of unmodified Jacobi operator eigenvalue $-2$) and their nodal domains for comparison.  
Note that both have frequency $\mu=0$ or $\mu=1$ and the index    in these frequencies behaves quite differently than in higher frequencies. 
In general in this article we use repeatedly detailed estimates for the eigenfunctions  
which rely     on  arguments similar to the ones        in \cite{kapouleas:1995}    for the approximate kernel using the $\chi$ metric.  

It turns out that the following slidings give negative eigenvalues: slidings along $\CC$ of nonzero frequency, 
slidings on equatorial bridges  towards the poles of zero frequency, and slidings along the polar bridges,  
with the exception of two combinations of the above which provide two of the Jacobi fiels. 
The remaining slidings give positive eigenvalues with the exception of a combination of slidings along $\CC$ providing $J^\CC$, 
three combinations of slidings along the axes of the bridges which give the three remaining Jacobi fiels, 
and a combination of slidings along the axes of the bridges which gives a negative eigenvalue (actually a component of the normal). 

Determining the sign of the negative eigenvalues is done in a few cases by Courant's nodal theorem applied to Jacobi fields or normal components;  
and in the remaining cases by constructing test functions where the following principle applies: the Rayleigh quotient of slidings along $\CC$ can be reduced by an increase of their frequency. 
For the slidings towards the poles of positive eigenvalue the opposite applies: the Rayleigh quotient can be reduced by a decrease of frequency; 
we can then compare with the Jacobi fields which have zero Rayleigh quotient.      

Finally all the scalings have positive eigenvalue and they are the hardest to control, especially the ones of zero frequency. 
The corresponding eigenfunctions can be viewed as approximate Jacobi fields which are the variation fields 
of a family of approximately minimal doublings with parameter the size of the catenoidal bridges. 
Our approach in this case is to estimate   the eigenvalue by simultaneously constructing the eigenfunction approximately but carefully enough. 
The construction of the eigenfunction is motivated by the gluing construction of the minimal doubling and uses again gluing methodology ideas.

\subsection*{Organization of the presentation}
$\phantom{ab}$
$\vspace{.072cm}$ 
\nopagebreak

This article is divided into seven sections. 
In Section \ref{S:intro} we recall  the minimal surface doublings of interest, discuss the results in this article and in \cite{IIDindex}, and outline the approach and main ideas in proving them. 
In Section \ref{S:general} we recall basic general facts and fix some general notation. 
In Section \ref{S:basic} we recall basic geometric facts about the Euclidean catenoid and about $\Sph^3$ and its great spheres and great circles.  
Most of the nonstandard notation used in this article is defined in Sections  \ref{S:intro}, \ref{S:general}, and \ref{S:basic}.  
It is based on the notation developed in \cite{LindexI,KW:Luniqueness}  and \cite{SdI,SdII,gLD},   and 
we hope it is well-suited to facilitate the presentation. 

In Section \ref{S:Ddou} we discuss minimal surface doublings both general and of gluing type and prove general results about the index and the approximate kernel which follow from our 
gluing approach. 
In Section \ref{S:symS} we discuss the implications of the symmetries of the $\Sph^2$ doublings to the eigenfunctions and their decompositions. 
In Section \ref{S:xi1m1} we discuss the geometric details of the equator doublings we need. 

In Section \ref{S:eq-p} we study the eigenvalues of interest and corresponding eigenfunctions in detail and prove the main results of this article. 
The crucial             proofs are presented in this section.

\bigbreak  
\section{General notation and conventions}                  
\label{S:general}

\subsection*{Basics}
$\vspace{.072cm}$ 
\nopagebreak

Most notation and definitions adopted below are either standard or as in 
\cite{LindexI,KW:Luniqueness} with some modifications. 

\begin{definition}
\label{Dsimc}
We write $a\Sim^c b$ to mean that 
$a,b\in\R\setminus\{0\}$, $c\in(1,\infty)$, and $\frac1c\le \frac ab \le c$. 
\end{definition}

\begin{notation} 
\label{DIA} 
We denote the identity map on any set $A$ by $I_A$. 
\qed
\end{notation} 

\begin{notation} 
\label{Nsub} 
We denote a subset of a set by providing its defining property as a subscript, or part of a subscript after a comma,     
see for example \ref{N:xx} and \ref{n:Kp}. 
\qed
\end{notation} 

\begin{notation}
\label{tubular}
For $(N, g)$ a Riemannian manifold and $A\subset N$ we write 
$\dbold^{N, g}_A$ for the distance function from $A$ with respect to $g$ 
and we define 
by $  D^{N, g}_A(\delta):=\left \{p\in N:\dbold^{N, g}_A(p)<\delta\right\} $  
the \emph{tubular neighborhood of $A$ of radius $\delta>0$}. 
If $A$ is finite we may just enumerate its points in both cases, for example if $A=\{q\}$ we write $\dbold_q(p)$. 
If $S\subset N$ is a two-dimensional submanifold we denote by $\area_h(S)$ its area induced by some Riemannian metric $h$ defined on $S$. 
\end{notation}

\begin{notation} 
\label{sph} 
We denote by $g$ the Euclidean metric on Euclidean spaces and by $g_S$ or simply $g$ the induced metric on a submanifold $S$. 
We denote by $\Sph^3 \subset \R^{4}$ the unit $3$-dimensional sphere and by $\B^3\subset \R^{3} $ the unit three-ball. 
\qed 
\end{notation}

\begin{notation}[Surface identifications] 
\label{Nid}
To simplify the notation we may \emph{identify surfaces $S_1$ and $S_2$ \emph{via} a diffeomorphism $Y:S_1 \to S_2$}, 
in the sense that we do not change notation for the pull-back or push-forward by $Y$ of  
a function (or geometric tensor field or operator on such) defined on a subset of $S_1$ or $S_2$; 
this does not apply however to the subsets themselves.  
For example if $u$ is a function on $S_2$ we denote $u\circ Y$ by $u$ as well, but we do not denote $S_1$ by $S_2$ unless they were defined as identical already. 
\end{notation}

We repeat now and in section \ref{S:basic} almost verbatim some definitions from \cite{LindexI,KW:Luniqueness,gLD} the notation of which we follow in general in this article. 

\begin{notation} 
\label{span} 
For any 
$A\subset \Sph^3 \subset \R^{4}$ 
we denote by $\Span(A)$ the span of $A$ as a subspace of $\R^{4}$  
and we set $\Sph(A):=\Span(A)\cap\Sph^3$. 
\qed 
\end{notation} 

\begin{notation}[Euclidean reflections and homotheties] 
\label{Nhom}
We denote by $V^\perp$ the orthogonal complement in $\R^{n  }$   
of a vector subspace $V \subset \R^{n}$, 
and we define the \emph{reflection in $\R^{n}$ with respect to $V$},   
\quad $\refl_V: \R^{n} \to \R^{n} $, \quad by 
\quad $\refl_V:= \Pi_V - \Pi_{V^\perp}$,      \quad  
where $\Pi_V$ and $\Pi_{V^\perp}$ are the orthogonal projections of $\R^{n}$ onto $V$ and $V^\perp$ respectively. 
\qquad 
Given also $\aunder\in \R_+$ we define the \emph{homothety} \   $\homot_{\aunder}:\R^n\to\R^n$ \   by \   $\homot_{\aunder}(x) \, := \, \aunder x \qquad (\forall x\in\R^n)$. 
\end{notation}

Clearly the fixed point set of $\refl_V$ is $V$ 
and it is the linear map which restricts to the identity on $V$ and minus the identity on $V^\perp$.  

\begin{definition}[$A^\perp$ and reflections $\refl_A$] 
\label{D:refl} 
Given any 
$A\subset  \Sph^3 \subset \R^{4  }$,  
we define $A^\perp:=\left(\,Span(A) \, \right)^\perp \cap \Sph^3 $  
and  
$\refl_A : \Sph^3 \to \Sph^3 $ to be the restriction to $\Sph^3$ of $\refl_{\Span(A)}$. 
Occasionally we will use simplified notation:  
for example for $A$ as before and $p\in\Sph^3$ we may write $\Sph(A,p)$ and $\refl_{A,p}$ instead of $\Sph(A\cup\{p\})$ and $\refl_{A\cup\{p\}}$ 
respectively. 
\qed 
\end{definition} 

Note that the set of fixed points of $\refl_A$ above is $\Sph(A)$ as in notation \ref{span},  
which is $\Sph^3$ 
or a great two-sphere  
or a great circle  
or the set of two antipodal points  
or the empty set, 
depending on the dimension of $\Span(A)$. 

$\vspace{.081cm}$ 

Following the notation in \cite{choe:hoppe} we have the following definition. 

\begin{definition}[The cone construction] 
\label{D:cone} 
For $p,q\in\Sph^3$ which are not antipodal we denote 
the minimizing geodesic segment joining them by $\overline{pq}$. 
For $A,B\subset\Sph^3$ such that no point of $A$ is antipodal to a point of $B$ 
we define the cone of $A$ and $B$ in $\Sph^3$ by 
$$ 
A\cone B := \bigcup_{p\in A, \, q\in B} \overline{pq}. 
$$ 
If $A$ or $B$ contains only one point we write the point instead of $A$ or $B$ respectively; 
we have then $p\cone q = \overline{pq} $ 
for any 
$p,q\in\Sph^3$ which are not antipodal. 
More generally, 
given linearly independent $p_1,\cdots,p_k\in \Sph^3$, 
we define inductively for $k\ge3$ 
$
\overline{p_1\cdots p_k} := p_k \cone \overline{p_1\cdots p_{k-1} }.  
$
\qed 
\end{definition} 

\begin{definition}[Tetrahedra] 
\label{TEV} 
If $p_1,p_2,p_3,p_4\in \Sph^3$ are linearly independent, 
then 
$T:= \overline{p_1p_2p_3p_4}$ is called a (spherical) tetrahedron 
with vertices $p_1,p_2,p_3,p_4$, 
edges $\overline{p_ip_j}$ ($i,j=1,2,3,4$, $i\ne j$),  
and faces 
$\overline{p_2p_3p_4}$, $\overline{p_1p_3p_4}$, $\overline{p_1p_2p_4}$, and $\overline{p_1p_2p_3}$.
We call two edges of $T$ \emph{adjacent} if they share a vertex and  
similarly a face and an edge if the face contains the edge.
Two edges which are not adjacent are called \emph{opposite}. 
Finally we use the notation $E_T:= \bigcup_{i,j=1}^4\overline{p_ip_j}$, the union of the edges, 
and $V_T:=\{p_1,p_2,p_3,p_4\}$, the set of the vertices.   
\end{definition}

$\vspace{.081cm}$ 

If $\Grp$ is a group acting on a set $B$ and if $A$ is a subset of $B$,
then we refer to the subgroup 
  \begin{equation}
  \label{stab}
    \stab_{\Grp}(A):=\{ \mathbf{g} \in \Grp \; | \; \mathbf{g}A = A \}
  \end{equation}
as the \emph{stabilizer} of $A$ in $\Grp$.
For $A,B$ subsets of $\Sph^3$ or $\Abold$ a finite collection of such subsets we define   
  \begin{equation}
  \label{Gsym}
\begin{gathered} 
    \Gsym^A:=\stab_{{O(4)}} A 
= \{ \mathbf{g} \in O(4) \; | \; \mathbf{g}A = A \}, 
\\    
\Gsym^{A,B}:= \Gsym^{A} \cap \Gsym^{B}, 
\\ 
    \Gsym^{\Abold} := \{ \mathbf{g} \in O(4) \; | \; \{ \mathbf{g}A : A\in\Abold\} = \Abold \}. 
\end{gathered} 
  \end{equation}

%%%%%%%%%%%%%%%%%%%%CUTOFF 

$\vspace{.126cm}$ 

Our arguments require extensive use of cut-off functions and the following will 
be helpful. 
\begin{definition}
\label{DPsi} 
We fix a smooth function $\Psi:\R\to[0,1]$ with the following properties:
\begin{enumerate}[label=\emph{(\roman*)}]
\item $\Psi$ is nondecreasing.

\item $\Psi\equiv1$ on $[1,\infty)$ and $\Psi\equiv0$ on $(-\infty,-1]$.

\item $\Psi-\frac12$ is an odd function.
\end{enumerate}
\end{definition}

Given $a,b\in \R$ with $a\ne b$,
we define smooth functions
$\psicut[a,b]:\R\to[0,1]$
by
\begin{equation}
\label{Epsiab}
\psicut[a,b]:=\Psi\circ L_{a,b},
\end{equation}
where $L_{a,b}:\R\to\R$ is the linear function defined by the requirements $L_{a,b}(a)=-3$ and $L_{a,b}(b)=3$.

Clearly then $\psicut[a,b]$ has the following properties:
\begin{enumerate}[label={(\roman*)}]
\item $\psicut[a,b]$ is weakly monotone.

\item 
$\psicut[a,b]=1$ on a neighborhood of $b$ and 
$\psicut[a,b]=0$ on a neighborhood of $a$.

\item $\psicut[a,b]+\psicut[b,a]=1$ on $\R$.
\end{enumerate}

Suppose now we have a real-valued function $d$ defined on a domain $\Omega$ 
and two sections $f_{(a)},f_{(b)}$ of some vector bundle $\Omega$.
(A special case is when the vector bundle is trivial and $f_{(a)},f_{(b)}$ real-valued functions).
We define a new section which transits from $f_{(a)}$ on a neighborhood of $d^{-1}(a)$ to $f_{(b)}$ on a neighborhood of $d^{-1}(b)$ by  
\begin{multline}
\label{EPsibold}
\Psibold\left [a,b;d \, \right]( \, f_{(a)} , f_{(b)} \, ) 
:=
\psicut[a,b\, ]\circ d \, f_{(a)}
+
\psicut[b,a]\circ  d \, f_{(b)} 
\\ 
\, = \, 
f_{(a)} 
+
\psicut[b,a]\circ  d \  (\, f_{(b)} - f_{(a)} \, ).  
\end{multline}
Note that
\quad  $\Psibold\left [a,b;d \, \right]( \, f_{(a)} , f_{(b)} \, )$ \quad  
is then a section which depends linearly on the pair $(f_{(a)},f_{(b)})$
and is smooth if $f_{(a)},f_{(b)}$ and $d$ are smooth.

%%%%%%%%%%%%%%%%%%%%END CUTOFF 

$\vspace{.045cm}$ 
\subsection*{Differential and Jacobi operators}  
$\vspace{.009cm}$ 
\nopagebreak

\begin{notation}
\label{ELjacobi} 
Given a Riemannian manifold $(N, g)$ and 
a surface $S\subset N$, we denote the Jacobi operator on $S$ by 
\begin{equation*} 
\Lcal \, = \, \Lcal_S \, = \, \Delta_S + |A|^2+ \Ric(\nu,\nu),  
\end{equation*} 
where \ $\Delta_S$ denotes the Laplace-Beltrami operator on $S$, 
\ $|A|^2$ the length squared of the second fundamental form, 
\ $\Ric$ the Ricci curvature of $(N,g)$, 
\ and \ $\nu$ the unit normal of $S$. 
Moreover if \ $\chi = \varrho^2 g_S$ \ is a metric conformal to the induced metric $g_S$ on $S$, we define the operator 
$$
\Lcal_\chi \, =  \, \Lcal_{S,\chi} \, :=  \, \varrho^{-2} \Lcal_S = \Delta_\chi + \varrho^{-2} ( \, |A|^2+ \Ric(\nu,\nu)\, ),  
$$
where $\Delta_\chi$ denotes the Laplace-Beltrami operator on $(S,\chi)$. 
\end{notation}

Extending an idea of R. Schoen which goes back to \cite{kapouleas:1990} and has been heavily used since (for example in 
\cite{kapouleas:1995,kapouleas:yang,CCD,kapouleas:general}, 
we define conformal changes of the metric and the Jacobi operator $\Lcal_S$ on a surface $S$ as in \ref{ELjacobi} 
satisfying \ $|A|^2+ \Ric(\nu,\nu) > 0$, \   by  
\begin{equation} 
\label{ELh} 
\begin{gathered}
h_S:= \frac{|A|^2+\Ric(\nu,\nu) }2 g_S, 
\\            
\Lcal_{h_S} \, =  \, \Lcal_{S,h_S} \, :=  \, \frac2{|A|^2+\Ric(\nu,\nu) } \Lcal_S = \Delta_{h_S} + 2. 
\end{gathered}
\end{equation}

In the case of a compact FBM surface $M$ in $\B^3$ we take the Jacobi operator to be 
\begin{equation} 
\label{ELrobin} 
\Lcal:=(\Lint,\BbouR),  
\qquad \text{where} \qquad 
\Lint := \Delta+|A|^2,  
\qquad 
\BbouR := {\eta}- 1 , 
\end{equation} 
where $\eta$ is the outward pointing unit conormal of $M$ on $\partial M    $, 
$\Lint$ acts on $C^2$ functions on $M$ and is the usual Jacobi operator, 
and $\BbouR$ supplements the Jacobi equation on $M$ with a boundary Robin condition on $\partial M$ as in 
\cite[Section 5]{chen-fraser-pang:2015:index:TAMS}.

$\vspace{.072cm}$ 

In the  remaining of this section we let $\F$ stand for either $\R$ or $\C$. 
The following definition is standard and we state it to fix the notation. 

\begin{definition}[Operators and eigenfunctions]  
\label{D:oper} 
\ \ \ 
Let $\Ut$ be a compact smooth surface, 
$\Lintt$ a linear differential operator acting on (smooth enough) $\F$-valued functions on $\Ut$, 
and $\Bbout$ a boundary differential operator acting at $\partial \Ut$ (if nonempty). 
Let $\Lcalt = \Lintt$ if $\partial \Ut = \emptyset$ and 
$\Lcalt:=(\Lintt,\Bbout)$ if 
$\partial \Ut\ne\emptyset$. 
We call $f$ an \emph{eigenfunction} of $\Lcalt$ with \emph{eigenvalue} $\lambda$ if $f\not\equiv0$ is a function on $\Ut$ satisfying 
$ \Lintt f + \lambda f =0 $ on $\Ut$ and (when $\partial \Ut\ne\emptyset$) $\Bbout f=0$ on $\partial\Ut$.  
\end{definition} 

\begin{definition}[Operator symmetries]  
\label{D:symL} 
We define a \emph{symmetry $\rot_*$ of $\F$-valued functions on $\Ut$} to be 
an $\R$-linear map on $\F$-valued functions on $\Ut$ which is the composition of the pushforward by a smooth diffeomorpism $\rot$ of $\Ut$ and a $\R$-linear map on $\F$. 
If moreover $\rot_*$ 
and an operator $\Lcalt$ as in \ref{D:oper} commute, we call $\rot_*$ 
a \emph{symmetry of $\Lcalt$}. 
\end{definition}

\begin{notation}[Eigenvalues, index and nullity] 
\label{D:eigenb} 
We assume given a compact smooth surface $\Ut$ 
and an operator $\Lcalt$ as in \ref{D:oper}. 
We assume moreover that $\Ut$ is equipped with a Riemannian metric $\gt$,   
and we are given a smooth real-valued function $\widetilde{A}$ on $\Ut$.  
Let $\Lintt:=\Delta_{\gu}+ \widetilde{A}$ be the Schr\"{o}dinger operator on $\Ut$ acting on $\F$-valued $C^2$ functions defined on $\Ut$.  
It is well-known that 
the eigenfunctions of $\Lcalt$ defined as in \ref{D:oper} are smooth with real eigenvalues which accumulate only at $+\infty$, 
if either $\Ut$ is closed, or $\Bbout$ provides (mixed) Dirichlet, Neumann, or smooth Robin conditions at $\partial\Ut$. 

Let \quad $\Vt:= \big\{ f:\Ut\to\F \, \big| \, \rot_* f = f \quad \forall \rot_*\in \grouptilde\big\}$, \quad     
where $\grouptilde       $ is a given group of symmetries of $\Lcalt$ (recall \ref{D:symL}), 
and \     $J\subset \R$. \   
We have then the following. 

\begin{enumerate}[label=\emph{(\roman*)}]
\item 
\label{D:eigenb1} 
Counting the eigenvalues of the eigenfunctions contained in $\Vt$ in non-decreasing order and with multiplicity, 
we denote the $i^{th}$ such eigenvalue by $\lambda_i(\Vt,\Lcalt,\Ut)$. 
\item 
\label{D:eigenb2} 
We denote by 
$E_J(\Vt,\Lcalt,\Ut ) \subset \Vt$   
the real (or complex) span of the eigenfunctions contained in $\Vt$ 
with eigenvalue contained in $J$.  
\item 
\label{D:eigenb3} 
For $J\subset\R$ we define 
\hfill 
$\#_J(\Vt,\Lcalt,\Ut ) := \dim_\R \,  E_J(\Vt,\Lcalt,\Ut )$      
\hfill 
and  (if applicable) 
\quad 
$\#_J^\C(\Vt,\Lcalt,\Ut ) := \dim_\C \,  E_J(\Vt,\Lcalt,\Ut )$.                            
\item 
\label{D:eigenb4} 
We define the \emph{index of $\Lcalt$ on $\Ut$} by 
\ 
$\ind(\Lcalt , \Ut ):= \#_{<0}(V ,\Lcalt,\Ut)$  
\ 
and the \emph{nullity of $\Lcalt$ on $\Ut$} by 
\ 
$\nul(\Lcalt,\Ut) := \#_{=0}( V ,\Lcalt,\Ut)$,  
where $V$ is the space of $\R$-valued functions on $\Ut$. 
\end{enumerate} 
We sometimes simplify the above notation as follows. 
	\begin{enumerate}[label=\emph{(\alph*)}]
\item 
\label{D:eigenb8} 
We may omit $\Vt$ when $\grouptilde$ is trivial. 
\item 
\label{D:eigenb5} 
We may omit $\Lcalt$ or $\Ut$ when they can be inferred from the context.  
\item 
\label{D:eigenb6} 
In particular for $\Ut$ a closed minimal surface we take $\Lcalt$ to be $\Lcal=\Lcal_{\Ut}$ as in \ref{ELjacobi} 
and for $\Ut$ a compact FBM surface $M$ in $\B^3$ we take $\Lcalt$ to be $\Lcal$ as in \ref{ELrobin},  
so that $\ind(\Ut)$ denotes the so called \emph{Morse index of $\Ut$} 
\cite[Section 5]{chen-fraser-pang:2015:index:TAMS} 
and $\nul(\Ut)$ denotes the \emph{nullity     of $\Ut$}. 
\item 
\label{D:eigenb7} 
We may write   ``$<\lambda$'', or ``$=\lambda$'', or ``$\le\lambda$'' 
instead of $J=(-\infty,\lambda)$, or $J=\{\lambda\}$, or $J=(-\infty,\lambda]$ in the subscripts respectively. 
\end{enumerate} 
\end{notation} 

\begin{notation}[Eigenvalues and restrictions] 
\label{D:eig-res}   
\ \ 
Under the same assumptions as in \ref{D:eigenb} and 
for smooth compact $\Ut_1\subset\Ut$,  
we extend the notation introduced in \ref{D:eigenb} as follows. 
We define 
$$ 
\Vt[ \Utp  ]    :=     \left. \Vt           \right|_{ \Utp  } := \Big\{ \left. f \right|_\Utp  \, : \, f \in \Vt           \, \Big\} 
$$ 
and allow the use of \ $(\Vt,\LcaltD,\Utp )$ \ or \ $(\Vt,\LcaltN,\Utp )$\  
in place of \ $(\Vt,\Lcalt,\Ut)$,\  
as a shorthand for 
\ $( \left. \Vt           \right|_{ \Utp  } , \LcaltD, \Utp  )$ or $( \left. \Vt           \right|_{ \Utp  } , \LcaltN, \Utp  )$, \  
where 
\ $\LcaltD:=(\Linttp,\BboutD)$ \  
and 
\ $\LcaltN:=(\Linttp,\BboutN)$ \  
with 
\ $\Linttp:=\Delta_{\gu}+ \widetilde{A}$ \ like \ $\Lintt$ \ but acting on $\F$-valued $C^2$ functions defined on $\Utp $,   
and $\BboutD$ or $\BboutN$ 
providing Dirichlet or Neumann conditions at \ $\partial\Utp \setminus\partial\Ut$ \ respectively,   
while acting the same as \ $\Bbout$ \ at \ $\partial\Utp \cap     \partial\Ut$.   
As a special case if \ $\Ut$ \ is closed, then 
\ $\LcaltD:=(\Linttp,\BboutD)$ \  
and 
\ $\LcaltN:=(\Linttp,\BboutN)$ \  
with 
$\BboutD$ or $\BboutN$ 
providing Dirichlet or Neumann conditions at \ $\partial\Utp$ \ respectively.   
\end{notation}

In the next definition 
$f$ can be thought of as a ``weight'' function because $f(x)$ controls the size of $u$ in the vicinity of
the point $x$.
$\rho$ can be thought of as a function which determines the ``natural scale'' $\rho(x)$
at the vicinity of each point $x$.

\begin{definition}[Weighted H\"{o}lder norms]  
\label{D:norm-g} 
Assuming that $\Omega$ is a domain inside a manifold,
$g$ is a Riemannian metric on the manifold, 
$\rho,f:\Omega\to(0,\infty)$ are given functions, 
$k\in \N_0$, 
$\beta\in[0,1)$, 
$u\in C^{k,\beta}_{loc}(\Omega)$ 
or more generally $u$ is a $C^{k,\beta}_{loc}$ tensor field 
(section of a vector bundle) on $\Omega$, 
and that the injectivity radius in the manifold around each point $x$ in the metric $\rho^{-2}(x)\,g$
is at least $1/10$,
we define
$$
\|u: C^{k,\beta} ( \Omega,\rho,g,f)\|:=
\sup_{x\in\Omega}\frac{\,\|u:C^{k,\beta}(\Omega\cap B_x, \rho^{-2}(x)\,g)\|\,}{f(x) },
$$
where $B_x$ is a geodesic ball centered at $x$ and of radius $1/100$ in the metric $\rho^{-2}(x)\,g$.
For simplicity we may omit any of $\beta$, $\rho$, or $f$, 
when $\beta=0$, $\rho\equiv1$, or $f\equiv1$, respectively.
\end{definition}

\bigbreak  
\section{Basic geometry}
\label{S:basic} 

\subsection*{Rotations along or about great circles} 
$\vspace{.126cm}$ 
\nopagebreak

Note that given a great circle $C$ in $\Sph^3$,  
the points of $C^\perp$ (recall \ref{D:refl}) are at distance $\pi/2$ in $\Sph^3$ from $C$ and any point of $\Sph^3\setminus C^\perp$ is at distance $<\pi/2$ from $C$, 
therefore $C^\perp$ is the great circle furthest away from $C$. 
Equivalently $C^\perp$ is the set
of poles of great hemispheres with equator $C$;
therefore $C$ and $C^\perp$ are linked. 
The group 
$\Gsym^{C \cup C^\perp } $
contains 
$\Gsym^{C } = \Gsym^{C^\perp } $ 
(which includes arbitrary rotation or reflection in the two circles)  
and includes also   
orthogonal transformations exchanging 
$C$ with $C^\perp$.  

\begin{definition}[Rotations $\rot_C^\phi$, $\rot^C_\phi$ and Killing fields $\vecK_{C}$, $\vecK^{C}$] 
\label{D:rot} 
\ \ 
Given a great circle $C\subset\Sph^3$, $\phi \in \R$,
and an orientation chosen on the totally orthogonal circle $C^\perp$,
we define the following: 
\begin{enumerate}[label=\emph{(\roman*)}]
\item 
\label{D:rot1}  
a rotation $\rot_C^\phi \in SO(4)$ about $C$ by angle $\phi$ 
preserving $C$ pointwise and rotating the totally orthogonal circle $C^\perp$ along itself by angle $\phi$ 
(in accordance with its chosen orientation);  
\item 
\label{D:rot2}  
a   Killing field $\vecK_{C}$ on $\Sph^3$ 
given by \ $\left.\phantom{}\vecK_{C}\right|_p := \left. \frac{\partial}{\partial\phi} \right|_{\phi=0} \rot_{C}^\phi(p)$
\ $\forall p\in\Sph^3$. 
\end{enumerate} 

Assuming further an orientation chosen on $C$ (which given the orientation of $\Cperp$ is equivalent to choosing an orientation of $\Sph^3$) 
we define the following: 
\begin{enumerate}[label=\emph{(\alph*)}]
\item 
\label{D:rot3}  
the rotation along $C$ by angle $\phi$ is $\rot^C_\phi := \rot_{C^\perp}^\phi$;  
\item 
\label{D:rot4}  
the Killing field $\vecK^{C}:=\vecK_{C^\perp}$ on $\Sph^3$.  
\hfill $\square$ 
\end{enumerate} 
\end{definition}

Note that $\rot^C_\phi = \rot_{C^\perp}^\phi$ in the vicinity of $C$ resembles a translation along $C$,  
while in the vicinity of $C^\perp$ it resembles a rotation about $\Cperp$; 
moreover $\vecK_{C}$ is defined to be a rotational Killing field about $C$, 
vanishing on $C$ and equal to the unit tangent on ${C}^\perp$.  
These observations motivate the following definition 
where 
$\rot^\vecv $ and $\vecK^\vecv$ do not depend on any orientations but 
$\rot_\vecv $ and $\vecK_\vecv$ depend on the orientation of $\Sph^3$. 

\begin{definition}[Rotations $\rot^\vecv $, $\rot_\vecv $ and Killing fields $\vecK^\vecv$, $\vecK_\vecv$] 
\label{D:rot:v} 
\ \ 
Given $p\in\Sph^3$ and $\vecv\in T_p\Sph^3 \subset T\Sph^3$ we define 
$\rot^\vecv := \rot^C_\phi$,  
$\rot_\vecv := \rot_C^\phi$,  
$\vecK^\vecv := \phi \vecK^C$, 
and 
$\vecK_\vecv := \phi \vecK_C$, 
where 
$C$ is an oriented great circle through $p$ tangent to $\vecv$ 
and $\phi$ is such that $\vecK^\vecv(p)=\vecv$ 
(and so $|\phi|=|\vecv|$).        
We call $\rot^\vecv$ a \emph{translation at $p$} and $\rot_\vecv$ a \emph{rotation at $p$}.  
\qed 
\end{definition}

\begin{lemma}[Orbits]  
\label{orbits} 
For $\vecK^C$ as in \ref{D:rot}, the orbits of $\vecK^C$ (that is its flowlines) are planar circles (in $\R^4$)  
and $\forall\pp\in C $ each orbit intersects the closed hemisphere $C^\perp \cone \pp$ exactly once.  
Moreover the intersection (when nontrivial) is orthogonal. 
\end{lemma} 

\begin{proof} 
This is straightforward to check already in $\R^4$ with the 
hemisphere $C^\perp \cone \pp$ replaced by the half-three-plane containing $\pp$ and with boundary $\Span(C^\perp)$. 
By restricting then to $\Sph^3$ the result follows. 
\end{proof} 

This lemma allows us to define a projection which effectively identifies the space of orbits in discussion 
with a closed hemisphere: 

\begin{definition}[Projections by rotations]  
\label{Pi} 
For $C$ and $\pp$ as in \ref{orbits} 
we define the smooth map 
$\Pi^C_\pp :\Sph^3\to C^\perp \cone \pp$ by requiring $\Pi^C_\pp x$ to be the intersection of $C^\perp \cone \pp$ with the orbit of $\vecK^C$ containing $x$,    
for any $x \in\Sph^3$.  
\end{definition}

\begin{definition}[Graphical sets]  
\label{graphical} 
A set $A\subset \Sph^3$ is called \emph{graphical with respect to $\vecK^C$} (with $C$ as above) 
if each orbit of $\vecK^C$ intersects $A$ at most once. 
If moreover $A$ is a submanifold and there are no orbits of $\vecK^C$ which are tangent to $A$, 
then $A$ is called 
\emph{strongly graphical with respect to $\vecK^C$}.   
\end{definition}

\subsection*{The elementary geometry of totally orthogonal circles}
$\vspace{.126cm}$ 
\nopagebreak

We fix now a great circle $\CC$, orientations on $\CC$ and $\CCperp$ (which imply an orientation on $\Sph^3$ as well), 
and arbitrarily points $\pp_0 \in \CC$ and $\pp^0 \in \CCperp$.  
We define $\forall \phi\in\R$ the points 
\begin{equation}
\label{points} 
\pp_{\phi} := \rot_{\CCperp}^{\phi}\,\pp_0\,\in \, \CC,
\qquad
\pp^{\phi} := \rot_{\CC}^{\phi}\,\pp^0\,\in \, \CCperp.
\end{equation}
Using \ref{span} we further define $\forall \phi\in\R$ the great spheres 
\begin{equation}
\label{hemisph} 
\Sigma^\phi :=  \,\Sph( \CC , \pp^\phi ) , 
\qquad  
\Sigma_\phi := \, \Sph( \CCperp , \pp_\phi),   
\end{equation}
and $\forall\phi,\phi'\in\R$ 
the great circles 
\begin{equation}
\label{circles}  
\CC_\phi^{\phi'} := \Sph( \, \pp_\phi, \pp^{\phi'} \, ).  
\end{equation}

\begin{definition}[Coordinates on $\R^4$]  
\label{D:x1} 
Given $\CC$ as above and points as in \ref{points}, 
we define coordinates $(x^1,x^2,x^3,x^4)$ on $\R^4\supset\Sph^3$ by choosing 
$$
\pp_0 :=(1,0,0,0), \     \pp_{\pi/2} :=(0,1,0,0), \     \pp^0 :=(0,0,1,0), \     \pp^{\pi/2} :=(0,0,0,1). 
$$
\end{definition}

\begin{definition}[Halves and quarter $\Sph^3$] 
\label{D:Om} 
We define (recall \ref{D:x1})  
\begin{equation*}
\begin{gathered} 
\Sph^{3+}:= \{(x^1,x^2,x^3,x^4)\in\Sph^3\subset\R^4:x^3\ge0\} \subset \Sph^3 , 
\\ 
\Sph^{3\und+}:= \{(x^1,x^2,x^3,x^4)\in\Sph^3\subset\R^4:x^4\ge0\} \subset \Sph^3 , 
\\ 
\Sph^3_+:= \{(x^1,x^2,x^3,x^4)\in\Sph^3\subset\R^4:x^1\ge0\} \subset \Sph^3 , 
\\ 
\Sph^{3++}:= \Sph^{3+} \cap \Sph^{3\und+} \subset \Sph^3 , 
\end{gathered} 
\end{equation*} 
and moreover we may use such  exponents and subscripts on any \     $A\subset\Sph^3$   \       
to denote the intersection of  \   $A$  \   with the corresponding subsets of  \   $\Sph^3$,  \   
for example 
\ $   A^{\und+} := A \cap \Sph^{3\und+}$,    
\ $\Sph^{3++}_+ := \Sph^{3++} \cap \Sph^{3}_+$,   \  
and 
\ $   A^{3++}_+ := A \cap \Sph^{3++} \cap \Sph^{3}_+$.   \  
\    Note that 
\ $\partial  \Sph^{3+} = \Sigma^0$,
\  
\ $\partial  \Sph^{3\und+} = \Sigma^{\pi/2}$,
\  
\ $\partial  \Sph^{3}_+ = \Sigma_0$,
\  
and 
\   $\partial  \Sph^{3++} \subset \Sigma^0\cup\Sigma^{\pi/2}$. 
\end{definition} 

\begin{lemma}[Geometry of $\CC$ and $\CCperp$ {\cite[Lemma 2.19]{LindexI}}]  
\label{Lobs} 
\ \ 
The following hold  
$\forall\phi,\phi', \phi_1,\phi'_1,\phi_2,\phi'_2  \in\R$.  
\begin{enumerate}[label=\emph{(\roman*)}]
\item 
\label{Lobs1} 
$ \pp_{\phi+\pi} = - \pp_\phi $  and 
$ \pp^{\phi+\pi} = - \pp^\phi $.   
Similarly 
$ \Sigma_{\phi+\pi} = \Sigma_\phi $  and 
$ \Sigma^{\phi+\pi} = \Sigma^\phi $.   
\item 
\label{Lobs2} 
$ 
\cC_\phi^{\phi'} 
=       
\overline{ \pp_\phi  \pp^{\phi'} } \, \lllcup \,      
\overline{ \pp^{\phi'} \pp_{\phi+\pi} } \, \lllcup \,      
\overline{ \pp_{\phi+\pi} \pp^{\phi'+\pi} } \, \lllcup \,      
\overline{ \pp^{\phi'+\pi} \pp_\phi  } 
$ 
\ where \ 
$\cC_\phi^{\phi'} \llcap \CC = \{ \pp_\phi , \pp_{\phi+\pi} \} $ 
\ and   \ 
$\cC_\phi^{\phi'} \llcap \CCperp  = \{ \pp^{\phi'} , \pp^{\phi'+\pi} \} $ 
\ 
with orthogonal intersections. 
\item 
\label{Lobs3} 
$\CC \cone \pp^\phi$  and $\CCperp \cone \pp_\phi$  
are closed great hemispheres 
with boundary $\CC$ and $\CCperp$ and poles $\pp^\phi$ and $\pp_\phi$ respectively. 
\item 
\label{Lobs4} 
$\Sigma_\phi = ( \CCperp \cone \pp_\phi ) \llcup ( \CCperp \cone \pp_{\phi+\pi} )$  
and 
$\Sigma^\phi = ( \CC \cone \pp^\phi ) \llcup ( \CC \cone \pp^{\phi+\pi} )$.  
\item 
\label{Lobs5} 
$\Sigma^\phi \llcap \CCperp = \{ \pp^\phi , \pp^{\phi+\pi} \}$ 
and 
$\Sigma_\phi \llcap \CC = \{ \pp_\phi, \pp_{\phi+\pi}\}$ 
with orthogonal intersections. 
\item 
\label{Lobs6} 
$\cC_\phi^{\phi'} = \Sigma_\phi \llcap \Sigma^{\phi'}$  
with orthogonal intersection. 
\item 
\label{Lobs7} 
$\left( { \cC_\phi^{\phi'} } \right)^\perp = \cC_{\phi \pm \pi/2}^{\phi' \pm \pi/2}$. 
\item 
\label{Lobs8} 
$\Sigma^\phi \llcap \Sigma^{\phi'} = \CC$ unless $\phi=\phi' \pmod \pi$ 
in which case 
$\Sigma^\phi = \Sigma^{\phi'}$.   
Similarly 
$\Sigma_\phi \llcap \Sigma_{\phi'} = \CCperp$ unless $\phi=\phi' \pmod \pi$ 
in which case 
$\Sigma_\phi = \Sigma_{\phi'}$.   
In both cases the intersection angle is $\phi'-\phi \pmod \pi$. 
\item 
\label{Lobs9} 
$\cC_{\phi_1}^{\phi_1'} \llcap \cC_{\phi_2}^{\phi_2'} = \emptyset$ 
unless 
$\phi_1=\phi_2 \pmod \pi$ 
or
$\phi'_1=\phi'_2 \pmod \pi$. 
If both conditions hold then  
$\cC_{\phi_1}^{\phi_1'} = \cC_{\phi_2}^{\phi_2'}$.  
If only the first condition holds then 
$\cC_{\phi_1}^{\phi_1'} \llcap \cC_{\phi_2}^{\phi_2'} = 
\{ \pp_{\phi_1} , \pp_{\phi_1+\pi} \}$ 
with intersection angle equal to $\phi'_2-\phi'_1 \pmod \pi$. 
If only the second condition holds then 
$\cC_{\phi_1}^{\phi_1'} \llcap \cC_{\phi_2}^{\phi_2'} = 
\{ \pp^{\phi_2} , \pp^{\phi_2+\pi} \}$ 
with intersection angle equal to $\phi_2-\phi_1 \pmod \pi$. 
\end{enumerate} 
\end{lemma}

\begin{proof} 
It is straightforward to verify all these statements by using the coordinates defined in 
\ref{D:x1}.  
\end{proof}

\subsection*{Killing fields} 
$\vspace{.009cm}$ 
\nopagebreak

\begin{definition}[Symmetries of Killing fields]  
\label{DKsymm} 
We call a Killing field $\vecK$ \emph{even (odd) under an isometry $\refl$} if it satisfies 
${\refl}_* \circ \vecK = \vecK \circ \refl$ 
($\,\, {\refl}_* \circ \vecK = - \vecK \circ \refl \,\, $).  
\qed 
\end{definition} 

\begin{lemma}[Some symmetries of Killing fields]  
\label{Ksymm} 
\hfill 
The following hold 
$\forall \phi,\phi'\in\R$.  
\begin{enumerate}[label=\emph{(\roman*)}]
\item 
\label{Ksymm1} 
$\vecK_\CC$ is 
odd under $\refl_{\Sigma^{\phi}}$ 
and 
$\refl_{ \cC_{ \phi }^{ \phi' } }$ 
and even under $\refl_{\Sigma_{\phi}}$.  
\item 
\label{Ksymm2} 
$\vecK_{\CCperp}$ is 
odd under $\refl_{\Sigma_{\phi}}$ 
and 
$\refl_{ \cC_{ \phi }^{ \phi' } }$  
and even under $\refl_{\Sigma^{\phi}}$.  
\item 
\label{Ksymm3} 
$\vecK_{ \cC_{ \phi }^{ \phi' } }$ is odd under 
$\refl_{\Sigma_\phi}$ and $\ \refl_{\Sigma^{\phi'}}$ 
and even under 
$\refl_{\Sigma_{\phi+\pi/2}}$ and 
\\ 
$\refl_{\Sigma^{\phi'+\pi/2}}$.  
Moreover 
${\Sigma_{\phi+\pi/2}}$ and ${\Sigma^{\phi'+\pi/2}}$  
are preserved under the flow of 
$\vecK_{ \cC_{ \phi }^{ \phi' } }$ and contain the fixed points $\pm\pp^{\phi'}\in {\Sigma_{\phi+\pi/2}}$ and $\pm\pp_{\phi}\in {\Sigma^{\phi'+\pi/2}}$ 
and the geodesic orbit 
${ \cC_{ \phi +\pi/2 }^{ \phi' +\pi/2 } } = {\Sigma_{\phi+\pi/2}}\llcap {\Sigma^{\phi'+\pi/2}}$.   
\end{enumerate} 
\end{lemma} 

\begin{proof} 
For any great circle $C'$ we have that $\vecK_{C'}$ is even (odd) with respect to a reflection $\refl$ 
if and only if $\refl(C'^\perp) = C'^\perp$ and $\refl$ respects (reverses) the orientation of $C'^\perp$.  
Applying this it is straightforward to confirm the lemma. 
\end{proof} 

Using the coordinates defined in \ref{D:x1}
and assuming appropriate orientations as in \cite{LindexI} 
we have as in \cite[(3.17)]{LindexI} 
\begin{equation}
\label{E:K} 
    \begin{aligned}
      \vecK_{\CCperp}(x)=\vecK^{\CC}(x)&=x^1 \, \pp_{\pi/2} - x^2 \, \pp_0 ,  
\\ 
      \vecK_\CC(x)=\vecK^{\CCperp}(x)&=x^3 \, \pp^{\pi/2} - x^4 \, \pp^0 , 
    \end{aligned}
   \end{equation}
\begin{equation}
\label{E:Kcom} 
    \begin{aligned}
      \vecK_{\cC_{\pi/2}^{\pi/2}}(x)=\vecK^{\cC_0^0}(x)&=x^1 \, \pp^0 - x^3 \, \pp_0 , \\
      \vecK_{\cC_0^{\pi/2}}(x)=\vecK^{\cC_{\pi/2}^0}(x)&=x^2 \, \pp^0 - x^3 \, \pp_{\pi/2} , \\
      \vecK_{\cC_{\pi/2}^0}(x)=\vecK^{\cC_0^{\pi/2}}(x)&=x^1 \, \pp^{\pi/2} - x^4 \, \pp_0 , \\ 
      \vecK_{\cC_0^0}(x)=\vecK^{\cC_{\pi/2}^{\pi/2}}(x)&=x^4 \, \pp_{\pi/2} - x^2 \, \pp^{\pi/2} . 
    \end{aligned}
   \end{equation}

\subsection*{Basic geometry of $\Sigma^0, \T \subset \Sph^3$} 
$\vspace{.126cm}$ 
\nopagebreak

We adopt now some notation and recall some calculations from \cite{SdI,SdII}. 

\begin{definition}[Coordinates $(\xx,\yy,\zz)$ for $\Sigma^0 \subset \Sph^3$]  
\label{D:xx} 
We define a map 
$\Thetahat:\R^3\to \Sph^3$ (recall \ref{sph} and \ref{D:x1}) by 
\begin{equation*} 
  \Thetahat (\xx,\yy,\zz) := (\cos\xx \cos\yy \cos\zz, \cos\xx \sin\yy \cos\zz, \sin\xx\cos\zz,\sin\zz),  
\end{equation*} 
and coordinates $(\xx,\yy,\zz)$ on $\Sph^3 \setminus \CCperp \subset\R^4$, with $\yy$ defined $\pmod{2\pi}$, 
by     
restricting $\Thetahat$ to 
a covering map $\Theta: (-\pi/2,\pi/2) \times \R \times (-\pi/2,\pi/2) \to\Sph^3\setminus \CCperp \subset\R^4$. 
\end{definition} 

Clearly then the $\R^4$ coordinates in \ref{D:x1} satisfy 
\begin{multline} 
\label{E:xx} 
x^1= \cos\xx \cos\yy \cos\zz, 
\qquad 
x^2= \cos\xx \sin\yy \cos\zz, 
\\    
x^3= \sin\xx\cos\zz, 
\qquad 
x^4= \sin\zz.     
\end{multline} 

Note that 
the coordinates $\xx$ and $\yy$ can be thought of as the geographic latitude and longitude on $\Sigma^0 $, 
the coordinate $\zz$ as the signed distance from $\Sigma^0$,  
and in $(\xx,\yy,\zz)$ coordinates we have $\Sigma^0 \setminus \{\pm\pp^0\} = \{ \, \zz=0 \, \}$ and $\CC = \{ \, \xx=\zz=0 \, \}$.  

\begin{notation}[Regions of $\Sph^3$] 
\label{N:xx} 
By \ref{Nsub} and \ref{D:xx} we have  for $I\subset\R$ 
$$
\Sph^3_{\xx\in I} := \Thetahat (I\times\R\times\R) 
\qquad \text{ and } \qquad  
\Sph^3_{\yy\in I} := \Thetahat (\R\times I\times\R).  
$$
For \ $a,a'\in\R$  \  we have also (for example) \     $\Sph^3_{\xx = a } = \Sph^3_{\xx\in \{a,a'\} }$. \   
We also define for $\xx_0 \in \R            $ the parallel circle (or pole) \    $\CC^\parallel_{\xx_0} := \Sigma^0_{\xx=\xx_0}$. 
\end{notation}

Clearly on $\Sph^3 \setminus \CCperp$ we have (recall \ref{sph}) 
\begin{equation}
\label{EThetag}
g = \cos^2\zz\,( \, d\xx^2 +\cos^2\xx \,d\yy^2 \,)\, + d\zz^2.
\end{equation}

We define now 
$\PiSph:\Sph^3 \setminus \{\pm\pp^{\pi/2} \} \to\Sz   $
by
\begin{equation}
\label{EPiSph}
\PiSph(x^1,x^2,x^3,x^4)=
\frac1{|(x^1,x^2,x^3,0)|}
(x^1,x^2,x^3,0).
\end{equation}
Clearly $\PiSph$ is  
a nearest-point projection 
and satisfies 
\begin{equation}
\label{EPiSphTheta}
\PiSph \circ \Theta(\xx,\yy,\zz)
=
\Theta(\xx,\yy,0).
\end{equation}

To simplify the notation given $A\subset\Sz$ we define $\rA :\Sph^3 \setminus \{\pm\pp^{\pi/2} \} \to \R$ by (recall \ref{EPiSph} and \ref{tubular}) 
\begin{equation}
\label{ErS}
\rA := \dbold^{\Sz, g}_A \circ \PiSph  
\end{equation}
where as in \ref{tubular} we may write $\rp$ instead of $\rA$ if $A=\{p\}$.

\begin{definition}[Coordinates $(\xxx,\yyy,\zzz)$ for $\T       \subset \Sph^3$]  
\label{D:xxx} 
We define a map $\Thetaringhat :\R^3\to \Sph^3$ 
(recall \ref{sph} and \ref{D:x1}) by 
\begin{multline*} 
  \Thetaringhat (\xxx,\yyy,\zzz) 
\, :=  \, (\cos\yyy \cos (\zzz+\pi/4) , \sin\yyy \cos (\zzz+\pi/4) , 
\\ 
\cos\xxx \sin (\zzz+\pi/4) , \sin\xxx \sin (\zzz+\pi/4)\, ) ,     
\end{multline*} 
and coordinates $(\xxx,\yyy,\zzz)$ on $\Sph^3 \setminus \left( \CC \cup \CCperp  \right) \subset\R^4$, 
with $\xxx$ and $\yyy$ defined $\pmod{2\pi}$, 
by     
restricting $\Thetaringhat$ to 
a covering map $\Thetaring: \R \times \R \times (-\pi/4,\pi/4) \to \Sph^3 \setminus \left( \CC \cup \CCperp  \right) \subset\R^4 $. 
Finally we continuously extend 
$\yyy,\zzz$ to $\CC$ and $\xxx,\zzz$ to $\CCperp$ by taking $\yyy=\yy$, $\zzz=-\pi/4$ on $\CC$ and $\xxx=\xx$, $\zzz=\pi/4$ on $\CCperp$ (recall \ref{D:xx}).    
\end{definition} 

Note that 
$\xxx = \sqrt2\xx $ $\yyy= \sqrt2\yy $ and $\zzz= \zz $, 
where just here $\xx \yy \zz$ denotes the coordinates defined in \cite{kapouleas:yang} instead of the coordinates in \ref{D:xx}.  
The $\R^4$ coordinates in \ref{D:x1} satisfy 
\begin{multline} 
\label{E:xxx} 
x^1=  \cos\yyy \cos (\zzz+\pi/4) , 
\qquad 
x^2=  \sin\yyy \cos (\zzz+\pi/4) , 
\\    
x^3= \cos\xxx \sin (\zzz+\pi/4) , 
\qquad 
x^4= \sin\xxx \sin (\zzz+\pi/4) ,     
\end{multline} 
and in $(\xxx,\yyy,\zzz)$ coordinates we have $\T                              = \{ \, \zzz=0 \, \}$,  
$\CC =\{\zzz=-\pi/4\}$  and $\CCperp=\{\zzz=\pi/4\}$. 
Clearly on $\Sph^3 \setminus \left( \CC \cup \CCperp  \right)$  we have (recall \ref{sph}) 
\begin{equation}
\label{ETringg}
g \, = \, \gSph      \, = \,\tfrac12 (1+\sin2\zzz) \, d\xxx^2 \,+ \,\tfrac12 (1-\sin2\zzz)  \, d\yyy^2 \, + \,d\zzz^2.
\end{equation}

In analogy with \ref{points} we define 
\begin{equation}
\label{Tpoints} 
\pp_{\phi}^{\phi'} := \Thetaringhat (\phi',\phi,0) = \overline{\,\pp_\phi\pp^{\phi'}\,} \cap \T \qquad \qquad \forall \phi',\phi\in \R .  
\end{equation}
We also define great circles 
\     $\forall \phi\in \R$,    
\begin{equation}
\begin{gathered} 
\label{Tcircles} 
\CCslash_{\phi} := \{ \pp_{\phi_1}^{\phi_2} : \phi_2-\phi_1=\phi \} \subset \T, 
\\     
\CCbackslash_{\phi} := \{ \pp_{\phi_1}^{\phi_2} : \phi_2+\phi_1=\phi \} \subset \T. 
\end{gathered} 
\end{equation}

\subsection*{Basic geometry of catenoids} 
\nopagebreak

\begin{notation}[The cylinder $\cyl$]     
\label{Ecyl}
Let 
$\cyl : = \R\times \Sph^1 \subset \R\times\R^2$ be the standard cylinder, 
and $\chiK$ the standard product metric on $\cyl$.  
We define \    $(\sss     ,\vartheta)$ \     to be the standard coordinates on $\cyl$ 
consistent with        the covering \    $\Thetacyl:\R^2\to\cyl$ \    given by 
$$
\Thetacyl(\sss,\vartheta   ) := (\sss,\cos\vartheta   ,\sin\vartheta   ) 
\qquad \text{so that} \qquad
\chiK = d\sss^2  + d\vartheta^2 .
$$
For \    $\sbar \in \R$ \    we define a \emph{parallel circle} \quad 
$$  
\cyl_{\sbar} :=  \{ \Thetacyl(\sbar,\vartheta) : \vartheta\in\R \}\subset \cyl,  
$$ 
and for \    $I\subset \R$, \    $f\in C^0(\cyl_I)$, \     $j\in\N$, 
\ we define \ 
$\cyl_I := \bigcup_{\sbar\in I} \cyl_{\sbar}$  \quad and 
$$ 
f_\avg =f_{(0\parallel )} , \qquad  f_\osc =  f_{(0\perp     )} , \qquad f_{(j\parallel)} , f_{(j\perp)}        \in C^0(\cyl_I), \ 
$$
by requesting that on each parallel circle in $\cyl_I$ we have the following. 
The average of \    $f$ \    equals \    $f_\avg$, 
\ the average of \    $f_\osc$ \    vanishes, 
\    $f_{(j\parallel)}$ \    is in the space of $j$-th harmonics of the circle, 
\    $f_{(j\perp)}$ \    is orthogonal under \    $\chiK$ \    to the $j$-th harmonics of the circle, 
\quad             $f=f_\avg + f_\osc$,  
\quad             and \quad             $f=   f_{(j\parallel)}  +  f_{(j\perp)}$.  
\end{notation}

\begin{definition}[Euclidean catenoids] 
\label{Dcat}
\qquad
Given $\tau\in\R_+$, 
we define 
a parametrization 
$X_{\tildecat} = \XKt : \cyl \rightarrow \tildecat[\tau]\subset \R^3 $ 
of a catenoid $\Kbb[\tau] \subset \R^3 $ of size $\tau$ 
by taking 
\begin{equation}
\label{Ecatenoid}
\begin{gathered}
\XKt(\sss, \vartheta) \, = \, ( \, \rho(\sss) \cos \vartheta  \, ,  \, \,  \rho(\sss) \sin \vartheta \, ,  \,  z (\sss) \, ), \\  
\text{ where } \qquad \rho(\sss) := \tau\cosh \sss, \qquad  z (\sss):=\tau\, \sss, \\ 
\text{ and hence } \qquad 
\sss = \arccosh (\rho/ \tau ), \qquad  
z= \tau\arccosh (\rho/ \tau ).  
\end{gathered}
\end{equation}
\end{definition} 

The part above the waist of $\tildecat[\tau] $ is the graph of a radial function $\phicat = \phicat[\tau]:[\tau,\infty)\to\R$ given by
\begin{multline}
\label{Evarphicat}
\phicat[\tau](\rr):=
\tau\arccosh (\rr / \tau ) 
\\ 
= \tau\left(\log \rr-\log \tau+\log\left(1+\sqrt{1-{\tau^2}{\rr^{-2}}\,}\right)\right)
 \\
=
\tau\left(\log \frac { 2 \rr } {\tau} 
+
\log\left(\frac12+\frac12\sqrt{1-\frac{\tau^2}{\rr^{2}}\,}\right)\right), 
\end{multline}
where $\rr$ is the distance from the origin on $\R^2\subset\R^3$  
so that 
\begin{equation}
\label{Erhoz}   
\quad  \phicat[\tau] \circ \rho = z, \quad 
\end{equation}
and by direct calculation or balancing considerations 
\begin{equation}
\label{Ecatder}
\frac{\partial\phicat}{\partial\rr_{\phantom{cat}}}(\rr) 
=
\frac\tau{\sqrt{\rr^2-\tau^2\,}}.
\end{equation}

\begin{notation}[Catenoids] 
\label{Ncyl}
We denote the first and second fundamental forms of $\Kbb=\Kbb[\tau]$ by $\gK$ and $\AK$,  
its Gauss map by \ $\nu_\Kbb = \nu_\Kbt :\Kbb[\tau] \to\Sph^2(1)$, \ 
and in analogy with \ref{ELh} we define 
\quad 
$\hK    \, := \, \frac12(|A|_\Kbb^2+2) \, g_\Kbb $. 
\quad 
Furthermore we identify \     $\cyl$ \ with \ $\Kbb[\tau]$ \ via \ $\XKt$ \ in the sense of \ref{Nid}. 
\end{notation}

By straightforward computations we have then 
\begin{equation} 
\label{Ecatg}
\begin{aligned}
g_\tildecat \, :=& \, X_{\tildecat}^*  g \, = \, \rho^2(\sss)  \left( d\sss^2 + d \vartheta^2 \right) \, = \, \rho^2\,  \chiK \, = \, \tau^2 \cosh^4 \sss \  \gS,  
\\ 
A_\tildecat \, :=& \, X_{\tildecat}^*  A \, =   \, \tau \left( d\sss^2 - d \vartheta^2 \right) , 
\quad  
|A|_\K^2 \, =  \, 2 \tau^{-2} \, \sech^4 \sss \, = \, 2\tau^2\rho^{-4}, 
\end{aligned}
\end{equation} 
\begin{equation} 
\label{Ecatnu}
\begin{gathered} 
\nuK (\sss, \vartheta) \, = \, 
 ( \, \sech \sss \, \cos \vartheta  \, , \, \sech \sss \, \sin \vartheta  \,  , \, \tanh \sss \, ) ,
\\
\gnuK      \, = \, \frac{|A_\Kbb|^2}2 \, g_\Kbb \, = \, \sech^2\sss \left( d \sss^2 + d \vartheta^2 \right) \, = \, \sech^2\sss \, \,  \chiK,
\end{gathered} 
\end{equation} 
\begin{multline} 
\label{EKh}
\hK    \, := \, \frac{|A_\Kbb|^2+2}2 \, g_\Kbb \, = 
\, ( \tau^{-2} \, \sech^4 \sss  + 1 ) \, g_\Kbb 
\, = \, \big( 1 + \tau^{-2} \rho^4 \big) \, \, \gnuK      
\\ 
= \,\big( 1 + \tau^{2} \cosh^4\sss \big) \, \, \gnuK       
\, = \, \left( \tau^2\rho^{-2}  +  \rho^2       \right) \,  \chiK      
\, = \, \left( \sech^2\sss + \tau^2 \cosh^2\sss \right) \,  \chiK   .  
\end{multline}

\begin{definition}[Parametrization of $\Sz$] 
\label{Dcyl}
We identify in the sense of \ref{Nid} 
\ $\cyl$ \ and \ $\Sz   \setminus \{\pm\pp^0\}$ \  
via the conformal diffeomorphism 
\ $\ThetaSphcyl:  \cyl \rightarrow \Sz   \setminus \{\pm\pp^0\}$ \ 
defined by 
\begin{multline}
\label{Esphcyl}
\ThetaSphcyl \circ\Thetacyl( \, \sss \, ,  \, \yy  \,   ) 
= \Thetahat(\xx,\yy,0) 
\\
= ( \, \sech \sss \, \cos \yy \, , \, \sech \sss \, \sin \yy  \,  , \, \tanh \sss \, , \, 0 \, ) ,
\end{multline}
where by 
\ref{D:xx}, \ref{EThetag}, \ref{Ecyl}, and straightforward computations, we have 
\begin{equation} 
\begin{gathered} 
\label{Esssxx}
\sss = {\displaystyle{\log\frac{1+\sin\xx}{\cos\xx} }} = {\displaystyle{\log\frac{\cos\xx}{1-\sin\xx} }} ,   
\\
\qquad \cos \xx = \sech \sss, 
\qquad \sin \xx = \tanh \sss , 
\qquad \vartheta = \yy  , 
\\ 
\frac{d \sss}{d \xx} = \frac{1}{\cos \xx}, 
\qquad \frac{d\xx}{d\sss} = \sech \sss,
\qquad  \gS      = (\sech^2 \sss) \, \chiK = \gnuK.      
\end{gathered}       
\end{equation} 
\end{definition}

By \ref{Ecyl}, \ref{Ncyl} and \ref{Dcyl} the metrics \     $\gS  = \gnuK   $,  \  $g_\Kbb$, \ $\hK   $, \ and \ $\chiK$ \      
can be considered defined on any of \     $\cyl$, \ $\Sz   \setminus \{\pm\pp^0\}$, \ $\Kbb[\tau]$ \      
and are conformal to each other by 
\begin{multline} 
\label{Econformal}
\chiK  \, = \, 
\tau^{-2}\sech^2\sss \, g_\Kbb 
\, = \, 
\cosh^2 \sss \,\, \gS      
\, = \, 
\cosh^2 \sss \,\, \gnuK      
\\ 
\, = \, 
\tfrac{2\rho^{-2}}{|A|_\Kbb^2+2} \, \hK    
\, = \, 
(\sech^2\sss + \tau^2 \cosh^2 \sss \, )^{-1} \, \hK   . 
\end{multline} 
Note that the definition of $\hK$ is motivated by the Jacobi operator of small catenoidal bridges in $\Sph^3(1)$ 
and it         can be approximated by 
$\gS  = \gnuK   $ at the core of the catenoid and by $\gK$ away from the core. 
In the $\hK$ metric $\Kbt$ then resembles a round sphere attached to a catenoidal end via a small neck 
whose size is determined by calculating 
that the length of the parallel circle $(\cyl_\sss , \hK   )$ is 
\quad
$2\pi \sqrt{ \sech^2\sss + \tau^2 \cosh^2\sss \, }$ 
\quad
and hence is minimized at 
\begin{equation} 
\label{Esh}
\sss= \sss_{short}[\tau] := \arccosh \tau^{-1/2}     
\quad \iff \quad 
\rho= \rho_{short}[\tau] := \sqrt\tau, 
\end{equation} 
implying minimum length \quad $2^{3/2} \, \pi \, \sqrt\tau$.     
Moreover the neck has small area because for $\sbar\in [10, \, \sss_{short}[\tau] \,]$ we have (recall \ref{tubular})  
\begin{equation} 
\label{Earea}
\area_{\hK}\big( \K[\tau,\sqrt{\tau} ] \setminus \K[\tau, \rho(\sbar)] \big) \le \tfrac52\pi \sech^2(\sbar) = \tfrac52\pi \tau^2/\rho^2(\sbar). 
\end{equation}

The corresponding operators (recall \ref{ELjacobi})  
\quad $\Lcal_{\Sigma^0}$, $\Lcal_\Kbb$, $\LhK$, and $\LchiK$ \quad   
are given by 
\begin{equation} 
\label{dLchi}
\begin{gathered} 
\Lcal_{\Sigma^0}  := \Delta_{\Sigma^0} + 2,         
\\     
\Lcal_\Kbb  = \Lcal_\Kbt  :=  \Delta_\Kbt + |A|^2_\Kbb = \Delta_\Kbt + 2\tau^{-2} \, \sech^4 \sss ,         
\\ 
\LhK = \Lcal_{\Kbt,\hK   } = \tfrac2{|A|_\Kbb^2+2} \Lcal_\Kbt = \Delta_{\hK   } + 2 ( 1 + \tau^{2} \, \cosh^4 \sss )^{-1},   
\\ 
\LchiK = \Lcal_{\Kbt,\chi} := \rho^2(\Delta_\Kbt+|A|_\Kbt^2) = \Delta_\chiK + 2 \sech^2\sss,    
\end{gathered} 
\end{equation} 
and are conformal to each other by 
\begin{equation} 
\label{EconformalL}
\LchiK \, = \, \sech^2 \sss \, \Lcal_{\Sigma^0} \, = \, 
\tau^{2}\cosh^2\sss \, \, \Lcal_\Kbb \, = \, 
(\sech^2\sss + \tau^2 \cosh^2 \sss \, ) \, \LhK. 
\end{equation} 
The Jacobi equation for a rotationally invariant function $\phi$ amounts then to 
\begin{align}
\label{ELchirot}
\LchiK \phi \, \equiv \, \frac{d^2\phi}{d\sss^2} + 2 \sech^2\sss \,  \, \phi \, = \, 0;
\end{align}
and the (conformally invariant) \emph{flux of $\phi$ } is given by 
\begin{align}
\label{Eflux}
F_\phi(\sss) \, := \, \int_0^{2\pi} \frac\partial{\partial\sss}\phi \,\, d\vartheta = {2\pi} \frac{\partial\phi}{\partial\sss}(\sss).   
\end{align}

\begin{definition}[Jacobi fields on $\Sz$ ] 
\label{D:JS} 
\       
We define 
\     $\phio, \phie\in C^\infty( \cyl )$ 
\quad 
by (recall \ref{Esssxx} and note the change of sign in $\phie$ relative to \cite{SdI,SdII}) 
\begin{equation*} 
\begin{aligned} 
\phio(\sss) \, := \, &\, \tanh \sss \, = \, \sin\xx , 
\\
\phie(\sss) \, := \, &\, \sss \tanh \sss -1 \, = \, \sin\xx \, \log\frac{1+\sin\xx}{\cos\xx} -1 \, = \, - \sin\xx \, \log\frac{1-\sin\xx}{\cos\xx} -1 . 
\end{aligned} 
\end{equation*} 
\end{definition}

\begin{lemma}[{\cite[2.19]{SdII}}]
\label{Lphie}
The following hold for \  $\phie$ \   and \   $\phio$ \   as in \ref{D:JS}. 
\begin{enumerate}[label=\emph{(\roman*)}]
\item 
\label{Lphie1}
$\phie$ and $\phio$ satisfy \ref{ELchirot} and are even and odd respectively in $\sss$ (or $\xx$).   
\item 
\label{Lphie2}
$\phio(\sss)$ is strictly increasing in $\R$ and corresponds to a translation on $\Kbt$ or $\Sigma^0 \subset \R^4$.  
\item 
\label{Lphie3}
$\phie(\sss)$ is strictly increasing on $[0, \infty)$ with a unique root 
\quad 
$\sroot \in  ( \, 1.1996 \, , \, 1.1997 \, )$ 
\quad 
and corresponds to scaling on $\Kbt$.  
\item 
\label{Lphie4}
$\phie(\sss)+\phio(\sss) $ is strictly increasing on $[0, \infty)$ with a unique root 
\quad 
$\sbarroot \in  ( \, 0.6837 \, , \, 0.6838 \, )$ 
\quad 
and \quad $\sbarroot/ \sroot \in ( \, 0.5699 \, , \, 0.57   \, )$.  
\item 
\label{Lphie5}
The fluxes of $\phie$ and $\phio$ (recall \ref{Eflux}) are given by  
$$
F_{\phio}(\sss)=  2\pi\sech^2\sss, 
\qquad 
F_{\phie}(\sss)=  2\pi (\, \tanh\sss + \sss \sech^2\sss\,), 
$$
and converge to $0$ and $2\pi$ as $\sss\to\infty$ respectively.  
\item 
\label{Lphie6}
The Wronskian satisfies
\[ W[\phio, \phie](\sss): = \phio(\sss) \partial \phie(\sss) - \phie(\sss)\partial \phio(\sss) = 1.\]
\end{enumerate}
\end{lemma}

\begin{proof}
It follows from straightforward calculations by using \ref{Esssxx}, definition \ref{D:JS}, and \ref{ELchirot}.  
The roots satisfy the equation 
\quad $(\sbar_j +  b    ) \tanh\sbar_j = 1$ \quad with $  b    =0,1$. 
Using this we    calculate their numerical values. 
\end{proof}

For future use we generalize now the ODE \eqref{ELchirot} to the ODE (for a rotationally invariant function $\phi$) 
\begin{align}
\label{ELchiODE}
\left( \Delta_{\Sz} + 2 + \lambda \right) \phi(\sss)  =0 \quad \iff \quad 
\frac{d^2\phi}{d\sss^2} + ( 2 +\lambda) \sech^2\sss \,  \, \phi \, = \, 0.
\end{align}

\begin{definition}[Perturbations of $\phio, \, \phie$ ] 
\label{D:phiDN} 
\       
Given \ $\lambda\in [-1/100,1/100]$, \, $\stil_0 \in [-1/100,1/100]$, \    and \ $\stil_1 \in [2\sroot,\infty),$    \    
we define (recall \ref{Dcyl})       
\linebreak 
\quad     $\phiD[\lambda,\stil_0 ] , \, \phiN[\lambda,\stil_0 ], \, \phien[\lambda,\stil_1 ] ,  \, \phisn[\lambda,\stil_1 ] \in C^\infty( \cyl )$ \quad 
to be the solutions of \ref{ELchiODE} with initial data 
\begin{equation*} 
\begin{aligned} 
&\phiD[\lambda,\stil_0 ] (\stil_0) =0, \qquad \qquad \tfrac{\partial}{\partial\sss}    \phiD[\lambda,\stil_0 ]   (\stil_0) =1, 
\\ 
&\phiN[\lambda,\stil_0 ] (\stil_0) =1, \qquad \qquad \tfrac{\partial}{\partial\sss} \phiN[\lambda,\stil_0 ] (\stil_0) =0, 
\\ 
&\phisn[\lambda,\stil_1 ] (\stil_1) =0, \qquad \qquad \tfrac{\partial}{\partial\sss} \phisn[\lambda,\stil_1 ] (\stil_1) =1, 
\\ 
&\phien[\lambda,\stil_1 ] (\stil_1) =1, \qquad \qquad \qquad \qquad \qquad 
\end{aligned} 
\end{equation*} 
and moreover 
\  $\phien[\lambda,\stil_1 ] $ \   smoothly extendible through \  $\pp^0$.       
\end{definition} 

\begin{lemma}[Perturbations of $\phio, \, \phie$ ]         
\label{LphiDN}
For \   $\lambda,\stil_0$ \   as in 
\ref{D:phiDN} and \   $\stil_1$ \   large enough, 
there is \ $C>0$ \   depending on \  $\stil_1$ \  
such that the following hold. 
\begin{enumerate}[label=\emph{(\roman*)}]
\item 
\label{LphiDN1}   
$\phiD[  0 , 0 ] = \phio =      \phio                    (\stil_1) \, \phien[  0    ,\stil_1 ]$   
\quad  and \quad    
$\phiN[  0 , 0 ] = - \phie $. 
\item 
\label{LphiDN2}   
$ \| \, \phiD[\lambda,\stil_0 ] - \phio \, : \, C^3( \cyl_{[0, \stil_1]} , \chiK ) \, \| \, \le \, C \, ( |\lambda| + |\stil_0| ). $
\item 
\label{LphiDN3}  
$ \| \, \phiN[\lambda,\stil_0 ] + \phie \, : \, C^3( \cyl_{[0, \stil_1]} , \chiK ) \, \| \, \le \, C \, ( |\lambda| + |\stil_0| ). $
\item 
\label{LphiDN4}   
$\phien[\lambda,\stil_1 ]$ \     is increasing in \  $\sss$, \   
$\tfrac{\partial}{\partial\sss} \phien[\lambda,\stil_1 ]$ \     is decreasing in \  $\sss$, \   
and \\ 
$\phien[\lambda,\stil_1 ] (\pp^0) - \phien[\lambda,\stil_1 ] (\stil_1) \, < \, \sech^2\stil_1 . $                  
\item 
\label{LphiDN5}   
$\phisn[\lambda,\stil_1] $  \     is increasing in \  $\sss$, \   
$\tfrac{\partial}{\partial\sss} \phisn[\lambda,\stil_1 ]$ \     is decreasing in \  $\sss$, \   
and \   
$1/2 \, < \, \tfrac{\partial}{\partial\sss} \phisn[\lambda,\stil_1] (\sss)                                                        \, < \, 1             $                         
\quad for \quad $\sss>\stil_1$. 
\item 
\label{LphiDN6}   
$ \phiD[\lambda,\stil_0 ] \, = \, A_{D} \, \phien[\lambda,\stil_1 ] \,+\, A'_D \, \phisn[\lambda,\stil_1 ] $ 
\quad with \\    
$ | A_D - \phio(\stil_1) | + | A'_D| \, \le \, C \, ( |\lambda| + |\stil_0| ). $
\item 
\label{LphiDN7}  
$ \phiN[\lambda,\stil_0 ] \, = \, A_{N} \, \phien[\lambda,\stil_1 ] \,+\, A'_N \, \phisn[\lambda,\stil_1 ] $ 
\quad with \\    
$ | A_N + \phie(\stil_1) | + | A'_N  +   \tfrac{\partial}{\partial\sss} \phie(\stil_1) | \, \le \, C \, ( |\lambda| + |\stil_0| ). $
\end{enumerate}
\end{lemma}

\begin{proof}
\ref{LphiDN1}  follows from \ref{Lphie} and ODE uniqueness.  
\ref{LphiDN2}  and   \ref{LphiDN3}   
follow from smooth dependence of ODE solutions on parameters and \ref{LphiDN1}.   
Integrating by parts we control the change of flux. 
For \ref{LphiDN4}  we use the $\xx$ coordinate and that the flux vanishes at the pole. 
For       \ref{LphiDN5}   we compare with the flux at $\stil_1$. 
\ref{LphiDN6}  and   \ref{LphiDN7}   
follow from 
\ref{LphiDN2},  \ref{LphiDN3}, and ODE uniqueness.    
\end{proof}

\bigbreak  
\begin{nopagebreak} 
\section{Side-symmetric doublings of gluing type} 
\label{S:Ddou} 

\subsection*{General theory}  
$\vspace{.062cm}$ 
\nopagebreak

We first recall a general definition proposed in \cite{gLD}.  
\end{nopagebreak} 

\begin{definition}[Surface doublings {\cite[Definition 1.1]{gLD}}] 
\label{Ddoubling} 
\ \ \ \   
Given a Riemannian three-manifold 
$(N,g)$ and a two-sided 
surface $\Sigma$ in $N$, 
we define a \emph{(surface) doubling $\Mhat$ over $\Sigma$ in $N$} (equivalently we say \emph{$\Mhat$ doubles $\Sigma$ in $N$}) 
to be a smooth surface $\Mhat$ in $N$ satisfying the following.  
\begin{enumerate}[label=\emph{(\roman*)}]
\item 
\label{dprojec} 
The nearest point projection $\Pi_\Sigma$ to $\Sigma$ in $N$ is well defined on $\Mhat$. 
\item 
$\Omegahat:=\Pi_\Sigma(\Mhat) \subset \Sigma $ is closed with smooth boundary $\partial\Omegahat$. 
\item 
\label{dgraphs} 
$\Mhat$ is the union of the graphs of $\ubreve^+$ and $-\ubreve^- \in C^0(\Omegahat) \cap C^\infty( \Omegahat \setminus\partial \Omegahat)$.  
\item 
\label{dbgraphs} 
$\ubreve^+ + \ubreve^- = 0$ on $\partial \Omegahat$, where the two graphs join smoothly with vertical tangent planes, and $\ubreve^+ + \ubreve^- > 0$ close to $\partial \Omegahat$ in $\Omegahat$. 
\item 
By 
the above 
$\left.\Pi_\Sigma\right|_{\Mhat}$ covers $\Omegahat \setminus \partial \Omegahat$ twice, 
$\partial\Omegahat$ once, 
and misses $\Sigma\setminus\Omegahat $. 
\end{enumerate} 
We call $(\Sigma,N,g)$ \emph{the background of the doubling $\Mhat$}, $\Sigma$ its \emph{base surface}, 
and each connected component of $\Sigma\setminus\Omegahat$ a \emph{doubling hole of $\Mhat$ over $\Sigma$}. 
Finally we call the doubling $\Mhat$ \emph{minimal} if $\Sigma$ and $\Mhat$ are minimal.  
\end{definition}

\begin{notation}[Graphical decomposition of doublings]       
\label{n:Mpm} 
Given $\Mbreve$ as in \ref{Ddoubling} we denote by $\Mbreve_\pm$ the graphs of $\ubreve^+$ and $-\ubreve^-$ so that $\Mbreve=\Mbreve_+\cup \Mbreve_-$ with disjoint interiors. 
We also define diffeomorphisms $\Xbreve_\pm:\Omegahat\to \Mbreve_\pm$ so that $\Pi_\Sigma\circ \Xbreve_\pm= I_{\Omegahat}$. 
\end{notation} 

\begin{definition}[Symmetric doublings \cite{k35,IIgLD}]
\label{dSS}
Given $(\Sigma,N,g)$ and $\Mhat$ as in \ref{Ddoubling} let $\group^{\Sigma,\Mhat,N,g}_\sym$ denote the group of isometries of $(N,g)$ fixing $\Sigma$ and $\Mhat$ as sets. 
We call the surface doubling $\Mhat$ \emph{$\group$-symmetric} if $\group$ is a subgroup of $\group^{\Sigma,\Mhat,N,g}_\sym$ 
and \emph{side-symmetric} if there is an isometry $\refl_\Sigma\in \group^{\Sigma,\Mhat,N,g}_\sym$ which fixes $\Sigma$ pointwise and exchanges its sides. 
In the side-symmetric case we clearly have the following. 
\begin{enumerate}[label=\emph{(\roman*)}]
\item 
\label{dSS1}
$\ubreve^+=\ubreve^-$ \quad and 
hence \ref{Ddoubling}\ref{dbgraphs} amounts to \quad $\ubreve=0$ \quad on $\partial \Omegahat$ and \quad $\ubreve> 0$ \quad close to $\partial \Omegahat$ in $\Omegahat$ 
where we write \quad $\ubreve:= \ubreve^+=\ubreve^-$ \quad in order to simplify the notation.  
\item 
\label{dSS2}
$A_\Sigma=0$.  
\end{enumerate}  
\end{definition}

\begin{convention}
\label{ConvArticle} 
In this article, unless stated otherwise, $\Sigma$ and $\Mhat$ as in \ref{Ddoubling} and \ref{dSS} are assumed embedded and connected. 
This  implies $\ubreve^+ + \ubreve^- > 0$ on $\Int(\Omegahat)$ and 
in the side-symmetric case $\ubreve>0$ on $\Int(\Omegahat)$. 
\end{convention}

The minimal surface doublings constructed by PDE gluing methods have special features which we describe in the side-symmetric case in \ref{dgtSS} below. 
We first recall the definition of LD solutions  (with slight modifications from the original {\cite[Definition 3.1]{SdI}}, some from {\cite[Definition 3.6]{gLD}}).  

\begin{definition}[LD solutions {\cite[Definition 3.1]{SdI}, \cite{IIgLD}}]
\label{dLD}
\ 
Let $\Sigma,g$ be as in \ref{Ddoubling}. 
We define a \emph{linearized doubling (LD) solution on $\Sigma$ of singular set $L\subset\Sigma$ and configuration   $\taubold: L \rightarrow \R_+            $}   
to be a function 
$
\varphi \in C^\infty (\, \Sigma \setminus L \, ) 
$ 
satisfying the following.   
\begin{enumerate}[label=\emph{(\roman*)}]
\item 
$L\subset \Sigma$  is dicrete    and 
\quad 
$\Lcal_\Sigma\varphi=0$ \ on \ $\Sigma\setminus L$, 
\ or in the case with kernel 
\quad 
$\Lcal_\Sigma\varphi \in \subker $,  
\quad 
where 
\quad 
$\subker \subset C^\infty( \Sigma\setminus L ) $ 
\quad 
is  a chosen subspace  
of dimension 
\quad 
$\dim \subker = \dim \ker \Lcal_\Sigma$.           
\item 
$\forall p \in L$ 
\quad 
the function   
\quad 
$\varphi - \tau_p \log \dbold_p^\Sigma$ 
\quad 
is bounded on some deleted neighborhood of $p$ in $\Sigma$,  
where $\tau_p$ is the value of $\taubold$ at $p$. 
\end{enumerate}
\end{definition}

\begin{definition}[Truncated catenoids]
\label{dK}
Given $\tau,\sbar,\rbar> 0$ we define 
\begin{equation*} 
\begin{gathered} 
\K[\tau,\rbar] := \K[\tau]\cap \{ ( x , y , z ) \in \R^3:  x^2 +  y^2 \le \rbar^2    \}
\\ 
\text{ and } \qquad 
\K[\tau;       \sbar] := \XKt ( \cyl_{[-\sbar,\sbar]} ) = \K[\tau, \tau\cosh \sbar]. 
\end{gathered} 
\end{equation*} 
\end{definition}

\begin{definition}[cf. {\cite[Definition 4.2]{gLD}}]
\label{dnorm} 
We have the following where 
$C^{k,0}$ stands for $C^k$, 
$C^{k,1}$ stands for $C^{k+1}$, 
and $\cbar(k)>0$ is a constant depending  only on $k$. 
\begin{enumerate}[label=\emph{(\roman*)}]
\item 
\label{dnormK}
Let $Y$ be an isometry from $\R^3$ with its standard Euclidean metric onto a Euclidean space.  
For a domain \quad $\Udom\subset\K_Y := Y \K[\tau, \tau^\alpha]$ \     
and \     $k \in \N, \     \beta \in [0, 1]$, \     $\gamma \in \R$, \        
we define (recall \ref{Ecatg}, \ref{D:norm-g}, and \ref{dK})   
\begin{equation*}
\qquad\qquad
\| u \|_{k, \beta, \gamma; \Udom } 
\, := \, \| u : C^{k, \beta}(\Udom , \rho, \gK, \rho^\gamma) \|   
\, \Sim^{\cbar(k)}  \, 
\| u : C^{k,\beta}(   \Udom , 1  , \chiK ,\, \rho^\gamma ) \|.  
\end{equation*}
\item 
\label{dnormS}
Let $\Sigma,g$ be as in \ref{Ddoubling} and $L$ as in \ref{dLD}. 
We use the notation \ $\rr:= \dbold^{\Sigma}_L$ \   (recall\ref{tubular}) and 
we define the conformal metric (not everywhere smooth) 
\  $\chibar := \rr^{-2} \gSigma$, also 
for $p\in L$ we define \ $\Sigma_{(p)}:= \Sigma_{\rr=\dbold^{\Sigma}_p} $. 
For a domain \quad $\Udom\subset\Sigma $ \quad 
and \quad $k \in \N, \quad \beta \in [0, 1]$, \quad $\gamma \in \R$, 
we define (recall \ref{Ecatg}, \ref{D:norm-g}, and \ref{dK})   
\begin{align*}
\qquad\quad
\| u \|_{k, \beta, \gamma; \Udom } 
\, := \, \| u : C^{k, \beta}(\Udom , \rr  , \gSigma , \rr^\gamma) \|   
\, \left( \Sim^{\cbar(k)}  \, 
\| u : C^{k,\beta}(   \Udom , 1  , \chibar ,\, \rr^\gamma ) \| \right),   
\end{align*}
where we assert the $\Sim$ relation only for       $\Udom$ contained in a small enough neighborood of some $p\in L$. 
\end{enumerate}
\end{definition}

\begin{definition}[Side symmetric doublings of gluing type {\cite{k35,IIgLD}}]
\label{dgtSS}
\ 
Let $\Mbreve$ be a minimal surface doubling over $\Sigma$ in $(N, g)$ as in \ref{Ddoubling},  
which is also side-symmetric and $\group$-symmetric (with $\refl_\Sigma\in\group$) as in \ref{dSS} 
and satisfies \ref{ConvArticle}.  
Let 
$\taubreve, \alpha \in (0,  1/3 )$  and $\gamma= 3/2$. 
We say the side-symmetric minimal surface doubling $\Mbreve$  is \emph{of $(\taubreve, \alpha, \group)$-gluing type}  
if the following hold. 
\begin{enumerate}[label=\emph{(\roman*)}]
\item 
\label{dgtSS0}
\emph{(The LD solution):} 
There is a  
$\group$-invariant \emph{LD solution} 
$\varphi : \Sigma \setminus L  \to \R$ \quad of 
configuration \quad $\taubold : L \rightarrow (\taubreve^{1+\alpha/100}, \taubreve]$  \quad 
satisfying the following. 

\item 
\label{dgtSS1}
\emph{(The graphical region):} 
We introduce the notation 
(recall \ref{n:Mpm})  
\begin{equation} 
\label{n:Mgr}
\begin{gathered} 
\SigmaU : = \Sigma \setminus \disjun_{p \in L} D_p(\tau^{5\alpha}_p) \subset \Sigmabreve \subset \Sigma, 
\\ 
\Mbreve_{\gr\pm} : = X_\pm(\SigmaU) \subset \Mbreve_\pm,  
\qquad \text{and} \qquad 
\Mbreve_{\gr} : = \Mbreve_{\gr+}  \disjun \Mbreve_{\gr-},   
\end{gathered} 
\end{equation} 
and then we have the following  
(recall \ref{dnorm}).   

\begin{enumerate}[label=\emph{(\alph*)}]
\item 
\label{dgtSS1z}
To simplify the notation we identify \quad 
$\SigmaU$ with  $\Mbreve_{\gr\pm}$ via      $X_\pm$ \quad  as in            \ref{Nid}.  \     
\item 
\label{dgtSS1a}
$\| \varphi - \ubreve : C^{3       }(\SigmaU) \| \leq  \taubreve^{1+ \alpha/ 4}$.
\item 
\label{dgtSS1b}
$\| \ubreve : C^{3       }(\SigmaU) \|  < \| \varphi : C^{3       }(\SigmaU) \| + \taubreve^{1+ \alpha/4} \leq  \taubreve^{8/9}$.
\item 
\label{dgtSS1c}
$\| \Xbreve^*_{\pm}  g_{\Mbreve} - g_{\Sigma} : C^{3       }(\SigmaU) \| \leq  \taubreve^{8/9}$.
\item 
\label{dgtSS1d}
$\| \Xbreve^*_{\pm}  A_{\Mbreve} : C^{2       }(\SigmaU) \|          \leq  \taubreve^{8/9}$ (recall \ref{dSS}\ref{dSS2}). 
\end{enumerate}

\item 
\label{dgtSS2}
\emph{(The catenoidal bridges):} 
\ 
We have the following $\forall p \in L$. 
	\begin{enumerate}[label=\emph{(\alph*)}]
\item 
\label{dgtSS2a}
There is an isometry  
\quad $Y_p: (\R^3,g) \to (T_pN, \left. g \right|_p ) $ \quad 
mapping 
\quad $\R^2 \simeq \  \{z=0\}$ \      onto  \      $T_p\Sigma$,  
and a diffeomorphism 
\quad  $\Xbreve_p : \K_p \rightarrow \Kbr_p \subset \Mbreve$   \quad   
with inverse \quad  $\Pi_{\K_p} : \Kbr_p \to \K_p $, \quad       
where (recall \ref{dK})     
\quad  $\K_p \, :=  \, Y_p \, \K [\tau_p, \tau^\alpha_p ] \subset T_pN$ \quad and \quad $\Kbr_p\cap \Sigma$ \quad is smooth and radially graphical in $\Sigma$ around $p$. 
\item 
\label{dgtSS2b}
To simplify the notation we identify as in \ref{Nid} \quad 
$\K [\tau_p, \tau^\alpha_p ] $ with  $\K_p$ via      $Y_p$ 
\quad and \quad $\K_p$ with $\Kbr_p$ via $\Xbreve_p $. 
\item 
\label{dgtSS2c}
The parametrizations $\Xbreve_p$ are $\group$-equivariant 
under the following actions of $\group$. 
The action of $\group$ on $N$ induces an action on $ \disjun_{p\in L} T_pN$ and hence on $\K_L \, := \, \disjun_{p\in L} \K_p $. 
The action of $\group$ on $N$ induces an action on $\disjun_{p\in L} \Kbr_p \subset \Mbreve$ by restriction. 
\item 
\label{dgtSS2d}
$\| \gM         - \gK    \, \|_{3, 0, 4           \, ;  \, \K_p                        } \, \leq \, 1$.   
\item 
\label{dgtSS2e}
$\| A_{\Mbreve} - A_{\K} \|_{2, 0, 2    \, ;  \, \K_p                        }  \leq  \tau_p    \,  | \log\tau_p    | $. 
\item 
\label{dgtSS2f}
$\| \, \ubreve - \phicat[\tau] \circ  \dbold^\Sigma_p        \, \|_{3, 0, \gamma \, ;  \, D_p^\Sigma(\tau^{\alpha}_p) \setminus D_p^\Sigma(100\, \tau_p) } \leq  \taubreve^{1+\alpha/6}$. \quad 
\item 
\label{dgtSS2h}
$D^\Sigma_p(\tau_p(1-\tau_p^{2  })) \subset \Dbreve_p \subset D^\Sigma_p(\tau_p(1+\tau^{2  }_p))$  
\quad 
where 
\quad 
$\Dbreve_p$ 
\quad 
is a doubling hole,  
\quad 
$\Sigma\setminus\Sigmabreve=\disjun_{p \in L} \Dbreve_p $,  
\quad 
and \quad $\partial \Dbreve_p = \Kbr_p\cap \Sigma$. 
	\end{enumerate}
\end{enumerate}
\end{definition}

Note that given $\Mbreve$ as in the previous definition, 
none of $\taubold$, $\varphi$, or $\Xbr_p$ are uniquely determined, because the defining conditions are open allowing small changes.

\begin{assumption} 
\label{A:sim0} 
From now on we assume that $\Mbreve$ is as in \ref{dgtSS} with $\taubreve$ as small as needed in terms of $\alpha$ and the background,    
and $\alpha$ as small as needed in terms of the background.    
\end{assumption}

\begin{notation}[Regions and metrics on $\Mbreve$]       
\label{n:Kp} 
Given $\Mbreve$  as in \ref{A:sim0} 
we define   
\ $\Kbr_L:=\disjun_{p\in L} \Kbr_p$ \ 
and 
\ $\forall p\in L$ \  and \  $\sbar\in[0,\infty)$ \  
(recall \ref{Nsub}, \ref{Ecatenoid}, and \ref{dgtSS}\ref{dgtSS2}), \    
$$ 
\K_{p,\sbar} := \big( \K_p \big)_{\sss\le\sbar}, 
\qquad 
\Kbr_{p,\sbar} := \big( \Kbr_p \big)_{\sss\le\sbar}.   
$$ 
We also define on  \  $\Kbr_p$ \  in analogy with \ref{Ecatg} \quad    $\chiM := ( \rho^{-2}\circ \Pi_{\K_p} ) \, \gM$.  \quad    
Recall also that by \ref{ELh} that when \quad    $|A|_\Mbr^2 + \Ric(\nu,\nu) > 0$  \quad    
we have on  $\Mbreve$ the metric 
\quad 
$\hM     \, := \, \frac12(\, |A|_\Mbr^2 + \Ric(\nu,\nu) \, ) \, \gM    $. \quad 
\end{notation}

\begin{notation}[Decompositions of gluing type doublings] 
\label{DSMbr}
Given $\Mbreve$ as in \ref{A:sim0}, 
a decomposition \ $\ L= \disjun_{j\in\jbold   }L_j$ \ where $\jbold$ is a finite set of indices, and  $\sbar_j>0$ for each $j\in\jbold        $, 
we define a decomposition into compact domains with disjoint interiors by defining 
\     
$
\Sbreve[L_j,\sbar_j] :=  \disjun_{p\in L_j} \Kbr_{p,\sbar_j} 
$ 
\     and \quad      
$
\Sbreve [\Sigma,\ \{\sbar_j\}_{j\in\jbold} \, ] \, := \, \Mbreve \setminus \,{\Int } \big( \disjun_{j\in\jbold   } \Sbreve[L_j,\sbar_j] \big) . \quad    
$
Moreover we introduce extra notation for a particular such decomposition by taking 
\quad     
$\forall p\in L$ \     
$
\Sbreve_p :=  \Sbreve[\, \{p\}\, , \, \sss_{short}[\tau_p] \,] = \K[\tau_p , \, \sqrt{\tau_p} \,] ,  
$ 
\quad  
for $L'\subset L$ 
\quad  
$
\Sbreve_{L'} \, := \, \disjun_{p\in L'}        \Sbreve_p            , 
$
\quad  
and 
\quad  
$
\SbreveS \, := \, \Mbreve \setminus \,{\Int } \big( \SbreveL \big)            . \   
$
\end{notation} 

Note that $\SbreveS$ has two connected components, one in each side of $\Sigma$.

\subsection*{Minimal doublings of $\Sigma=\Sigma^0$ in $\Sph^3$ of gluing type} 
$\vspace{.062cm}$ 
\nopagebreak

\begin{assumption} 
\label{A:sim} 
In the rest of this section we assume that $\Mbreve$ is as in \ref{A:sim0} 
with background $( \, \Sigma=\Sigma^0, N=\Sph^3(1) , g \,)$. 
\end{assumption}

Claerly the operators (recall \ref{ELjacobi})  
\     $\LcalS$ \ on $\Sigma =\Sigma^0$  corresponding to \ $\gS$, \ and \ 
\     $\LcalM$, \ $\LhM$ \  on $\Mbreve$,  \ and \ $\LchiM$ on $\KbrL$, corresponding to \ $\gM$, \ $\hM$,  \ and \ $\chiM$, \  
are given by 
\begin{equation} 
\label{dLchiM}
\begin{gathered} 
\LcalS      = \Delta_\Sigma  +2 , \qquad 
\LcalM      = \Delta_\Mbreve + |A|^2_\Mbreve +2, 
\\ 
\LhM    = \Lcal_{\Mbreve , \hM   } = \tfrac2{|A|^2_\Mbreve+2} \LcalM     = \Delta_{h_\Mbreve} + 2,  
\\ 
\LchiM = \Lcal_{\Mbreve,\chiM} = \rho^2 \LcalM                  = \Delta_\chiM + \rho^2 ( |A|^2_\Mbreve +2 ) . 
\end{gathered} 
\end{equation}

\begin{lemma}[Operators on the graphical regions]          
\label{LSh} 
Given $\Mbreve$  as in \ref{A:sim} 
we have $\forall u\in C^{k,\beta}(\SigmaU)$   
that 
\quad 
$\|\LcalM u - \LcalS u : C^1(\SigmaU) \| \, \le \, \taubreve^{7/9} \| u :C^{3      }(\SigmaU) \|$. 
\end{lemma} 

\begin{proof} 
Follows from 
\ref{dgtSS}\ref{dgtSS1}\ref{dgtSS1c}-\ref{dgtSS1d}. 
\end{proof}

\begin{lemma}[Metrics on the bridges]          
\label{LKh} 
\quad 
Given $\Mbreve$  as in \ref{A:sim} with $\taubreve$ small enough in terms of given $\sbar\in[0,\infty)$,  
we have the following $\forall p\in L$,  a domain $U\subset \Kbr_p$, and $\beta\in[0,1]$. 
\begin{enumerate}[label=\emph{(\roman*)}]
\item 
\label{LKh1} 
$\| \hM         - \gnuK  \, : \, C^3(  \K_{p,\sbar}  , \gnuK )  \, \| \leq C(\sbar) \, \taubreve^2 \,  | \log\taubreve | $ \quad (bridge core).   
\item 
\label{LKh2} 
$\| \chiM         - \chiK    \, \|_{3, 0, 2 \, ;  \, \K_p                        } \, \leq \, 1$.   
\item 
\label{LKh3} 
$\| \, |A|^2_\Mbreve - |A|^2_\K   \, \|_{2, 0, -2          \, ;  \, \K_p                        } \, \leq \, \tau_p^2 \, |\log \tau_p |  $.     
\item 
\label{LKhj} 
$\|\LchiM u - \LchiK u \|_{0,\beta,2;U} \, \le \, 5 \| u :C^{2,\beta}(U, \chiK ) \|$ \quad for \quad  $u\in C^{2,\beta}(U)$.  
\item 
\label{LKhf} 
$\|\ubreve\|_{3,0,0;\K_p} \le \tau_p    \,  | \log\tau_p    | $ 
\     and \    
$\| \, \ubreve_\osc \, \|_{3, 0, \gamma \, ;  \, D_p^\Sigma(\tau^{\alpha}_p) \setminus D_p^\Sigma(100\, \tau_p) } \leq  \taubreve^{1+\alpha/6}$.       
\item 
\label{LKh5} 
$\| \, \hM - \hK \, : \, C^2 \big( \K_p, \chiK, \max(\tau_p^2 |\log\tau_p|          ,      \rho^{ 4})   \big)                \, \|           \, \leq \, 2          $.   
\item 
\label{LKh5b} 
$\| \, \hM - \hK \, : \, C^0(\K_p, \hK, \rho^2 \, \|           \, \leq \, 2  |\log\tau_p|         $.   
\item 
\label{LKh4} 
$\forall \sbar\in [10, \, \sss_{short}[\tau] \,]$ \quad we have \quad  
\\
$\phantom{k}$ \hfill 
$\area_{\hM}\big( \K[\tau,\sqrt{\tau} ] \setminus \K[\tau, \rho(\sbar)] \big) \le 3\pi \sech^2(\sbar) = 3\pi \tau^2/\rho^2(\sbar)$.   
\end{enumerate}       
\end{lemma} 

\begin{proof} 
For simplicity in this proof we fix $p\in L$ and write $\tau$ for $\tau_p$. 

\ref{LKh1} 
By \ref{EKh} and \ref{Esssxx} we have 
\quad 
$\hK = \big( 1 + \tau^{2} \, \cosh^4 {\log\frac{1+\sin\xx}{\cos\xx} }  \big) \, \gS $     
\quad 
and then we complete the proof using \ref{dgtSS}\ref{dgtSS2}\ref{dgtSS2d}-\ref{dgtSS2e}. 

\ref{LKh2} 
Follows from \ref{dgtSS}\ref{dgtSS2}\ref{dgtSS2d} and \ref{n:Kp}.                

\ref{LKh3} 
We use     
\quad $\rr^2\log\taubreve + \rr^\gamma\taubreve^{1+\alpha/6}\le \rr^{1+\gamma+ \alpha/6} \le \rr$ \quad 
and \   \ref{EKh}, \ref{Esssxx}, \ref{dgtSS}\ref{dgtSS2}\ref{dgtSS2d}-\ref{dgtSS2e}. 

\ref{LKhj} 
Since 
\quad 
$\LchiM-\LchiK \, = \, \Delta_\chiM- \Delta_\chiK \, + \, \rho^2 \big( |A|^2_\Mbreve - |A|^2_\Kbb \big) + 2\rho^2$,     
\quad 
this follows from \ref{LKh2},\ref{LKh3}.  

\ref{LKhf} 
We clearly have by \ref{Dcat} \ 
$\|  z    \|_{3,0,0;\K_p} \, \le \, (1-\alpha) \tau_p    \,  | \log\tau_p    | $ \     
and \ $z_\osc=0$. 
The result follows then by \ref{dgtSS}\ref{dgtSS2}\ref{dgtSS2f}.       

\ref{LKh5} and \ref{LKh5b}  
Follow from \ref{LKh2}, \ref{LKh3}, and the definitions.

\ref{LKh4} 
Follows from \ref{LKh5b} and \ref{Earea}. 
\end{proof} 

\begin{cor} 
\label{CKh} 
\quad 
Given $\Mbreve$  as in \ref{A:sim} with $\taubreve$ small enough in terms of given $\cbar_0\in(0,\infty)$,  
we have the following 
for any $p\in L$,  $\sbar\in [ \cbar_0, 9\cbar_0 ]$,        and $\beta\in[0,1]$  
on  \ $\Kbreve'_p \, := \, \Kbreve_{p ,\sbar \le \sss }$.  \ 
\begin{enumerate}[label=\emph{(\roman*)}]
\item 
\label{CKh1} 
The first Dirichlet eigenvalue \     $\lambda_1( \LhMD  , \Kbreve'_p )> \cbar(\cbar_0)>0$,  
where \  $\LhMD := (\LhM , \BbouD)$ \, 
and \      $\BbouD$ denotes restriction to $\partial \Kbreve'_p$ (recall \ref{D:oper} and \ref{D:eigenb}). 
\item 
\label{CKh2} 
There is a linear map 
\ $\Rcal_{ \Kbreve'_p }  : C^{0,\beta}( \, \Kbr'_p \, ) \to  C^{2,\beta}( \, \Kbr'_p \, ) $ \     
such that 
$\forall E \in C^{0,\beta}( \, \Kbr'_p \, )$ 
\ $ \up = \Rcal_{ \Kbreve'_p }  E $ \ 
is the (unique by \ref{CKh1}) solution to the Dirichlet problem 
\begin{equation} 
\label{Eu}
\begin{gathered} 
\LchiM        \up  = E    
\qquad \text{on} \qquad \Kbr'_p,  
\qquad \quad   
\up  =0 \qquad \text{on} \qquad \partial \Kbr'_p.  
\end{gathered} 
\end{equation} 
\item 
\label{CKh3} 
For $u'$ and $E$ as in \ref{CKh2} we have the estimate 
$$ 
\| \up  \, : \;  C^{2,\beta}( \, \Kbr'_p , \chiK \, ) \, \|    
\, \le \, C(\cbar_0) \, |\log\tau_p           | \, 
\| E \, : \;  C^{0,\beta}( \, \Kbr'_p , \chiK \, ) \, \|.   
$$ 
\end{enumerate}       
\end{cor} 

\begin{proof} 
Using arguments which are standard by now
as for example in \cite[Appendix B]{kapouleas:1990}, 
and a logarithmic cutoff function supported on $\Kbr_{p,\sbar'}\subset \Sbreve_p$ with $\sbar'$ independent of $m$ and large enough in terms of $\cbar_0$, 
and estimating also on the $\hM$-neck \ $\Kbr'_{p,\sss\ge\sbar'}$ \ the Dirichlet eigenfunctions on $\Kbr'_p$ using the $\chiK$ geometry as in \cite{kapouleas:1995} for example, 
we can prove that 
$$ 
\left| \, \lambda_1( \LhMD  , \Kbreve'_p )> \cbar(\cbar_0) - \lambda_1( \Delta_{\Sph^2}+2  , \Sph^2_{\sss\ge\sbar}       ) \, \right| 
$$  
is as small as needed in terms of $\cbar_0$. 
%is isometric to a domain contained in a hemisphere of $\Sph^2$ at a distance from the equatorial circle bounded below uniformly in terms of $\cbar_0$. 
Since $\sbar\ge\cbar_0$ \ref{CKh1} follows. 

We claim now that \ref{CKh2}-\ref{CKh3} hold if we replace $\LchiM$ with $\LchiK$. 
Using separation of variables the claim can be reduced to the case \  $u'_\osc=E_\osc=0$; \ 
$u'$ by \ref{dLchi} satisfies then on $\Kbr'_p$ the ODE 
\quad 
$
\LchiK u' = \Delta_\chiK u' + 2\tau\rho^{-2} u' = E.   
$
Consider now $u_-$ and $u_+$ satisfying the ODE \ $\LchiK u_\pm =0$ \  and \ $u_\pm=0$   \   at  one boundary circle of $\Kbr'_p$ each. 
Their logarithmic derivatives  at some $\sbar''$ large independently of $m$ differ by at least a positive constant also independent of $m$ by the argument for \ref{CKh1}. 
By comparing $u'$ with $u_\pm$ we prove the claim. 
The proof is then competed by referring to \ref{LKh}\ref{LKhj} which implies that 
\quad $\|\LchiM u - \LchiK u :C^{0,\beta}( \Kbr'_p , \chiK ) \|_{0,\beta,2;U} \, \le \, 5 \tau^{2\alpha}\, \| u :C^{2,\beta}(U, \chiK ) \|$ \quad for \quad  $u\in C^{2,\beta}(U)$.  
\end{proof}

\subsection*{Approximate low spectrum} 
$\vspace{.062cm}$ 
\nopagebreak

The following norm, which is useful in the study of the eigenfunctions, is the sum of an $L^2$ norm and an interior $C^{2,\beta}$ norm. 
%??? 

\begin{definition}[Norm on $\Sbreve_p$, $\SbreveL$, or $\SbreveS$]      
\label{DSbr-norm}   
For $\epsilon>0$ and $f\in C^{2,\beta}(\Sbreve)$, 
where $\Sbreve=\SbreveS$ or  $\Sbreve=\SbreveL$ or  $\Sbreve=\Sbreve_p$ for some $p\in L$,  
we define (recall \ref{tubular}) 
$$
\| \, f \, \|_{\epsilon;\Sbreve} := \| \, f \, : \, L^2 (\Sbreve, h) \, \| + \| \, f \, : \, C^{2,\beta} ( \Sbreve\setminus D^{\Sbreve,h}_{\partial \Sbreve}(\epsilon) \,  ) \, \|. 
$$ 
\end{definition}

\begin{definition}[Constants and slidings]   
\label{Dsl}
Let $\Mbreve$ be    as in \ref{A:sim}, and $\grouptilde$ and $\Vt$ be as in \ref{D:eigenb} with $\Lcalt=\Lcalh:=\LhM$ as in \ref{n:Kp}, 
where (recall \ref{D:symL}) the diffeomorphism in each     symmetry in $\grouptilde$ is in $\group$ (as in \ref{dgtSS}). 
We define then the following. 
\begin{enumerate}[label=\emph{(\roman*)}]
\item 
\label{Dsl1}
The space of $\Vt$-constants \\    ${\Vt}_{\con}:= \big\{ {f}:L\to \F \, \big| \, {f} \text{ is preserved by } \grouptilde\big\}$.     
\item 
\label{Dsl2}
The space of $\Vt$-slidings (sections of $TN \big|_L \otimes\F$) 
\\    
${\Vt}_{\Sl} := \big\{ \vec{f}:L\to TN\otimes\F \, \big| \, \vec{f} (p)\in T_p N\otimes \F, \quad \vec{f} \text{ is preserved by } \grouptilde\big\}$.     
\end{enumerate}       
\end{definition}

On the regions in \ref{DSMbr} we have the operators 
$\LcalhD:= (\LhM  ,\BbouD)$ and $\LcalhN:= (\LhM  ,\BbouN)$ defined as in \ref{D:eig-res},  
where 
$\BbouD$ denotes restriction to the boundary, and $\BbouN$ action by $\eta_h$, corresponding to Dirichlet and Neumann boundary conditions respectively. 

\begin{lemma}[Approximate low spectrum on catenoidal bridges] 
\label{L:appr} 
\ \  
Under the same assumptions as in \ref{Dsl} 
and if $\taubreve$ is small enough in terms of a given $\epsilon\in (0,1/2)$  and the given background $(\Sigma,N,g)$,  then the following hold 
(recall \ref{D:eigenb} and \ref{DSMbr}). 
\begin{enumerate}[label=\emph{(\roman*)}]
\item 
\label{L:appr1} 
$\#_{<-2}(\Vt,\LcalhD,\SbreveL) =  \#_{<-2}(\Vt,\LcalhN,\SbreveL) =0$.
\item 
\label{L:appr2} 
$
\#_{[-2 \, , \, -2+\epsilon]}(\Vt,\LcalhD,\SbreveL) \, = \, \#_{[-2 \, , \, -2+\epsilon]}(\Vt,\LcalhN,\SbreveL) \, = \, \dim \Vt_{\con} 
$.
\item 
\label{L:appr3} 
$
\#_{[-2+\epsilon \, , \, -\epsilon ]}(\Vt,\LcalhD,\SbreveL) \, = \, \#_{[-2+\epsilon \, , \, -\epsilon] }(\Vt,\LcalhN,\SbreveL) \, = \, 0              
$.
\item 
\label{L:appr4} 
$
\#_{[-\epsilon \, , \, \epsilon ]}(\Vt,\LcalhD,\SbreveL) \, = \, \#_{[-\epsilon \, , \, \epsilon ]}(\Vt,\LcalhN,\SbreveL) \, = \, \dim \Vt_{\Sl}   
$. 
\item 
\label{L:appr5} 
$
\#_{[\epsilon \, , \, 4-\epsilon ]}(\Vt,\LcalhD,\SbreveL) \, = \, \#_{[\epsilon \, , \, 4-\epsilon] }(\Vt,\LcalhN,\SbreveL) \, = \, 0              
$.
\end{enumerate}       
\end{lemma} 

\begin{proof}
We assume the background $(\Sigma,N,g)$ fixed so that we will not mention the dependence of constants on it. 
Recall that $\forall p\in L$ we have on $\Sbreve_p$ the metrics $\gM$, $\gS$, $\chiM$, and $\chiK$.  
We choose $\sbar$ as large as needed and assume $\taubreve$ small enough in terms of the chosen $\sbar$. 
We have then by \ref{LKh}\ref{LKh1} that $\gM$ and $\gS$ are as close as needed on $\K_{p,\sbar}\subset \Sbreve_p$ 
and by \ref{LKh}\ref{LKh2} that $\chiM$ and $\chiK$ are as close as needed on $\Sbreve_p$. 
We have then uniform estimates on the eigenfunctions of $\LcalhD$ and $\LcalhN$ on $\Sbreve_p$. 
Using standard arguments, 
as in \cite[Appendix B]{kapouleas:1990} for example, 
and a logarithmic cutoff function supported on $\K_{p,\sbar}\subset \Sbreve_p$, 
we can prove using also \ref{LKh}, 
that the low eigenvalues of $\LcalhD$ and $\LcalhN$ on $\Sbreve_p$ are close to the low eigenvalues of $\Delta_{\gS}+2$ on $\Sigma^0$. 
Using this we conclude the proof. 
\end{proof}

\begin{prop}[Approximate low spectrum] 
\label{P:appr} 
Under the same assumptions as in \ref{Dsl} 
we have the following. 
If $\taubreve$ is small enough in terms of a given $\epsilon\in (0,1/2)$ and the given background $(\Sigma,N,g)$,  
then the following hold 
(recall \ref{D:eigenb}). 
\begin{enumerate}[label=\emph{(\roman*)}]
\item 
\label{P:appr1} 
$\#_{<-2}(\Vt,\Lcalh,\Mbreve ) =0$.
\item 
\label{P:appr4} 
$\forall a\in [-2 +\epsilon     , -\epsilon  ]$ we have 
\begin{multline*}   
\dim \Vt_{\con}  +      \#_{[-2, a]}(\Vt,\LcalhD,\SbreveS )  
\, \le \, 
\#_{[-2, a]}(\Vt,\Lcalh,\Mbreve ) 
\\ 
\, \le \, \dim \Vt_{\con}   +     \#_{[-2, a]}(\Vt,\LcalhN,\SbreveS ).   
\end{multline*}   
\item 
\label{P:appr3} 
$\dim \Vt_{\Sl} \le     \#_{(-\epsilon, \epsilon)} (\Vt,\Lcalh,\Mbreve ) $.
\item 
\label{P:appr5} 
$\forall a\in [\epsilon, 4-\epsilon]$ 
we have 
$$
\dim \Vt_{\con}  + \dim \Vt_{\Sl} + \#_{[-2, a]}(\Vt,\LcalhD,\SbreveS ) \, \le \, 
\#_{[-2, a]}(\Vt,\Lcalh,\Mbreve ). 
$$
\item 
\label{P:appr6} 
$\forall a\in [0, 4-\epsilon]$ 
we have 
$$
\#_{[-2, a]}(\Vt,\Lcalh,\Mbreve )  
\, \le \, \dim \Vt_{\con}  + \dim \Vt_{\Sl} + \#_{[-2, a]}(\Vt,\LcalhN,\SbreveS ).   
$$
\end{enumerate}       
\end{prop} 

\begin{proof}
We use the decomposition with disjoint interiors $\Mbreve= \SbreveS\cup\SbreveL$. 
Imposing Dirichlet or Neumann conditions and applying \ref{L:appr} we obtain the desired lower and upper bounds respectively.   
\end{proof}

\begin{corollary}[Index bounds]  
\label{C:appr} 
If $\Mbreve$  is of \emph{$(\taubreve, \alpha, \group)$-gluing type} as in \ref{A:sim} and $\taubreve$ is small enough in terms of the given background $(\Sigma,N,g)$,  
then we have (recall \ref{D:eigenb}) 
\begin{multline*} 
\genus (\Mbreve ) \, - \, 2\genus(\Sigma) \, + \, 1 
\, \le \, 
\ind(\Mbreve) 
\\ 
\, \le \, 
4\genus (\Mbreve ) \, - \, 8\genus(\Sigma) \, + \, 
\ind(            \LcalhN,\SbreveS ) \, + \, 4     .   
\end{multline*} 
\end{corollary} 

\begin{proof}
We apply the proposition with $\Vt$ the space of $\R$-valued functions on $\Mbreve$ and trivial $\grouptilde$. 
The corollary then follows from \ref{P:appr}\ref{P:appr4},\ref{P:appr6} and the obvious facts that $\dim \Vt_{\con} =   |L|$,    $\dim \Vt_{\Sl} =   3|L|$,    and 
$\genus (\Mbreve ) =  2\genus(\Sigma) + |L| - 1$, where the order $|L|$ of $L$ equals the number of catenoidal bridges.
\end{proof}

\bigbreak  
\section{Symmetries of minimal doublings of $\Sigma^0$ in $\Sph^3$} 
\label{S:symS} 

\subsection*{Symmetry assumptions}          
$\vspace{.126cm}$ 
\nopagebreak

In this section we study the symmetries shared by the $\Sph^2$ minimal doublings we consider in this article. 
We assume they have been positioned appropriately with respect to the coordinate system in \ref{D:x1} 
and in particular they are side-symmetric doublings over $\Sigma^0$.

Following \cite[(2.13)]{SdI} and \cite[(2.10) and Remark 2.11]{KW:Luniqueness}, 
given $m\in \N\setminus \{1,2\}$, 
we define the union of $m$ meridians of $\Sigma^0$ symmetrically arranged around the poles $\pp^0$ and $\pp^\pi = -\pp^0$, 
by  
\begin{equation} 
\label{ELmer}
\Lmer=\Lmer[m]:=\bigcup_{i\in 2\Z} ( \overline{ \ptti \pp^0 } \llcup \overline{ \ptti \pp^\pi} ).  
\end{equation} 
Following  
\cite[Definition 3.8(iv)]{KW:Luniqueness} 
we define a collection of $2+m$ great two-spheres   
\begin{multline}
\label{ESigma} 
\bfSigma  := \bfSigma [m] 
\\ 
:= \big\{ \, \Sigma^{ j\pi /2 } = \Sph( \CC, \pttj ) \, \big\}_{j\in\Z} \lllcup \big\{ \, \Sigma_{ i\pi /m } = \Sph( \CCperp , \ptti ) \, \big\}_{i\in\Z}   
\end{multline} 
and as in \cite[Definition 3.9(ii)]{KW:Luniqueness} we define the group 
\begin{equation} 
\label{Egroup} 
\GrpSigma = \GrpSigmam 
:= \Gsym^{\Lmer[m]} = 
\big\langle \refl_{\Sigma^0}, \refl_{\Sigma^{\pi/2}}, \refl_{\Sigma_0}, \refl_{\Sigma_{\pi/m}} \big\rangle \simeq \Z_2\times \Z_2\times D_{2m}, 
\end{equation} 
where $D_{2m}$ denotes the dihedral group of order $2m$.

By \cite[Lemmas 3.10 and 3.11]{KW:Luniqueness}, $\GrpSigma$ is an index two subgroup of $\Gsym^{\xL_{m-1,1}}$ 
where the Lawson surface $\xL_{m-1,1}$ is positioned so that it contains the $2m$ great circles in 
\begin{equation}
\label{ECQ} 
\mathbf{C}_Q:=  \left\{ \Sph( \ptti,\pttj ) = {\cC_{i\pi/m}^{j\pi/k} } \right\}_{i,j\in\frac12+\Z}.  
\end{equation} 
Note that 
reflections with respect to these circles are contained in $\Gsym^{\xL_{m-1,1}}$ but not in $\GrpSigma$.

For large $m$, $\GrpSigmam$  is the symmetry group of 
$\xibreveomo$ and $\xibrevemm $ (discovered in \cite{SdI}), 
and $\xbreve_{(m:\kcir)}$ and $\xbreve_{(1,m:\kcir,1)}$ (discovered in \cite{SdII}).   
Each of these doublings corresponds to an LD solution with singular set 
\begin{equation} 
\label{EL} 
L=L[m]:= \Lpar\cap\Lmer, 
\end{equation} 
where $\Lpar$ is a union of circles parallel to $\CC$ in $\Sigma^0$. 
$\Lpar$ is invariant under $\refl_{\Sigma^{\pi/2} }$  
and therefore 
$\Lpar$ and $L$ are also invariant under $\GrpSigmam$. 
Note that more minimal doublings of $\Sph^2$ are discovered in \cite{gLD,IIgLD}. 
The ones in \cite{gLD} and some in \cite{IIgLD} have the same symmetry group $\GrpSigmam$ but with larger singular set $L$ compared to \ref{EL}. 
The remaining doublings of $\Sph^2$ discovered in \cite{IIgLD} have $\Lpar$ which is not invariant under $\refl_{\Sigma^{\pi/2} }$. 
They are invariant under reflections with respect to the $m+1$ spheres of  $ \bfSigmap [m]$,   
and their symmetry group is \   $\GrpSigmap$, defined by  
\begin{equation} 
\label{Egroupp} 
\begin{gathered} 
\bfSigmap  := \bfSigmap [m] 
:= \bfSigma [m] \  {\big\backslash }  \, \big\{ \, \Sigma^{ \pi/2 } = \Sph( \CC, \pp^{\pi/2} ) \, \big\} ; 
\\ 
\GrpSigmap = \GrpSigmapm 
:= 
\big\langle \refl_{\Sigma^0}, \refl_{\Sigma_0}, \refl_{\Sigma_{\pi/m}} \big\rangle \simeq \Z_2\times D_{2m}. 
\end{gathered} 
\end{equation} 
Clearly $\GrpSigmap = \GrpSigmapm$ is  an index two subgroup of $\GrpSigmam$ and has order $|\GrpSigmap|=4m$.    
In the remaining of this section however we restrict our attention to those minimal surfaces in $\Sph^3$ which satisfy the following. 

\begin{assumption}[Required symmetries]
\label{AMbreve}   
In the rest of this section $\Mbreve=\Mbreve[m]$ denotes a surface in $\Sph^3$ which is symmetric under $\GrpSigmam$ (defined in \ref{Egroup}).  
\end{assumption} 

\subsection*{Function and eigenfunction decompositions} 
$\vspace{.126cm}$ 
\nopagebreak

As in \cite{LindexI}, our approach is based on function decompositions induced by the symmetries. 
In this article we also consider $\C$-valued functions in order to facilitate 
the description of some refined new decompositions (see \ref{Dderot} and \ref{Lderot}): 

\begin{notation}[The spaces $V,W$] 
\label{DVW}   
We will denote by $W$ the $\C$-linear space of $\C$-valued functions on $\Mbreve=\Mbreve[m]$  
and by $V\subset W$ the $\R$-linear space of $\R$-valued functions on $\Mbreve=\Mbreve[m]$.  
We define $\Re,\Im:W\to V$ by taking $\Re f$ and $\Im f$ to be the real and imaginary parts of $f\in W$. 
\end{notation} 

\begin{definition}[Decompositions by reflections] 
\label{Dderefl}   
\ 
Using symmetrizations 
\\ 
and antisym\-me\-tri\-zations we define projections in $W$ as follows. 
\begin{enumerate}[label=\emph{(\roman*)}]
\item 
$\forall f\in W$, 
$\PP^{\pm} f := \frac12 (\, f \, \pm  \, f\circ\refl_{\Sigma^0} \, )$,  
$\PP^{*\pm} f := \frac12 (\, f \, \pm  \, f\circ\refl_{\Sigma^{\pi/2} } \, )$,  
and $\PP_{\pm} f := \frac12 (\, f \, \pm  \, \overline{ f\circ\refl_{\Sigma_0}}  \, )$, 
where $\PP_{\pm}$ is only $\R$-linear.  
\item 
$\PP^{\pm\pm}:=\PP^\pm\circ \PP^{*\pm}$  
and 
$\PP^{\pm\pm}_{\pm}:=\PP^\pm\circ \PP^{*\pm} \circ \PP_{\pm}$.  
\item 
The restrictions of the projections to $V$ are projections in $V$ which we denote by removing the underline, 
for example $\PV^{+-}_- := \PP^{+-}_- \big|_V:V\to V$.  
\end{enumerate} 
We define also 
$W^{\pm\pm} :=\PP^{\pm\pm}W$,  $W^{\pm\pm}_{\pm} := \PP^{\pm\pm}_{\pm} W$,  
$V^{\pm} :=\PP^{\pm}V$,  $V^{*\pm} :=\PP^{*\pm}V$,  $V^{\pm\pm} :=\PP^{\pm\pm}V$,  $V^{\pm\pm}_{\pm} := \PP^{\pm\pm}_{\pm} V$  
and for $f\in W$ 
$f^{\pm\pm} :=\PP^{\pm\pm}f$,  $f^{\pm\pm}_{\pm} := \PP^{\pm\pm}_{\pm} f$.   
\end{definition} 

Since the reflections 
$\refl_{\Sigma^0}$,  
$\refl_{\Sigma^{\pi/2} }$,  
and $\refl_{\Sigma_0}$ 
commute with each other, 
the projections 
$\PP^{\pm}$, 
$\PP^{*\pm}$,   
and $\PP_{\pm}$,   
also commute with each other. 
Moreover we have various decompositions, as for example 
\begin{equation} 
\label{EIVpm} 
\begin{aligned} 
I_V 
\, =& \, \, \PV^{++}   +    \PV^{+-}_+   +    \PV^{+-}_- + \PV^{-+}_+   +    \PV^{-+}_- + \PV^{--},
\\
V 
\, =& \, \, V^{++} \Oplus V^{+-}_+ \Oplus V^{+-}_- \Oplus V^{-+}_+ \Oplus V^{-+}_- \Oplus V^{--}. 
\end{aligned} 
\end{equation} 
Since the Jacobi operator $\Lcal = \Delta+|A|^2+2$ on $\Mbreve$ commutes with these projections by \ref{AMbreve}, 
the decomposition above induces decompositions of the eigenspaces of $\Lcal$. 
This played a fundamental role in \cite{LindexI}. 
In this article we refine these decompositions further.  

\begin{definition}[Projections by rotations] 
\label{Dderot}   
For $\mi\in\Z$ we define the following. 
\begin{enumerate}[label=\emph{(\roman*)}]
\item 
\label{Dderot1}   
$\PP_{\mi} f := \frac1m \sum_{l=0}^{m-1}  \, e^{-2l\mi\pi i/m} \, f\circ\rot_{2l\pi/m}^\CC $ ($\, \forall f\in W$).  
\item 
\label{Dderot2}   
$\PP^{\pm\pm}_{\mi} := \PP^{\pm\pm} \circ \PP_{\mi}$ and $\PP^{\pm\pm}_{\mi\pm} := \PP^{\pm\pm}_{\pm} \circ \PP_{\mi}$.  
\item 
\label{Dderot3}   
$\PV_{\mi} := \big(\PP_{\mi}+\PP_{-\mi}\big)\big|_{V}$ unless $\mi=0$ or $m/2$ $\pmod m$ in which case $\PV_{\mi} := \PP_{\mi}\big|_{V}$.  
\item 
\label{Dderot4}   
$\PV^{\pm\pm}_{\mi} := \PV^{\pm\pm} \circ \PV_{\mi}$ and $\PV^{\pm\pm}_{\mi\pm} := \PV^{\pm\pm}_{\pm} \circ \PV_{\mi}$.  
\end{enumerate} 
We define $\C$-linear    spaces 
\ $W_{\mi} :=\PP_{\mi}W$  \    
and $\R$-linear spaces 
\ $W^{\pm\pm}_{\mi\pm} := \PP^{\pm\pm}_{\mi\pm} W \subset W $ \  (but of $\C$-valued functions),  
\ $V_{\mi} :=\PV_{\mi}V \subset V $,  \  and \ $V^{\pm\pm}_{\mi\pm} := \PV^{\pm\pm}_{\mi\pm} V \subset V $.  
\ For $f\in W$ we define 
$f_{\mi} :=\PP_{\mi}f$,  $f^{\pm\pm}_{\mi\pm} := \PP^{\pm\pm}_{\mi\pm} f$.  
\end{definition} 

Because of the definition of $\PV_\mi$ we will need the following.

\begin{notation}  
\label{N:mb} 
We define $\mbar:=[m/2]$,  $\mo := \N \cap \, [ \, 2 \, , \, m/2 \,]$,  
and 
$\mtwo := \N \cap \, [ \, 2 \, , \, m/2 \,)$ 
so that the following hold. 
\begin{enumerate}[label=\emph{(\roman*)}]
\item 
\label{N:mb0}     
$\N \cap \, [ \, 0 \, , \, m/2 \,) \, = \, \{0,1\} \sqcup \mtwo$   
and 
$\N \cap \, [ \, 0 \, , \, m/2 \,] \, = \, \{0,1\} \sqcup \mo$.  
\item 
\label{N:mb1}     
If $m$ even, then 
$\mbar =m/2$, $\mo := \{2,...,\mbar \} $ and $\mtwo := \{2,...,\mbar-1 \} $ 
\item 
\label{N:mb2}     
If $m$ odd, then 
$\mbar =(m-1)/2$ and $\mo = \mtwo = \{2,...,\mbar \} $.  
\end{enumerate} 
\end{notation}  

\begin{lemma}[Decompositions by rotations] 
\label{Lderot}   
\hfill 
The following hold,  
where 
\\ 
$\sum_{\pm\pm\pm}$ denotes summation 
and $\Oplus_{\pm\pm\pm}$ denotes direct sum over the eight possible combinations of $+$ and $-$  
(in both cases). 
\begin{enumerate}[label=\emph{(\roman*)}]
\item 
\label{Lderot1}   
$\PP^{\pm}$, 
$\PP^{*\pm}$,   
all $\PP_\mi$ with $\mi\in\Z$ (or $\PP_{\pm}$),   
and $\Lcal$ (or $\Lcalh$), 
commute with each other; 
each $\PP_\mi$ does not commute with $\PP_{\pm}$ unless $\mi=0$ or $m/2$ $\pmod m$; 
$\Lcal$ and $\Lcalh$ also do not commute. 
\item 
\label{Lderot2}   
$I_W= \sum_{\mi=0}^{m-1  } \PP_{\mi} = \sum_{\pm\pm\pm} \PP^{\pm\pm}_{\pm}$  
and 
$\PP^{\pm\pm}_{\pm} = \sum_{\mi=0}^{\mbar} \PP^{\pm\pm}_{\mi\pm}$ (recall \ref{N:mb}).   
\item 
\label{Lderot3}   
$\forall f\in W$ and $\forall \mi\in\Z$ we have $\PP_{-\mi} f = \overline{ \PP_{\mi} \overline{f} }$ 
and $W_{-\mi}=\overline{ W_{\mi} } $. 
\item 
\label{Lderot4}   
$\forall \mi\in\Z$ we have 
$W_\mi \, = \, \big\{ \, f\in W \, : \, f\circ\rot_{2\pi/m}^\CC = e^{2\mi\pi \ii/m}  f \, \big\}$.  
\item 
\label{Lderot4b}   
$\forall \mi\in\Z$ we have 
$W_\mi \, = \, \big\{ \, f\in W \, : \, f\circ\rot_{2j\pi/m}^\CC = e^{2j\mi\pi \ii/m}  f \, ( \forall j\in\Z ) \, \big\}$.  
\item 
\label{Lderot4c}   
$\forall \mi\in\Z$ 
and $\Om$ any fundamental domain of the action of $\rot_{2\pi/m}^\CC$ on $\Sph^3$,    
$W_\mi \, = \, \big\{ \, f\in W \, : \, f\circ\rot_{2j\pi/m}^\CC = e^{2j\mi\pi \ii/m}  f \, \text{ on } \, \Mbreve \cap \Om \, ( \forall j\in\Z ) \, \big\}$.   
\item 
\label{Lderot4+}   
$\forall \mi\in\Z$ we have 
\quad 
$W_{\mi + }= V_{\mi+} \, \Oplus_\R \, \ii\, V_{\mi-}$, 
\quad 
$W_{\mi - }= V_{\mi-} \, \Oplus_\R \, \ii\, V_{\mi+}$, 
\quad 
and on $V$ we have 
\quad 
$\PV_{\mi} = 2 \Re \circ \PP_{\mi} = 2 \Re \circ \PP_{-\mi} $ if $\mi\ne0$ or $m/2$ $\pmod m$ and 
$\PV_{\mi} = \PP_{\mi} = \PP_{-\mi} $ if $\mi = 0$ or $m/2$ $\pmod m$. 
In particular $W_\mu=\left<W_{\mu_\pm}\right>^{\C}$ where $\left< . \right>^{\C}$ denotes the $\C$-span. 
\item 
\label{Lderot5}   
$\forall \mi\in\Z$,  
\quad 
$V_\mi= \big\{ \, f\in V \, : \, f - 2 \cos \frac{2\mi\pi}m \, f\circ\rot_{2\pi/m}^\CC +  f\circ\rot_{4\pi/m}^\CC = 0 \, \big\}$.  
\item 
\label{Lderot6}   
$\forall \mi\in\Z$ we have $V_{\mi}= V_{\mi+} \, \Oplus_\R \, V_{\mi-} $ and $V_{-\mi\pm}= V_{\mi\pm} $. 
\item 
\label{Lderot7}   
$\forall \mi,\mi'\in\Z$ with $\mi=\mi'\pmod m$ we have $W_{\mi'}={ W_{\mi} } $ and $V_{\mi'}= V_{\mi} $. 
\item 
\label{Lderot8}   
$\PV^{\pm\pm} = \sum_{\mi=0}^{\mbar} \PV^{\pm\pm}_{\mi}$,  
\quad 
$\PV^{\pm\pm}_{\pm} = \sum_{\mi=0}^{\mbar} \PV^{\pm\pm}_{\mi\pm} $, 
\quad 
and 
\\    
$\phantom{kk}$ \hfill 
$I_V = \sum_{\pm\pm\pm} \sum_{\mi=0}^{\mbar} P^{\pm\pm}_{\mi\pm} $. 
\item 
\label{Lderot9}   
$W^{\pm\pm}_{\mi\pm}=W^{\pm} \cap W^{*\pm} \cap W_{\pm} \cap W_{\mi}$ and 
$V^{\pm\pm}_{\mi\pm}=V^{\pm} \cap V^{*\pm} \cap V_{\pm} \cap V_{\mi}$. 
\item 
\label{Lderot10}   
$V^{\pm\pm} = \Oplus_{\mi=0}^{\mbar} V^{\pm\pm}_{\mi}$,  
$V^{\pm\pm}_{\pm} = \Oplus_{\mi=0}^{\mbar} V^{\pm\pm}_{\mi\pm} $, 
and 
$V = \Oplus_{\pm\pm\pm} \Oplus_{\mi=0}^{\mbar} V^{\pm\pm}_{\mi\pm} $. 
\end{enumerate} 
\end{lemma} 

\begin{proof} 
It is straightforward to verify these statements.  
\end{proof} 

\begin{example}[Fourier modes on {$\Sigma^0$}] 
\label{Ex:fourier} 
Note that $\Sigma^0$ satisfies \ref{AMbreve} where the generating reflections act as follows in 
the spherical coordinates $(\xx,\yy)$ defined in \ref{D:xx}.  
$\refl_{\Sigma^0}$ acts trivially, $\refl_{\Sigma^{\pi/2}}$ by $\xx\to-\xx$ (exchanging the hemispheres defined by the equatorial circle $\CC$), 
$\refl_{\Sigma_0}$ by $\yy\to-\yy$, and $\refl_{\Sigma_{\pi/m}}$ by $\yy\to \frac{2\pi}m-\yy$. 
For $f$ a $\C$-valued function on $[-\pi/2 , \pi/2 ]$ and $l\in\Z$ the Fourier mode $f(\xx)  e^{ l \yy \ii }$ 
is in $W_\mi$ 
if 
$l=\mi \pmod m $, 
in $W^{*+}$ ($W^{*-}$) if 
$f$ is even (odd), 
and is always in $W_+$ with $f(\xx) \ii e^{ l \yy \ii } \in W_-$.  
If $f$ is $\R$-valued we have that 
$f(\xx)  \cos( l \yy )\in V_{\mi+}$ 
and 
$f(\xx)  \sin( l \yy )\in V_{\mi-}$ 
if 
$l= \pm \mi \pmod m $, 
and both modes are in $W^{*+}$ ($W^{*-}$) if 
$f$ is even (odd). 
\end{example}

\begin{lemma}[Decomposition of the complex eigenfunctions] 
\label{Lmu}   
\ \ 
Assuming \ref{AMbreve} we have 
\ 
$\forall j \in \N$, $\mi\in \Z$ 
\ 
with 
\ 
$\mi\ne0,m/2 \pmod m $ 
\ 
that    
\\ 
$\lambda_j^\C(W_{\mi}^{},\Lcalu) = \lambda_{j}^\C(W_{-\mi}^{},\Lcalu) = \lambda_j(V_{\mi+}^{},\Lcalu) = \lambda_j(V_{\mi-}^{},\Lcalu)$,
\\     
and similarly 
\quad 
$\lambda_j^\C(W_{\mi}^{\pm\pm},\Lcalu) = \lambda_{j}^\C(W_{-\mi}^{\pm\pm},\Lcalu) = \lambda_j(V_{\mi+}^{\pm\pm},\Lcalu) = \lambda_j(V_{\mi-}^{\pm\pm},\Lcalu)$,   
\\     
where  
\ 
$\Lcalu = \Lcal$ as in \ref{ELjacobi} or $\Lcalu = \Lcalh$ as in \ref{ELh}.  
\end{lemma} 

\begin{proof} 
It is straightforward to verify these statements using the symmetries and the definitions.  
\end{proof}

\subsection*{Constructions of test functions} 
$\vspace{.126cm}$ 
\nopagebreak

\begin{notation}[The region $\YY_0 \subset\Sph^3$] 
\label{NYY}   
We adopt the notation (recall \ref{N:xx}) 
\qquad 
$\YY_{0} \, := \Sph^3_{\yy\in [-\pi/m, \pi/m] } $ 
\qquad     
and 
\qquad     
${ \partial_\pm \YY_0 \, :=  \Sph^3_{\yy=\pm\pi/m} \subset \partial \YY_0 }$,   
\qquad     
so that 
\qquad     
$\, \rot_{2\pi/m}^\CC \partial_- \YY_0 = \partial_+ \YY_0$.     
\end{notation} 

\begin{lemma}[Restricting to $\YY_0$] 
\label{LEmu}   
$\phantom{kkk}$ 
Suppose  \       $U\subset\Mbreve$ \ is invariant under \  $ \rot_{2\pi/m}^\CC$ \  and   \  $\refl_{\Sigma^0}$. \  
The following are equivalent then \ $\forall \mi\in\Z$  \      
and \  \    $\forall f\in\left. W_{+} \right|_{ U \cap \YY_0}  $  \     
(recall \ref{D:eig-res} and \ref{NYY}). 
\begin{enumerate}[label=\emph{(\roman*)}]
\item 
\label{LEmu1}  
$ f\in\left. W_{\mi+} \right|_{ U \cap \YY_0}  $. 
\item 
\label{LEmu2}  
There is a function        \quad $f_\partial : U\cap\partial_+\YY_0 \to \R $  \quad such that 
\begin{equation} 
\label{EEmu} 
\begin{aligned} 
f \, = \, \, &  e^{-\mi\pi \ii/m}  f_\partial \circ\rot_{2\pi/m}^\CC \quad  & \text { on } \quad & U\cap\partial_-\YY_0,
\\
f \, = \, \, &  e^{\mi\pi \ii/m} \, \,  f_\partial \quad  & \text { on } \quad & U\cap\partial_+\YY_0 .  
\end{aligned} 
\end{equation} 
\item 
\label{LEmu3}  
$\Im f \, =  \, \pm \tan      ({\mi\pi /m})  \, \Re f \quad \text{ on } \quad U\cap\partial_\pm\YY_0 $.      
\end{enumerate} 

Moreover the restriction map 
\quad  
$\Rcal_{U,\mi} : 
\left. W_{\mi+} \right|_{ U }  
\to 
\left. W_{\mi+} \right|_{ U \cap \YY_0}  , 
$
\quad  
defined by 
\quad  
$f \to \left. f \right|_{ U \cap \YY_0} $,  
\quad  
is a linear isomorphism whose inverse we will denote by  
\quad  
$\Ecal_{U,\mi} $.  
\end{lemma} 

\begin{proof} 
The equivalence of \ref{LEmu1} and \ref{LEmu2} follows from the periodicity in \ref{Lderot}\ref{Lderot4c} and elementary calculations. 
The equivalence of \ref{LEmu2} and \ref{LEmu3} follows easily from the definitions and elementary calculations. 
The proof is then completed by explicitly constructing 
$\Ecal_{U,\mi} $ using the periodicity and then verifying compatability. 
\end{proof}

\subsection*{Jacobi fields and normal componenents} 
$\vspace{.126cm}$ 
\nopagebreak

The approach in \cite{LindexI} was based on the observation that 
each summand in the decomposition 
of $V$ in \ref{EIVpm} 
contains exactly one Jacobi field (up to scaling).  
In the next lemma we record a more detailed version of this information, and also the corresponding information on the components of the unit normal, 
which are well known to be eigenfunctions of eigenvalue $-2$ for the Jacobi operator $\Lcal_{\Mbreve}$ (see \ref{ELjacobi}).  

\begin{lemma}[Jacobi fields and the normal]                     
\label{LJnu} 
The following hold for $\Mbreve=\ML$ or $\Mbreve=\Mbreve[m]$ a minimal surface doubling of $\Sigma^0$ symmetric under $\GrpSigma$ (see \ref{Egroup}).  
\begin{equation*} 
\begin{gathered} 
J_{ \CCperp } = \, J^{ \CC } \in V^{++}_{0-}, \qquad J_{ \CC } = \, J^{ \CCperp  } \in V^{--}_{0+}, 
\\
J_{ \cC_{\pi/2}^{\pi/2} } = \, J^{ \cC_0^0 } \in V^{+-}_{1+}, \qquad J_{ \cC_0^{\pi/2} } = \, J^{ \cC_{\pi/2}^0 } \in V^{+-}_{1-}, 
\\ 
J_{ \cC_{\pi/2}^0 } = \, J^{ \cC_0^{\pi/2} } \in V^{-+}_{1+}, \qquad J_{ \cC_0^0 } = \, J^{ \cC_{\pi/2}^{\pi/2} } \in V^{-+}_{1-}, 
\\ 
\nu \cdot \pp^0 \in V^{+-}_{0+}, \quad 
\nu \cdot \pp^{\pi/2} \in V^{-+}_{0+},  \quad 
 \nu \cdot \pp_0 \in V^{++}_{1+}, \quad \nu \cdot \pp_{\pi/2} \in V^{++}_{1-}.  
\end{gathered} 
\end{equation*} 
\end{lemma} 

\begin{proof} 
It is straightforward to verify all these statements.  
\end{proof}

Recall that following \cite[Definition 1.1]{gLD} embedded doublings $\Mbreve$ of $\Sigma^0$ are smooth surfaces which are unions of the graphs in $\Sph^3$ of $\pm\ubreve:\Sigmabreve\to\R$ 
for some $ \ubreve\in C^0(\Sigmabreve)\cap C^\infty(\,\inter(\Sigmabreve)\,)$, 
where $\Sigmabreve\subset\Sigma^0$ with smooth boundary $\partial\Sigmabreve$. 
In this article we only consider doublings respecting the reflectional symmetry $\refl_{\Sigma_0}$.   
This implies that $\ubreve=0$ on $\partial\Sigmabreve$ and embeddedness implies $\ubreve>0$ on $ \inter(\Sigmabreve)   $.  
A straightforward explicit calculation shows that the unit normal $\nu:\Sigmabreve\to \Sph^3\subset\R^4$ to the graph of $\ubreve$ 
satisfies the following (recall \ref{D:x1} and \ref{D:xx}). 

\begin{equation} 
\label{Enu} 
\begin{gathered} 
\nu\cdot \pp^0=\frac{-\cos{\xx}\sin{\xx}\cos{\ubreve}\sin{\ubreve}-\cos^2{\xx}\,\ubreve_\xx}{\sqrt{\cos^2{\xx}(\cos^2{\ubreve}+\ubreve^2_\xx)+\ubreve^2_\yy}} \in V^{+-}_{0+},  
\\     
\nu\cdot \pp^{\pi/2}=\frac{\cos{\xx}\cos^2{\ubreve}}{\sqrt{\cos^2{\xx}(\cos^2{\ubreve}+\ubreve^2_\xx)+\ubreve^2_\yy}} \in V^{-+}_{0+}.   
\end{gathered} 
\end{equation} 

\begin{multline} 
\label{Enucom} 
\nu\cdot \pp_0 + \ii \, \nu\cdot \pp_{\pi/2} \\ 
= \frac{-\cos^2{\xx}\cos{\ubreve}\sin{\ubreve}+\cos{\xx}\sin{\xx} \,\ubreve_\xx - \ii\,\ubreve_\yy}{\sqrt{\cos^2{\xx}(\cos^2{\ubreve}+\ubreve^2_\xx)+\ubreve^2_\yy}} \, e^{\yy\ii} \, \in \, W^{++}_{1+}, 
\end{multline} 

Combining with \eqref{E:K} and \eqref{E:Kcom} we conclude that on the graph of $\ubreve$ we have 

\begin{equation} 
\label{EcCJ} 
\begin{aligned}
J_{\CCperp} &=J^\CC =\frac{-\cos{\xx}\cos{\ubreve} \,\, \ubreve_\yy}{\sqrt{\cos^2{\xx}(\cos^2{\ubreve}+\ubreve^2_\xx)+\ubreve^2_\yy}} \, \in \, V^{++}_{0-}, 
\\
J_\CC &=J^{\CCperp} =\frac{\cos{\xx}\sin{\xx}\cos{\ubreve}+\cos^2{\xx}\sin{\ubreve} \,\, \ubreve_\xx}{\sqrt{\cos^2{\xx}(\cos^2{\ubreve}+\ubreve^2_\xx)+\ubreve^2_\yy}} 
\, \in \, V^{--}_{0+}, 
\\
\end{aligned}
\end{equation} 

\begin{multline} 
\label{EJcom1} 
J_{\cC^{\pi/2}_{\pi/2}} + \ii \, J_{\cC^{\pi/2}_{0}} = 
J^{\cC^{0}_{0}} + \ii \, J^{\cC^{0}_{\pi/2}} 
\\ 
= \,  
\frac{-\cos{\xx}\cos{\ubreve} \,\, \ubreve_\xx + \ii \, \sin{\xx}\cos{\ubreve} \,\, \ubreve_\yy}{\sqrt{\cos^2{\xx}(\cos^2{\ubreve}+\ubreve^2_\xx)+\ubreve^2_\yy}} \, e^{\yy\ii}  
\, \in \, W^{+-}_{1+}, 
\end{multline} 
\begin{multline} 
\label{EJcom2} 
J_{\cC^{0}_{\pi/2}} - \ii \, J_{\cC^{0}_{0}} = J^{\cC^{\pi/2}_{0}} - \ii \, J^{\cC^{\pi/2}_{\pi/2}} = 
\\
= \,  
\frac{\cos^2{\xx}\cos{\ubreve}-\cos{\xx}\sin{\xx}\sin{\ubreve} \,\, \ubreve_\xx + \ii \, \sin{\ubreve} \,\, \ubreve_\yy}{\sqrt{\cos^2{\xx}(\cos^2{\ubreve}+\ubreve^2_\xx)+\ubreve^2_\yy}} \, e^{\yy\ii} 
\, \in \, W^{-+}_{1+}. 
\end{multline}

\bigbreak  
\section{The equator-poles minimal doublings $\xibreveomo$}  
\label{S:xi1m1}

\subsection*{Basic structure of $\xibreveomo$}  
$\vspace{.062cm}$ 
\nopagebreak

In this section we restrict our attention to the doublings $\xibreveomo$  
discussed already in the introduction. 
We first define the configuration $\taubold : L \rightarrow \R_+$ of a related LD solution as follows. 
The singular set $L$ is as in 
\ref{ELmer} and \ref{EL} with 
\begin{equation} 
\label{ELp} 
\begin{gathered} 
\Lpar= \CC\cup\{\pp^0,-\pp^0\} 
\qquad \text{and} \qquad  
L=L[m]:= 
\Leq\cup \Lpol  , 
\\
\text{ where } \qquad 
\Leq := \{\ptti: i\in 2\Z \} 
\qquad \text{and} \qquad  
\Lpol   := 
\{\pp^0,-\pp^0\}.  
\end{gathered} 
\end{equation} 
The values of $\taubold$ are 
$\tau_0$ on $L_0$ and $\tau_2$ on $L_2$ where 
\begin{equation}
\label{Etau02}
\begin{gathered} 
\tau_0 = \tau_0[m]
:= m^{-3/4} \, e^{\zeta_0} \, e^{-{\sqrt{m/2}}},
\\     
\tau_2 = \taubreve =  \tau_2[m]
:= \tau_0\,\left(\zeta_2-\frac14\log m+\sqrt{\frac{m}2\,}\,\,\right), 
\end{gathered} 
\end{equation}
where $\tau_0,\tau_2, \zeta_0,\zeta_2$ are constants which depend only on $m$ 
with $|\zeta_0|$ and $|\zeta_2|$ uniformly bounded independently of $m$ 
\cite[(6.23)]{SdI}. 
The periodic harmonic Green's function \cite{Ammari} defined by         
\begin{equation} 
\label{Gperiodic} 
\begin{gathered} 
G_\infty : \R^2 \, \, {{ \mathlarger{\mathlarger{\setminus}}  }}  \, \big\{(2k\pi , 0):k\in \Z \big\} \rightarrow \R, 
\\       
G_\infty (\thetatilde, \ssshat) :=  \frac{1}{2} \log \left( \sin^2 \frac{\thetatilde}{2} + \sinh^2 \frac{\ssshat}{2}\right), 
\end{gathered} 
\end{equation} 
is also relevant \cite[Lemma 9.26]{gLD}.  

\begin{theorem}[Existence and structure of $\xibreveomo$ \cite{SdI}] 
\label{Tep} 
\quad 
For each large 
\linebreak 
enough $m\in\N$ there exists a smooth embedded closed and connected minimal surface $\Mbreve = \xibreveomo$ of genus $m+1$,  
which is a side-symmetric minimal surface doubling over $\Sigma=\Sz$ in $\Sph^3$ of $(\taubreve, \alpha, \GrpSigma )$-gluing type  
and configuration $\taubold : L \rightarrow \R_+$ as in \ref{dgtSS},  
where  
$\alpha \in (0,  1/3 )$ is a small constant independent of $m$, 
$\GrpSigma = \GrpSigmam$ is as in \ref{Egroup},     
$\taubreve$ and $\taubold$ are as in \ref{ELp} and \ref{Etau02}, 
and furthermore the following hold,   
\quad 
where $C$ denotes absolute constants independent of $m$, 
\quad 
$\gamma := 3/2$, 
\quad 
$\zeta_0,\zeta_2,\mubreve_0,\mubreve_2$ are small constants depending only on $m$,  
\quad 
\begin{enumerate}[label=\emph{(\roman*)}]
\item 
\label{Tep1}   
Recall from \ref{Ddoubling} that 
$\xibreveomo$ 
is the union of the graphs over $\Sigmabreve$ of 
$\pm\ubreve \in C^0(\Sigmabreve) \cap C^\infty( \Sigmabreve \setminus\partial \Sigmabreve)$.   
\item 
\label{Tep0}   
For $j=0,2$ $\quad |\zeta_j|<C \quad$  and $\quad |\mubreve_j| < C \tau_j / \sqrt{m}. \quad $ 
\item 
\label{Tep10}   
On $\OL:= \Sz\setminus \disjun_{j=0,2} D_{L_j  }( \, \tau_j^\alpha/2 \, )$
we have  
$ \| \, \ubreve \, : \, C^{3} ( \OL , g) \, \| \, \le \, \tau_0^{1-4\alpha} \, \le \, \tau_0^{8/9}$.  
\item 
\label{Tep6}   
On $\Opar:=\Sz\setminus D^{\Sz,g}_{\Lpar} ( 9/m )$ we have 
$\ubreve= \phi + \phicheck + \ubreve_\osc$  
where 
$\phi, \phicheck, \ubreve_\osc\in C^\infty(\Opar )$ 
satisfy \ref{Tep7}-\ref{Tep9} below (recall \ref{Esssxx}).   
\item 
\label{Tep7}   
$\phi \, = \, \frac{m}2\, (\tau_0+\mubreve_0) \, \sin|\xx| \, - \, (\tau_2+\mubreve_2) \, \phie \     $ 
(recall \ref{D:JS}), and $ \     \Lcal_{\Sz} \, \phi=0$.  
\item 
\label{Tep8}   
$\phicheck$ is constant on parallel circles of $\Sz$ and satisfies the estimate (recall \ref{Ncyl}) 
\qquad 
$\| \, \phicheck \, : \, C^3(\Opar,\chiK) \, \|  \le \tau_2^{1+\alpha/4}$.   
\item 
\label{Tep9}   
$\|\, \ubreve_\osc \,  : \, C^{3}( \Opar , 1 , m^2 \chiK , e^{- m \ssst /2 } ) \, \| \, \le  \, C \, \tau_0$ 
(recall \ref{D:norm-g}, \ref{Ecyl}, \ref{Ncyl} and \ref{Dcyl}).   
\item 
\label{Tep11}   
$\forall p\in L_0$ \quad we have for some absolute constant $a_\infty$ that (recall \ref{Gperiodic} and \ref{Nhom}) 
\quad $\ubreve \circ \exp_p^{\Sigma} \circ Y_p  \circ \homot_{1/m} \, \to \, G_\infty +a_\infty $ \quad in $C^k$ norm on any compact subset of \quad 
$ \R^2 \, \, {{ \mathlarger{\mathlarger{\setminus}}  }}  \, \big\{(2k\pi , 0):k\in \Z \big\} $, 
\quad where \quad          $Y_p\{y=z=0\} = T_p\CC$ \quad by construction.  
\end{enumerate}  
\end{theorem} 

\begin{proof} 
Proofs of \cite[5.23, 6.24]{SdI} and statements of \cite[6.5, 6.6]{SdI}.  
\end{proof} 

Each catenoidal bridge is centered (by the symmetries) at a point $p\in L$; 
we will call the bridge \emph{equatorial} when $p\in \CC$, and \emph{polar} when $p$ is either of $\pm\pp^0$.  
Note that each $p\in L_0= L\cap \CC$ is contained in exactly three mutually orthogonal great spheres in $\mathbf{\Sigma}$; 
these great spheres subdivide the bridge into eight congruent pieces. 
If $p\in\{\pm\pp^0\}$, then $p$ is contained in exactly $m+1$ great spheres in $\mathbf{\Sigma}$; 
these great spheres subdivide the bridge into $4m$ congruent pieces.

$\vspace{.072cm}$ 

\paragraph{\bf{Comparison of the Lawson surfaces {${{\ML}}$}  with  the minimal doublings $\xibreveomo$ }} 
$\phantom{ab}$ 
$\vspace{.072cm}$

%%%%%%%%%%% TESSELATIONS 

For visualization purposes we briefly compare now the equator-poles doubling $\xibreveomo$ with the Lawson surface $\ML$, 
where both are symmetric under  $\GrpSigma = \GrpSigmam$ (recall \ref{Egroup} and \ref{ECQ}).   
We recall first some notation    related to the Lawson construction. 
$\forall i,j \in \frac{1}{2}\Z$ we define a tetrahedron 
\begin{equation}
\label{Omij} 
\Om_i^j := \overline{\, \pp_{(i-1/2)\pi/m} \, \pp_{(i+1/2)\pi/m} \, \pp^{(j-1/2)\pi/2} \, \pp^{(j+1/2)\pi/2} \, },  
\end{equation} 
and we then consider the tesselations of $\Sph^3$ \cite[Definition 2.14]{KW:Luniqueness}  
\begin{equation}
\label{Etess} 
\begin{aligned} 
\Ombold:=&\{\Om_i^j\}_{i,j\in\Z}, \qquad & \qquad \Omboldu:=&\{\Om_i^j\}_{i,j\in\frac12+\Z}, 
\\ 
\Ombold_e:=&\{\Om_i^j\}_{i+j\in2\Z,i\in\Z},\qquad & \qquad \Ombold_o:=&\{\Om_i^j\}_{i+j\in2\Z+1,i\in\Z}. 
\end{aligned} 
\end{equation} 
Each of the tesselations $\Ombold$ and $\Omboldu$ consists of $8m$ tetrahedra.  
$\GrpSigma$ acts simply transitively on $\Omboldu$.  
Each of $\Ombold_e$ and $\Ombold_o$ has $4m$ tetrahedra and $\Ombold = \Ombold_e \sqcup \Ombold_o $.  
%%%%%%%%%%%%%%%% END OF TESSELATIONS 

Recall that 
\ $\ML\subset \bigcup\Ombold_e$ \ with \ $\ML\cap\Om$ \ a minimal disc whose boundary is a quadrilateral consisting of four edges of $\Om$ (for each $\Om\in \Ombold_e$),  
while \ $\ML\cap\Omu$ \ is a free boundary disc in $\Omu$ (for each $\Omu\in \Omboldu$) with boundary a quadrilateral 
whose sides are geodesic segments in the intrinsic geometry of $\ML$.  
On the other hand \ $\xibreveomo \cap\Omu$ \ is a free boundary disc in $\Omu$ (for each $\Omu\in \Omboldu$) 
with boundary a pentagon 
whose sides are geodesic segments in the intrinsic geometry of $\ML$.  
Two of the sides of the pentagon lie on $\Sigma^0$ with the rest on the other three faces of $\Omu$, one on each face. 
\ $\xibreveomo\cap\Omu$ \ intersects exactly two catenoidal bridges of $\xibreveomo$, one equatorial and one polar. 
It contains one of eight congruent pieces of the equatorial bridge and one of the $4m$ pieces of the polar bridge.  

For a more global picture we consider the quarter of the surfaces contained in \   $\Sph^{3++}$:       
\  $\ML^{++}$    \  (recall \ref{D:Om}) is topologically a disc which is graphical along \  $K_\CC=K^\CCperp$  \   
\cite[Lemmas 4.11 and 4.14]{LindexI},     
and     
\  $\xibreveomo^{++}$ \  is an annulus which is graphical over its image under $\PiSph$ (recall \ref{EPiSph}). 
The boundary \  $\partial \ML^{++}$    \  
is            a necklace of $2m$ congruent arcs alternatingly lying on $\Sigma^0$ and $\Sigma^{\pi/2}$. 
The boundary \  $\partial\xibreveomo^{++}$ \  has two connected components one of which is the waist of $\Kbreve_{\pp^0}$. 
The other one is a necklace of $2m$ arcs alternatingly lying on $\Sigma^0$ and $\Sigma^{\pi/2}$ 
which is similar to \ $\ML^{++}$, \     
the difference being however 
that the arcs on $\Sigma^0$ 
are much shorter than the ones on $\Sigma^{\pi/2}$. 
Note also that each of the arcs on $\Sigma^0$ is the half on one side of $\CC$ 
of the waist in $\Sigma^0$   of an equatorial catenoidal bridge.

\subsection*{Geometric features of $\xibreveomo$}  
\nopagebreak

\begin{lemma}[Maximum of $\phi$] 
\label{Lmaxphi} 
The maximum of $\phi$ is attained on 
$\CC^\parallel_{\xxmax}$ (recall \ref{N:xx}) 
where 
\quad
$ 
\xxmax \, = \, \frac\pi2 - (m/2)^{-1/4} \, + \,  O( m^{-1/2} \log m  \, ) . 
$ 
\quad
Moreover $\phi$ is increasing in $\xx$ for $\xx<\xxmax$ and decreasing for $\xx>\xxmax$.  
\end{lemma} 

\begin{proof} 
This follows by direct calculation from \ref{Tep}\ref{Tep7} and \ref{D:JS}. 
\end{proof} 

\begin{corollary}[Approximate nodal parallel circles]  
\label{CJnodal} 
Let  \ $\xx_\pm := \xxmax \pm \tau $   \   (recall \ref{Lmaxphi}) and   for future reference \ $\sss_\pm := \sss(\xx_\pm) $;            
\quad $J^{ \cC_0^0 } \in V^{+-}_{1+}$ \     (and similarly \quad $J^{ \cC_{\pi/2}^0 } \in V^{+-}_{1-} $)  \quad  
has a nodal (topological) circle \quad   $\Cnodal \subset \xiomo^{++}$  \quad  
such that \quad      $\Pi_\Sigma(\Cnodal)  \subset  \Sigma^0_{\xx\in (\xx_-,\xx_+)}$  \            
is a small graphical perturbation of \        $\CC^\parallel_{\xxmax}$  \        
in \    $\Sigma^0$.  
\end{corollary} 

\begin{proof} 
This follows easily using the implicit function theorem, \ref{Lmaxphi}, the explicit expression from \ref{EJcom1}, and the information from \ref{Tep}. 
\end{proof} 

\begin{lemma}[The sign of $\nu \cdot \pp^0 $ ] 
\label{Lnu0} 
$\nu \cdot \pp^0 \in V^{+-}_{0+}$ changes sign on $\xiomo^{++}$. 
\end{lemma} 

\begin{proof} 
Using \ref{Enu} we see that close to $\Sigma^0\cap\Sigma_0$ where $\ubreve$ is small but $|\ubreve_\xx|$ large and $\xx\ne 0,\pi/2$, 
the sign of the numerator is the sign of $\ubreve_\xx$ which is positive for small $\xx$ and negative for $\xx$ close to $\pi/2$. 
\end{proof}

\begin{lemma}[The sign of $\nu \cdot \pp^{\pi/2} $ ] 
\label{Lnup} 
$\nu \cdot \pp^{\pi/2} \in V^{-+}_{0+}$  
does not change sign on $\xiomo^{++}$. 

\end{lemma} 

\begin{proof} 
Follows clearly from \ref{Enu}. 
\end{proof}

\begin{lemma}[The sign of $J^{ \CCperp } $   ] 
\label{LJperp} 
\ 
$J_\CC = J^{ \CCperp } \in V^{--}_{0+}$  
does not change sign on $\xiomo^{++}$. 
\end{lemma} 

\begin{proof} 
By \ref{EcCJ} the sign of \ $J^{ \CCperp }$ \ is the sign of  \      $\cot \ubreve + \cot\xx \, \ubreve_\xx$ \ 
which is clearly $>0$ when \  $\ubreve_\xx  \ge -1/9$.   \   
The region \  $  \{ \, q \in \xiomo^{++} \, : \,  \ubreve_\xx \le -1/9 \, \}$ \  is contained in the very core of $\Kbr_{\pp^0}$ 
where the Killing field is close to $\pp^{\pi/4}$ which is normal to $\Sigma^0$ and tangential to $\Sph^3$ at $\pp^0$,  
so the sign does not change there as well. 
\end{proof} 

\begin{lemma}[The sign of $J^{ \cC^{\pi/2}_{\pi/2} } $   ] 
\label{LJpp} 
\ 
$J_{ \cC_0^0 } = \, J^{ \cC^{\pi/2}_{\pi/2} } \in V^{-+}_{1-}$   
does not change sign on $\xiomo^{++}\cap\Sph^3_+$. 
\end{lemma} 

\begin{proof} 
By \ref{EJcom2} the sign of $J_{ \cC_0^0 }$ 
does not change in the region 
\  $  \{ \, q \in \xiomo^{++} \cap\Sph^3_+ \, : \, |\ubreve_\xx| \le 1/9, \,  |\ubreve_\yy| \le 1/9, \, |\xx| \le 8/9  \, \, \}$. \  
The complement of this region in \    $\xiomo^{++} \cap\Sph^3_+$ \      is the disjoint union of a neighborhood of a polar bridge and quarters or halves of cores of equatorial bridges.  
On the cores of the half catenoidal bridges, $J_{ \cC_0^0 }$ is dominated by slidings and there can no change of sign. 
Because of the odd symmetries with respect to $\Sigma_0$ and $\Sigma^0$ there cannot be change of sign on the remaining components either by using that 
the Dirichlet eigenvalue of such a region for $\Lcalh$ is clearly $>3$. 
\end{proof}

%\pagebreak[4]
\section{Results for the equator-poles doublings $\xibreveomo$}  
\label{S:eq-p} 

\subsection*{The approxinate spectrum on the equator-poles doublings $\xibreveomo$ } 
$\vspace{.126cm}$ 
\nopagebreak

Recall that we have fixed $\xibreveomo$ to be a minimal surface doubling of $\Sigma^0$ with $L$ as in \ref{ELp}. 
We have the following.

\begin{lemma}[Approximate spectrum on $\SbreveS \subset \xibreveomo$]       
\label{LapprS}
\, 
For $m\in\N$ large enough in terms of given $\epsilon\in (0,1/2)$ 
we have the following 
(recall \ref{D:eig-res} and      
\ref{DSbr-norm}).    

\smallbreak 

\begin{enumerate}[label=\emph{(\roman*)}]
%%%%%%%%%%%%%% SbreveS Neumann  ++  
\item 
\label{LapprS1}
$\lambda_1( \, V^+ ,\LcalhN , \SbreveS ) = -2$   
\quad 
and 
\quad 
$4-\epsilon \, < \, \lambda_2( \, V^{++}_{0+} , \LcalhN , \SbreveS  )$.  
\item 
\label{LapprS2}
$\lambda_1^\C( \, W^{++}_1 , \LcalhN , \SbreveS  )  
= \lambda_1( \, V^{++}_{1\pm} , \LcalhN , \SbreveS  ) \in (-\epsilon, \epsilon) $    
\\     $\phantom{kkk}$ \hfill 
and \quad 
$4-\epsilon \, < \, \lambda_2^\C( \, W^{++}_1 , \LcalhN , \SbreveS  ) 
= \lambda_2( \, V^{++}_{1\pm} , \LcalhN , \SbreveS  ) $.    
\\ 
Moreover there is an eigenfunction 
\      
${f'}^{N+} \in \left. W^{++}_{1+} \right|_{\SbreveS } $ 
\      
corresponding to 
the first eigenvalue 
satisfying 
\     $\| \, {f'}^{N+} -  \nu\cdot (\pp_{0} + \ii \pp_{\pi/2} )   \, \|_{ \epsilon ; \SbreveS } \, < \, \epsilon $. 
\item 
\label{LapprS3} 
$4-\epsilon \, < \, \lambda_4( \, V^{++} , \LcalhN , \SbreveS  )$.

\smallbreak \smallbreak \smallbreak 

%%%%%%%%%%%%%% SbreveS Neumann  +-  
\item 
\label{LapprS4}
$\lambda_1( \,   V^{+-}     , \LcalhN,   \SbreveS ) \in (-\epsilon, \epsilon) $    
\quad 
with 
\quad 
$f_1^{N-} \in \left. V^{+-}_{0+} \right|_{\Sbreve_p} $ 
\quad 
a corresponding eigenfunction 
satisfying   
\quad 
$\| \, f_1^{N-} -  \nu\cdot \pp^{0}  \, \|_{ \epsilon ; \SbreveS } \, < \, \epsilon $  (translation). 
\item 
\label{LapprS5} 
$4-\epsilon \, < \, \lambda_2( \, V^{+-} , \LcalhN , \SbreveS ) \, \le \, \lambda_1( \, V^{+-}_- , \LcalhN , \SbreveS ) $ 
\\ 
$\phantom{kkkkkkkkkkkkk}$ \hfill $ \, \le \, \lambda_1  ( \, W^{+-}_1 , \LcalhN , \SbreveS ) $.  

$\phantom{ab}$ 

\smallbreak \smallbreak 

%%%%%%%%%%%%%% SbreveS Dirichlet 
\item 
\label{LapprS6}
$ -\epsilon \, < \, \lambda_1( \, V^+  , \LcalhD , \SbreveS ) \, = \,  \lambda_1( \, V^{++}_{0+}   , \LcalhD , \SbreveS ) \,$    
\\  
$\phantom{k}$ \hfill 
$ < \, \lambda_1( \, V^{+-}_{0+}   , \LcalhD , \SbreveS )  = \,  \lambda_2( \, V^+ , \LcalhD , \SbreveS ) \, < \, \epsilon $. 
\\     
Moreover there are corresponding eigenfunctions 
\quad 
$f_1^D \in \left. V^{++}_{0+} \right|_{\SbreveS } $ 
\quad 
and 
\quad 
$f_2^D \in \left. V^{+-}_{0+} \right|_{\SbreveS } $ 
satisfying 
\\    
$\phantom{k}$ \hfill 
$\| \, f_1^D -  | \nu\cdot \pp^{0}  |  \, \|_{ \epsilon ; \SbreveS } \, < \, \epsilon $ 
\quad 
and   
\quad 
$\| \, f_2^D -  \nu\cdot \pp^{0}  \, \|_{ \epsilon ; \SbreveS } \, < \, \epsilon $. 
\item 
\label{LapprS7} 
$4-\epsilon \, < \, \lambda_2( \, V^{+\pm} , \LcalhD , \SbreveS ) $. 
\end{enumerate} 
\end{lemma} 

\begin{proof} 
Note that \ref{LapprS4}, \ref{LapprS6}, and  \ref{LapprS7} are not necessary for the proof of the final result \ref{Tmain}  
and so we do not prove them here; 
similarly we do not need and do not prove the norm estimates. 
By \ref{dgtSS}\ref{dgtSS1}\ref{dgtSS1b} we can ensure now that $\Lcalh$ on $\SbreveS$ is very close to \ $\Delta+2$ \ on \ $\Pi_\Sigma \SbreveS \subset \Sigma^0$.\  
It is easy to obtain $C^0$ bounds for the eigenfunctions by the usual arguments and then by standard arguments as for example in 
\cite[Appendix B]{kapouleas:1990} involving the use of logarithmic cutoff functions we approximate the eigenvalues with the corresponsing eigenvalues of 
\ $\Delta+2$ \ on \ $\Sigma$. 
Since the lowest eigenvalues of $\Delta$ on $\Sigma$ are $0$, $2$ and $6$ the proof follows. 
\end{proof}

\begin{lemma}[Approximate spectrum on $\xibreveomo$] 
\label{Lapp} 
The following hold for $m$ large enough in terms of given $\epsilon\in (0,1/2)$. 
\\ $\phantom{ab}$ 
\begin{enumerate}[label=\emph{(\roman*)}]
%%%%%%%%%%%%%%     V^{++} 
\item 
\label{Lapp1} 
$\lambda_1(V_{0+}^{++},\Lcalh)=-2  \, < \, \lambda_2(V_{0+}^{++},\Lcalh)<-2+\epsilon$, 
\quad
$\lambda_3(V_{0+}^{++},\Lcalh)<\epsilon$, 
\quad
$4-\epsilon < \lambda_4(V_{0+}^{++},\Lcalh) $. 
\item 
\label{Lapp2} 
$\lambda_1(V_{0-}^{++},\Lcalh)=0 $ 
\qquad
(corresponding to $J_{ \CCperp } = \, J^{ \CC } \in V^{++}_{0-}$), 
\\     
$\phantom{k}$ \hfill 
$4-\epsilon < \lambda_2(V_{0-}^{++},\Lcalh)$. 
\item 
\label{Lapp3} 
$\lambda_1^\C(W_{1}^{++},\Lcalh) = \lambda_1(V_{1\pm}^{++},\Lcalh) < -2+\epsilon $  
$< -\epsilon < $ 
$ \lambda_2^\C(W_{1}^{++},\Lcalh)  =  \lambda_2(V_{1\pm}^{++},\Lcalh) $ 
$ \le \lambda_3^\C(W_{1}^{++},\Lcalh) = \lambda_3(V_{1\pm}^{++},\Lcalh)< \epsilon $,   
\quad   
(\ note also 
\quad $\nu \cdot ( \pp_0 + \ii \pp_{\pi/2} )  \in W^{++}_{1+}$ ), 
\hfill 
$4-\epsilon < \lambda_5^\C(W_{1}^{++},\Lcalh) = \lambda_5(V_{1\pm}^{++},\Lcalh) $. 
\item 
\label{Lapp4} 
$\forall \mu\in\mtwo$ we have 
\qquad
$\lambda_1^\C(W_{\mu}^{++},\Lcalh) = \lambda_1(V_{\mu\pm}^{++},\Lcalh) < -2+ \epsilon < -\epsilon <
\lambda_2^\C(W_{\mu}^{++},\Lcalh) = \lambda_2(V_{\mu\pm}^{++},\Lcalh) < \epsilon$,   \qquad  
$4-\epsilon < \lambda_3^\C(W_{\mu}^{++},\Lcalh) = \lambda_3(V_{\mu\pm}^{++},\Lcalh) $. 
\item 
\label{Lapp5} 
If $m\in2\Z$, then for $\mu=m/2$, we have 
\\   
$\lambda_1(V_{\mu+}^{++},\Lcalh) < -2+ \epsilon$, 
\qquad 
$4-\epsilon < \lambda_2(V_{\mu+}^{++},\Lcalh) $,  
\\    
and \qquad 
$\lambda_1(V_{\mu-}^{++},\Lcalh) < \epsilon$, 
\qquad
$4-\epsilon < \lambda_2(V_{\mu-}^{++},\Lcalh)$.  
\ 
\\ 
\ 
%%%%%%%%%%%%%%     V^{+-} 
\item 
\label{Lapp7} 
$\lambda_1(V_{0+}^{+-},\Lcalh)<-2+\epsilon$, 
\qquad
$ \lambda_2(V_{0+}^{+-},\Lcalh) \in  \,(-\epsilon,0)$, 
\\                        
$\lambda_3(V_{0+}^{+-},\Lcalh) \in  \,(-\epsilon,\epsilon)$, 
\qquad 
$4-\epsilon < \lambda_4(V_{0+}^{+-},\Lcalh)$, 
\qquad 
$\nu \cdot \pp^0 \in V^{+-}_{0+}$. 
\item 
\label{Lapp8} 
$4-\epsilon < \lambda_1(V_{0-}^{+-},\Lcalh) $.    
\item 
\label{Lapp9} 
$ -\epsilon < \lambda_1^\C(W_{1}^{+-},\Lcalh) = \lambda_1(V_{1\pm}^{+-},\Lcalh)   \, < $ 
\\ 
$\phantom{kkkkk}$ \hfill 
$< \, \lambda_2^\C(W_{1}^{+-},\Lcalh)  = \lambda_2(V_{1\pm}^{+-},\Lcalh) = 0$ 
\\     
($J^{ \cC^0_0   } +  \ii J^{ \cC^0_{\pi/2} } \in W^{+-}_{1+  }$),  
\qquad
$4-\epsilon < \lambda_3^\C(W_{1}^{+-},\Lcalh) = \lambda_3(V_{1\pm}^{+-},\Lcalh) $.  
\qquad
\item 
\label{Lapp10} 
$\forall \mu\in\mtwo$ we have 
\qquad
$-\epsilon < \lambda_1^\C(W_{\mu}^{+-},\Lcalh) = \lambda_1(V_{\mu\pm}^{+-},\Lcalh) < \epsilon$, 
\\ $\phantom{kkkkk}$ \hfill and \qquad 
$4-\epsilon < \lambda_2^\C(W_{\mu}^{+-},\Lcalh) = \lambda_2(V_{\mu\pm}^{+-},\Lcalh) $. 
\item 
\label{Lapp11} 
If $m\in2\Z$, then for $\mu=m/2$, we have 
\qquad
$-\epsilon < \lambda_1(V_{\mu+}^{+-},\Lcalh) < \epsilon$, 
\qquad
\\ 
$\phantom{kkkkk}$ \hfill $4-\epsilon < \lambda_2(V_{\mu+}^{+-},\Lcalh) $, 
\qquad and \qquad 
$4-\epsilon < \lambda_1(V_{\mu-}^{+-},\Lcalh) $. 
\\ 
$\phantom{kkk}$ 

%%%%%%%%%%%%%%     V^{-+} 
\item 
\label{Lapp-+} 
$\nu \cdot \pp^{\pi/2} \in V^{-+}_{0+}$,  
\quad 
note also 
\quad 
$J^{ \cC^{\pi/2}_0   } +  \ii J^{ \cC^{\pi/2}_{\pi/2} } \in W^{-+}_{1+}$.  
\\ 
$\phantom{kkk}$ 

%%%%%%%%%%%%%%     V^{--} 
\item 
\label{Lapp--} 
$ J^{ \CC^\perp  } \in V^{--}_{0+}$,  
\quad 
$\lambda_1(V^{--},\Lcalh) = 0   < \lambda_2(V^{--},\Lcalh) $. 
\end{enumerate} 
\end{lemma} 

\begin{proof} 
Most inequalities follow by combining \ref{P:appr} with \ref{LapprS}. 
The remainining follow by elementary arguments based on the nodal lines of the Jacobi fields with the exception of the proof of \ref{Lapp3} which requires a careful analysis 
of one of the eigenfunctions. 
We examine now each case in detail. 
\\ 
$\phantom{kkk}$ 

\ref{Lapp1} 
By the symmetries $V_{0+\con}^{++}$ is two-dimensional (one constant at the pole and one at the equator) and $V_{0+\Sl}^{++}$ is trivial. 
In \ref{P3rd} we prove later by a gluing construction that there is an eigenvalue in $(0,\epsilon)$, 
but here we can use \ref{LapprS}\ref{LapprS6}  (although not required for the final result \ref{Tmain}) to establish there is a third eigenvalue $<\epsilon$.  
By \ref{LapprS}\ref{LapprS1} the fourth eigenavalue $> 4-\epsilon$. 

\ref{Lapp2} 
By the symmetries $V_{0-\con}^{++}$ is trivial and $V_{0-\Sl}^{++}$ is one-dimensional (a sliding along $\cC$). 
By \ref{LapprS}\ref{LapprS3}  there is no other eigenvalue $\le 4-\epsilon$ and by \ref{LJnu} we conclude \ref{Lapp2}.    

\ref{Lapp3} 
By the symmetries $W_{1\con}^{++}$ is spanned by $f$ with $f(\pp_0)=1$ and $f(\pp^0)=0$ (a constant at $\CC$), 
and $W_{1\Sl}^{++}$ is spanned by $\vecf$ 
with $\vecf(\pp_0)= \partial_{x^1}$ and $\vecf(\pp^0)=0$ (sliding along $\CC$), 
and $\vecfp$ 
with $\vecfp(\pp_0)= 0$ and $\vecfp(\pp^0)=\partial_{x^1}+i\partial_{x^2}$ (sliding at the pole). 
Hence $\dim_\C ( W_{1\con}^{++} )=1$   
and $\dim_\C ( W_{1\Sl}^{++} )=2$.  
By    \ref{P:appr} and \ref{LapprS}\ref{LapprS2} we conclude there is one eigenvalue in \   $(-2,-2+\epsilon)$,  \   
no eigenvalues in \      $[-2+\epsilon, -\epsilon]$, \   and either two or three (in \ref{Pcos4} later we prove only two) in \      $(  -\epsilon, -\epsilon)$. \   

\ref{Lapp4} We have one equatorial constant and one sliding along $\CC$. 

\ref{Lapp5} 
$V_{\mu+\Sl }^{++}$ and $V_{\mu-\con}^{++}$ are trivial   
and 
we have one equatorial constant 
in $V_{\mu+\con}^{++}$ and one sliding along $\CC$ in $V_{\mu-\Sl}^{++}$.  
\\ 
$\phantom{kkk}$ 

\ref{Lapp7} 
$V_{0+\con}^{+-}$ is clearly one-dimensional with 
a single constant at the pole, and therefore 
\qquad
$\lambda_1(V_{0+}^{+-},\Lcalh)<-2+\epsilon$. 
\qquad
$V_{0+\Sl}^{+-}$ is clearly one-dimensional also with 
one sliding to the pole at $\CC$.  
Since $\nu \cdot \pp^0 \in W^{+-}_{0+}$ changes sign by \ref{Lnu0}  
we have at least two negative eigenvalues. 
By \ref{LapprS}\ref{LapprS5} we have  
\qquad 
$4-\epsilon < \lambda_4(V_{0+}^{+-},\Lcalh)$. 
\qquad 
The bounds for the third eigenvalue follow 
from \ref{LapprS}\ref{LapprS4},\ref{LapprS6}, 
or alternatively 
from \ref{P3rd} which is proved below independently.   

\ref{Lapp8} 
\quad $ 0 <  \lambda_1(V_{0-}^{+-},\Lcalh) $    \quad  
follows by comparison with \ref{Lapp2} and this is enough for the proof of the final result \ref{Tmain}. 
The bound given here follows from 
the clear (by the symmetries) triviality of $V_{0-\con}^{+-}$ and $V_{0-\Sl}^{+-}$ and 
\ref{LapprS}\ref{LapprS6},\ref{LapprS7}.   

\ref{Lapp9} 
By the symmetries $W_{1\con}^{+-}$ is trivial 
and $W_{1\Sl}^{++}$ is spanned by $\vecfpp$ 
with $\vecfpp(\pp_0)= \partial_{x^3}$ and $\vecfpp(\pp^0)=0$ (sliding at $\CC$ to the pole), 
and $\vecfp$ 
with $\vecfp(\pp_0)= 0$ and $\vecfp(\pp^0)=\partial_{x^1}+i\partial_{x^2}$ (sliding at the pole). 
Hence $\dim_\C ( W_{1\con}^{+-} )=0$   
and $\dim_\C ( W_{1\Sl}^{+-} )=2$.  
By 
\ref{LapprS}\ref{LapprS5}    
we then have two eigenvalues in $(-\epsilon, \epsilon)$ with the third $>4-\epsilon$. 
The proof is completed by the existence of a Jacobi field which is not the ground state by \ref{CJnodal}. 

\ref{Lapp10} We only have one sliding to the pole at $\CC$.  

\ref{Lapp11} We only have one sliding to the pole at $\CC$ in $V_{\mu+\Sl}^{+-}$ while $V_{\mu+\con}^{+-}$, 
$V_{\mu-\Sl}^{+-}$, and $V_{\mu-\con}^{+-}$ are all trivial. 
\\ 
$\phantom{kkk}$

\ref{Lapp-+} Just recall \ref{LJnu}. 

\ref{Lapp--} 
Follows from \ref{LJperp}. 
\end{proof}

\bigbreak 
\subsection*{Index and nullity in $V^+_{0+}$ of the equator-poles doublings $\xibreveomo$ } 
\nopagebreak

\begin{lemma}[Approximate ground states]   
\label{L:3e0} 
There are $\sbar_0,\sbar_2 >0 $ so that the overlapping regions  
\ $\Sbreve[L_0,\sroot]$, \ $\Sbreve[\Sigma ,\sbar_0, \sbar_2]$, \ and \ $\Sbreve[L_2,\sroot]$,   
satisfy the following (recall \ref{DSMbr}, \ref{D:eig-res}, and \ref{Lphie}).         
\begin{enumerate}[label=\emph{(\roman*)}]
\item 
\label{L:3e01} 
$\left| \lambda_1(V_{0+}^{++},{\LcalhD} , \Sbreve[L_j,\sroot ] ) \right | < \tau_j $ for $j=0,2$. 
\item 
\label{L:3e00} 
$\sbar_0,\sbar_2 \, \in \, ( \, \sroot/2 \, , \, 3 \, \sroot/5 \, )    $. 
\item 
\label{L:3e02} 
$\left| \lambda_1(V_{0+}^{++},{\LcalhD} , \Sbreve[\Sigma ,\sbar_0, \sbar_2] ) \right | < \taubreve^{2\alpha} $. 
\end{enumerate} 
\end{lemma} 

\begin{proof} 
The proof is by construcing test functions which aproximately satisfy the Jacobi equation with Dirichlet condition at each boundary. 
The test function we use on $\Sbreve[L_j,\sroot ]$ is simply $\phie$ (recall \ref{Lphie}); \ref{L:3e01} follows then from \ref{LKh}\ref{LKh5}. 

Our construction of a test function $f$ for \ref{L:3e02} is motivated by the observation that the graph of $\ubreve$ is minimal and so is the base surface $\Sigma=\Sigma^0$, 
hence  \    $\Lcal_\Sigma \ubreve$ \    equals a quadratic expression with uniform coefficients in $\ubreve$ and its first and second derivatives. 
\quad 
This implies by \ref{dgtSS}\ref{dgtSS1}\ref{dgtSS1b} that 
\\     $\| \, \Lcal_\Sigma \ubreve \, : \, C^1(\SigmaU, g_\Sigma)\| \, <   \, \taubreve^{15/9}$, \qquad 
which implies by \ref{dgtSS}\ref{dgtSS1}\ref{dgtSS1b} and \ref{LSh}   
\qquad 
\begin{equation} 
\label{Eugr}
\| \, \Lcal_\Mbreve \ubreve \, : \, C^1( \, \Mbreve_{\gr} \, , \, g_\Mbreve \, )\, \| \, <   \, 2\taubreve^{15/9}.   
\end{equation} 

To construct now the test function $f$ we first 
define \ $f'\in C^\infty(\Mbreve_\pm)$ \  by 
(recall \ref{D:JS},  \ref{n:Mpm} and \ref{EPsibold})   
\begin{equation} 
\label{Efp}   
\begin{gathered} 
\fp := \ubreve \quad   \text{ on }   M_\pm \setminus \disjun_{p \in L} \Kbreve_{p \, , \, \rho \le 2\tau_p^{5\alpha} \, },  
\\ 
\fp  := \Psibold\left [ \tau_p^{5\alpha} , 2\tau_p^{5\alpha} ; \rho \, \right]( \, \phitp \, , \, \ubreveM \, )  = \ubreve + \psihat_3  \, ( \phiext - \ubreve )  
\\  
\quad   \text{ on }   
\     \Kbr_p \cap\Mbreve_\pm   \quad (\forall p\in L),  
\qquad 
\text{ where } 
\phiext := \lambda_\ev\phie + \lambda_\od\phio, 
\end{gathered} 
\end{equation} 
where 
$\psihat_3$ is a smooth cutoff functions on $\Mbreve$ 
supported    on \ $\disjun_{p\in L} \, \Kbreve_p $, \    
with detailed definition as in           \ref{EPsibold} 
so that  
\  $\forall p\in L$ \  
\ $\psihat_3 = 1$ \  on  \  $\{ q\in \Kbreve_p \, : \, \rho(q) \le  \tau_p^{5\alpha} \}$,  \     
\  $\psihat_3=0$ \   
on  \  $\{ q\in \Kbreve_p \, : \, \rho(q) \ge 2 \tau_p^{5\alpha} \}$,   \   
and the coefficients \ $\lambda_\ev, \lambda_\od$ \ are determined by the initial conditions (recall \ref{Ecyl}) 
\begin{equation} 
\label{Ephiext}
\qquad 
\phiext= 
\ubreve_\avg 
\quad 
\text{and} 
\quad 
\tfrac{\partial}{\partial\sss} \phiext= 
\tfrac{\partial}{\partial\sss} \ubreve_\avg 
= 
\left( \tfrac{\partial \ubreve}{\partial\sss} \right)_\avg 
\quad \text{at} \     
\rho=\tau_p^{5\alpha}.    
\quad 
\end{equation}

We prove next that there are unique \     $\sbar_j>0$ ($j=0,2$) \     such that \     $\phiext(\sbar_j) =0$  
\     when \     $p\in L_j$. \     
By \ref{Etau02} we have 
\begin{equation} 
\label{Elog}
|\log\tau_p| \  \Sim^{1+ m^{-1/3}} \ \sqrt{m/2}. 
\end{equation} 
We have then by \ref{Ecatenoid} that at \quad $\rho=\tau_p^{5\alpha}$ \quad     
\quad 
\begin{equation} 
\label{Ez}
\begin{gathered} 
\sss \, = \, (5\alpha-1) \log \tau_p +O(1) \quad  \Sim^{1+ m^{-1/3}} \quad (1-5\alpha) \sqrt{m/2}, 
\\
z    \, = \, (5\alpha-1) \tau_p \log \tau_p +O(\tau_p) \quad  \Sim^{1+ m^{-1/3}} \quad (1-5\alpha) \sqrt{m/2} \,  \tau_p, 
\end{gathered} 
\end{equation} 
and by \ref{dgtSS}\ref{dgtSS2}\ref{dgtSS2f} at \quad $\rho=\tau_p^{5\alpha}$ \quad     
$$ 
\left| \frac{\ubreve  -z}z \right | 
\, \le \, 
\taubreve^{(5\gamma+1/6) \alpha}, 
\qquad 
\frac1z \left| \frac{\partial\ubreve  }{\partial\sss} - \frac{\partial z }{\partial\sss} \right | 
\, \le \, 
\taubreve^{(5\gamma+1/6) \alpha}, 
$$ 
and hence by \ref{Ephiext} at \quad $\rho=\tau_p^{5\alpha}$ \quad     
\begin{equation} 
\label{E:phiext}   
\left| \frac1{\phiext}  \frac{\partial {\phiext} }{\partial\sss} 
-
\frac1{   z   }  \frac{\partial {   z   } }{\partial\sss} \right| 
\, \le \, 
\taubreve^{5\gamma\alpha}. 
\end{equation} 
A calculation based on \ref{D:JS} and \ref{Ecatenoid} gives for large $\sss$ 
$$
\frac1\phie\frac{\partial\phie}{\partial\sss} - \frac1 z \frac{\partial z }{\partial\sss} 
\, = \, \frac{\tanh \sss + \sss \sech^2 \sss}{\sss \tanh\sss -1} - \frac1\sss 
\, = \, \frac{\sech^2\sss + 1/\sss^2}{\tanh \sss -1/\sss} 
\  \Sim^{1+2/\sss} \  \frac1{\sss^2}. 
$$
We conclude that 
\begin{equation} 
\begin{gathered} 
\label{E:dflux} 
\left. \left( \frac1\phie\frac{\partial\phie}{\partial\sss} - \frac1{\phiext}  \frac{\partial {\phiext} }{\partial\sss} \right) \right|_{\rho=\tau_p^{5\alpha}} 
\  \Sim^{1+3/\sss } \  \left. \frac1{\sss^2} \right|_{\rho=\tau_p^{5\alpha}} 
\end{gathered} 
\end{equation} 
with the value of $\sss$ given by \ref{Ez}. 
An easy calculation shows that for $|\lambda|<9$ and large $\sss$ we have 
$$
\frac{\partial^2}{\partial \lambda \, \partial \sss   } \log(\phie+\lambda\phio) \, = \, - \frac{1+\csch^2\sss}{(\sss+\lambda-\coth\sss)^2} 
\  \Sim^{1+2/\sss} \  - \frac1{\sss^2},  
$$
and therefore by \ref{E:dflux} 
\qquad 
$
{\lambda_\od} \  \Sim^{1+ 6\sqrt{2/m} } \   {\lambda_\ev} .              
$
\qquad 
This implies that $\phiext$ is strictly increasing in  $\sss>0$. 
Since \ $\phiext(0)<0$ we conclude that there is a unique root \ $\sbar_j>0$ \ 
such that \ $\phiext(\sbar_j)=0$. \ 
We define    
\     $\sbar_j>0$ ($j=0,2$) \     
to be this unique root 
and since 
\quad $ \lambda= \lambda_\od / \lambda_\ev \Sim^{1+6\sqrt{2/m}} 1$   \quad      
we conclude by \ref{Lphie}\ref{Lphie4} that 
\ref{L:3e00} holds.  
To simplify the notation in the remaining of this proof we define then \     
\begin{equation} 
\label{EnA}
\begin{gathered} 
\Sbreve':=\Sbreve[\Sigma ,\sbar_0, \sbar_2],  
\\ 
\forall p\in L \quad  \Kbreve'_p := \Sbreve'\cap \Kbreve_p,  \quad 
\   \Abreve'_p := \{ q\in \Kbreve_p \, : \, \rho(q) \in [ \,  \tau_p^{\alpha}/2 , \tau_p^{\alpha} \, ] \}.    
\end{gathered} 
\end{equation} 

We define now  the test function 
\ $f \in  C^{2,\beta}( \, \Sbreve' \, ) $, \ 
by  
\begin{equation} 
\label{Ef}
\begin{gathered} 
f := \fp = \ubreveM 
\quad \text{ on } \quad 
\Sbreve'     \setminus \disjun_{p \in L} \Kbr_p ,  
\\ 
f := \fp - \Rhat_p \, \LchiM \fp 
\quad \text{ on } \quad 
\Kbreve'_p := \Sbreve' \cap \Kbr_p  \quad (\forall p\in L), 
\end{gathered} 
\end{equation} 
where \ $\Rhat_p\, : C^{0,\beta}( \, \Kbr'_p \, ) \to  C^{2,\beta}( \, \Kbr'_p \, ) $ \ is a linear map 
defined by (recall \ref{CKh}) 
\begin{equation} 
\label{ERhat}
\begin{gathered} 
\Rhat_p \, E \, :=  \, \Psibold\left [ \tau_p^{\alpha}/2 , \tau_p^{\alpha} ; \rho \, \right]      
( \, \up \, , \, 0      \, )    
\, = \,   
\psihat_1 \up    \quad \text{ on } \Kbreve'_p,    
\\      
\up \, := \, \Rcal_{\Kbr'_p} \Psibold\left [ 2\tau_p^{5\alpha} , 3\tau_p^{5\alpha} ; \rho \, \right]( \,  E                  \, , \, 0 \, )     
\, = \, \Rcal_{\Kbr'_p} \psihat_2 E  \quad \text{ on } \Kbreve'_p,    
\end{gathered} 
\end{equation} 
where 
$\psihat_1$ and $\psihat_2$ are smooth cutoff functions on $\Mbreve$ 
supported    on \ $\disjun_{p\in L} \, \Kbreve_p $, \    
with detailed definition as in           \ref{EPsibold} 
so that  
\  $\forall p\in L$ \  
\ $\psihat_1 = 1$ \  on  \  $\Kbr_p\setminus \Abreve'_p$, \ 
 \  $\psihat_2=1$  \  on 
 \  $\{ q\in \Kbreve_p \, : \, \rho(q) \le 2 \tau_p^{5\alpha} \}$,  \     
\  $\psihat_2=0$ \   
on  \  $\{ q\in \Kbreve_p \, : \, \rho(q) \ge 3 \tau_p^{5\alpha} \}$,   \   
and so 
\ $\psihat_1 \psihat_2= \psihat_2$  \   and   \     $\supp(1-\psihat_2) \subset \Mbreve_\gr$. 
We clearly have then 
\ $f=0$ \ on \   $\partial\Sbreve'$, 
\ $f=\ubreve$ \ on 
\ $\Sbreve'     \setminus \disjun_{p \in L} \Kbr_p$, \    
and \     $\forall p\in L$ 
\begin{equation} 
\label{Efup}
\begin{gathered} 
f \, = \, f' - \psihat_1 \Rcal_{\Kbr'_p} \big( \psihat_2 \, \LchiM f' \big) 
\qquad \text{on} \quad \Kbr'_p,  
\\             
\LchiM f \, = \, (1-\psihat_2) \, \LchiM \ubreve \, + \, [ \psihat_1 , \LchiM ] \, \Rcal_{\Kbr'_p} \big( \psihat_2 \LchiM f' \big) 
\qquad \text{on} \quad \Kbr'_p,  
\end{gathered} 
\end{equation} 
where we used 
\ $f'=\ubreve$ \  on the support of \ $1-\psihat_2$. \   
We define now \  $  E_j \in \C^{\infty}(\Sbreve')$ ($j=1,2,3$)  \ supported on
\ $\disjun_{p\in L} \, \Kbreve'_p $\    
by taking $\forall p\in L$ 
on $\Kbreve'_p$      
\begin{equation} 
\label{EE}
\begin{gathered} 
E_1 \, := \, \psihat_2 \LchiK f'_\avg, 
\quad  
E_2 \, := \, \psihat_2 \LchiK f'_\osc, 
\quad  
E_3 \, := \, \psihat_2 (\, \LchiM - \LchiK \, )  f',  
\\
\up \, := \, \Rcal_{\Kbr'_p} (E_1+E_2+E_3) 
\, = \, \Rcal_{\Kbr'_p} \big( \psihat_2 \, \LchiM f' \big).      
\end{gathered} 
\end{equation} 
By the above and \ref{dLchiM} we have the decompositions 
on \      $\Sbreve'$ 
\begin{equation} 
\label{EEu}
\begin{gathered} 
\psihat_2 \, \rho^2 \, \LcalM f' = 
\psihat_2 \, \LchiM f' = E_1+E_2+E_3, 
\qquad f \, = \, f' - \psihat_1 \, \up,   
\\ 
\tfrac{|A|^2_\Mbreve+2}2 \LhM   f   \, = \,  
\LcalM f \, = \, (1-\psihat_2) \, \LcalM \ubreve \, + \, [ \psihat_1 , \LcalM ] \, 
\up.     
\end{gathered} 
\end{equation}

We derive now the estimates we need. 
By \ref{Eugr} and by assuming \  $\alpha< 1/180$, \  
and by \ref{LKh}\ref{LKhj},\ref{LKhf},  we conclude that 
\  $\forall p\in L$ \ 
\begin{equation} 
\label{EugrK} 
\begin{gathered} 
\| \, \LchiM             \ubreve \, \|_{1,0,2; \Xbreve_p^{-1}\Mbreve_{\gr} } \, < \, \taubreve^{14/9},    \qquad  
\\
\| \, \LchiM             \ubreve  - \LchiK             \ubreve \, \|_{1,0,2; \Xbreve_p^{-1}\Mbreve_{\gr} } \, < \, 
C \tau_p |\log \tau_p| . 
\end{gathered} 
\end{equation} 
$E_1$ and $E_2$ are supported on \    $\disjun_{p \in L} \Abreve''_p$  \     
and $E_3$ is supported on \    $\disjun_{p \in L} \Kbreve''_p$  \     
where \    $\Abreve''_p := \Kbr_{p, \rho\in [ \tau_p^{2\alpha}, 3\tau_p^{2\alpha} ]}$   \     
and \    $\Kbreve''_p := \Sbreve' \cap \Kbr_{p, \rho \le                    3\tau_p^{2\alpha} }$.   \     
Moreover \  $\phiext - \ubreve_\avg$ \  satisfies the ODE \    $\LchiK \, ( \phiext - \ubreve_\avg ) = - \LchiK \ubreve_\avg  $ \      
with vanishing initial data by \ref{Ephiext}, 
and by \ref{EE} and \ref{Efp} 
\  $E_1 \, := \, \psihat_2 \LchiK \big( \ubreve_\avg + \psihat_3  \, ( \phiext - \ubreve_\avg ) \big) $,  \ 
\  $E_2 \, := \, \psihat_2 \LchiK \big( \ubreve_\osc \, (1 - \psihat_3)  \, \big) $. \ 
Therefore using 
also   \ref{LKh}\ref{LKhf} for $E_2$ and 
\ref{EugrK} for $E_3$ we conclude that 
\begin{equation} 
\label{EEj}
\begin{gathered} 
\| \phiext - \ubreve_\avg \, : \, C^3(\Abreve''_p ) \| \, \le \, C \, \tau_p^{1+10\alpha} \, |\log \tau_p|,  
\qquad \forall p\in L, 
\\ 
\| E_1 \, : \, C^{0,\beta}(\Abreve''_p ) \| \, \le \, C \, \tau_p^{1+10\alpha} \, |\log \tau_p|  
\qquad \forall p\in L, 
\\ 
\| E_2 \, : \, C^{0,\beta}(\Abreve''_p ) \| \, \le \, C \, \tau_p^{1+5\alpha\gamma}  \,    |\log \tau_p|  
\qquad \forall p\in L, 
\\ 
\| E_3 \, : \, C^{0,\beta}(\Kbreve''_p ) \| \, \le \, C \, \tau_p^{1+10\alpha} \, |\log \tau_p|  
\qquad \forall p\in L. 
\end{gathered} 
\end{equation} 
Clearly \ $ \, [ \psihat_1 , \LcalM ] \, \up =      \, \rho^{-2} [ \psihat_1 , \LchiM ] \, \up $ \       
is supported on $\Abreve'_L := \disjun_{p \in L} \Abreve'_p$  \     
and by \ref{EEj}, \ref{LKh}\ref{LKh3}, \ref{Ez}, and \ref{EEu} we conclude that 
$$
\|  [ \psihat_1 , \LcalM ] \, \up \, : \, C^0 (\Abreve') \| \, \le \, \taubreve^{1+ 5\alpha}.
$$ 
By \ref{Eugr} and \ref{EEu} we conclude then that 
$$
\|  \LhM           f              \, : \, L^2 (\Sbreve',h ) \| \, \le \, 2\taubreve^{1+ 5\alpha}.
$$ 
Since clearly \   $\|f: L^2(\Sbreve',h) \| \ge \taubreve$ \    we conclude that there is at least one eigenvalue 
\     $\left| \lambda_j(V_{0+}^{++},{\LcalhD} , \Sbreve'                          ) \right | < \taubreve^{4\alpha} $.  \      
Let   $\lambda'>0$ be the smallest eigenvalue of the geodesic disc in $\Sph^2  $ with boundary  the circle at latitude \ $\xx ( \sroot/3 )$ \ (recall \ref{Esssxx}).  
Using a decomposition of $\Sbreve'$ as usual we see then that 
\     $\lambda_2(V_{0+}^{++},{\LcalhD} , \Sbreve'                          ) > \lambda' >0 $  \      
and the proof is complete. 
\end{proof} 

\begin{notation}[Ground states]   
\label{D:eigenf} 
For smooth compact $\Sbreve\subset \xibreveomo$ we define for use in the rest of this subsection 
the \emph{normalized eigenfunction corresponding to} 
\      $\lambda_1( \, V_{0+}^{++}, \LcalhD, \Sbreve              \, )$    \       
to be the eigenfunction 
\      $f_1 = f_1( \, V_{0+}^{++}, \LcalhD, \Sbreve              \, )$    \       
uniquely characterized by the following 
(recall \ref{D:eigenb} and \ref{D:eig-res}).        
\begin{enumerate}[label=\emph{(\roman*)}]
\item 
\label{D:eigenf0} 
$f_1\in E_{\lambda_1}( \, V_{0+}^{++}, \LcalhD, \Sbreve              \, )$.  
\item 
\label{D:eigenf1} 
$f_1>0$ on $\inter (\Sbreve)$.      
\item 
\label{D:eigenf2} 
$\|f_1: L^2(\Sbreve,h) \| =1$. 
\end{enumerate} 
\end{notation}

\begin{lemma}[Ground states]   
\label{L:3ef} 
There are absolute constants $\cbar,C>0$ independent of $m$ such that 
for \ $\sss_0,\sss_2\in [\sroot/3,\sroot] $ \  and \      $j=0,2$ \   we have the following estimates. 
\begin{enumerate}[label=\emph{(\roman*)}]
\item 
\label{L:3ef1} 
$\lambda_1( \, V_{0+}^{++},{\LcalhD} , \Sbreve[L_j,\sss_j ] \, ) \, < \, C               $ \quad      
and the corresponding normalized eigenfunction (recall \ref{D:eigenf}) 
\  $f' := f_1( \, V_{0+}^{++},{\LcalhD} , \Sbreve[L_j,\sss_j ] \, ) $ \     
satisfies 
$$ 
\| f'_\osc : C^3( \Sbreve[L_j,\sss_j ] \, , \, \chiK )\| \le \tau_j^{3/2}.  
$$
\item 
\label{L:3ef2} 
$4 <   \lambda_2( \, V_{0+}^{++},{\LcalhD} , \Sbreve[L_j,\sss_j ] \, ) $.  
\item 
\label{L:3ef4} 
$1/100 <   \lambda_2( \, V_{0+}^{++},{\LcalhD} , \Sbreve[\Sigma ,\sss_0, \sss_2]  \, ) $. \    
\item 
\label{L:3ef3} 
If \quad 
$\big|  \lambda_1( \, V_{0+}^{++},{\LcalhD} , \Sbreve[\Sigma ,\sss_0, \sss_2]  \, ) \big| \, < \, 1/m^{3/2}        $, \    
then we have the following where 
\     $f'' := f_1( \, V_{0+}^{++},{\LcalhD} , \Sbreve[\Sigma ,\sss_0, \sss_2]  \, )  $ \    
denotes the corresponding normalized eigenfunction (recall \ref{D:eigenf}).  
\begin{equation*} 
\begin{gathered} 
\tfrac\partial{\partial \sss} f''_\avg  \, \Sim^{\cbar} \, 1/m     \qquad \text{at} \quad \partial \Sbreve[L_0,\sss_0 ] \subset \partial \Sbreve[\Sigma ,\sss_0, \sss_2],  
\\ 
\tfrac\partial{\partial \sss} f''_\avg  \, < {\cbar} /\sqrt{m}     \qquad \text{at} \quad \partial \Sbreve[L_2,\sss_2 ] \subset \partial \Sbreve[\Sigma ,\sss_0, \sss_2],  
\\ 
\| \, f''_\osc \, : \, C^3( \,  \Sbreve[L_j,\sss_j+1 ]  \cap   \Sbreve[\Sigma ,\sss_0, \sss_2]  \, , \, \chiK \, ) \, \| \le \tau_j^{3/2} . 
\end{gathered} 
\end{equation*} 
\end{enumerate} 
\end{lemma} 

\begin{proof} 
\ref{L:3ef1} 
We consider the operator \ $\Lcal_{\gnuK}^D = (\Delta_{\gnuK} +2, \Bcal)$ \ with $\Bcal$ restriction to the boundary, that is specifying Dirichlet data. 
The statement clearly holds for 
\  $\lambda_1( \, V_{0+}^{++},\Lcal_{\gnuK}^D , \Sbreve[L_j,\sss_j ] \, ) \, < \, C               $ \     
and its corresponding normalized eigenfunction which is rotationally invariant. 
By  comparing then using \ref{LKh}\ref{LKh1} in order to control the error we complete the proof.  

\ref{L:3ef2} 
\  $4  \, < \, \lambda_2( \, V_{0+}^{++},\Lcal_{\gnuK}^D , \Sbreve[L_j,\sss_j ] \, ) $ \     
by monotonicity of domain and comparing with the corresponding eigenvalue of the hemisphere. 
We complete the proof by using \ref{LKh}\ref{LKh1} as for \ref{L:3ef1}.  

\ref{L:3ef4} 
We use the subdivision with disjoint interiors 
$$
\Sbreve[\Sigma ,\sss_0, \sss_2] \, = \, \SbreveS \, \cup \, ( \Sbreve[\Sigma ,\sss_0, \sss_2] \cap \Sbreve_{L_0} )  \, \cup \, ( \Sbreve[\Sigma ,\sss_0, \sss_2] \cap \Sbreve_{L_2} )  
$$ 
and we impose Neumann conditions on the new boundaries (that is $\partial \SbreveS$)  so that we can compare with the corresponding eigenavalues. 
As usual then we can approximate by replacing with the operator \ $\Delta_{\Sph^2}+2$ \  on spherical caps or $\Sigma$. 
Using also \ref{LapprS}\ref{LapprS1} the eigenvalue bound follows. 

\ref{L:3ef3} 
In the rest of the proof we simplify the notation by using \    $S:= \Sbreve[\Sigma ,\sss_0, \sss_2]$. \   
By \ref{dLchi} and \ref{LKh}\ref{LKhj} we conclude that \ $f''_\avg$ \  satisfies approximately the ODE \ $\LchiK f''_\avg =0$. 
Comparing with \ $\phio$ \  (recall \ref{Lphie}) and assuming \  $\delta'$ \  small enough, 
we conclude that \quad $\frac{\partial}{\partial \sss} \log f''_\avg(\sbar) \Sim 1$, \quad  where $\sbar$ is chosen large enough independently of $m$,  
so that we can conclude by the ODE again that the change of 
\quad $\frac{\partial}{\partial \sss} f''_\avg$ \quad  is appropriately small to ensure that 
at \     $\partial \Kbr_{\pp_0}$ \  we have again \quad 
$ \frac{\partial}{\partial \sss} \log f''_\avg(\sbar) \Sim m^{-1/2}$. \quad 
Let 
\quad $A:= \left. f''_\avg \right|_{\partial  \Kbr_{\pp_0}}$, \quad  that is the average of the eigenfunction at the boundary of an equatorial bridge. 
By arguments similar to arguments in the proof of \ref{Pcos4} (but simpler) we can establish that at \  $\Mbreve_{\xx=1/m}$ \ 
(using the notation in \ref{Dcyl} and \ref{Esssxx}) we have 
\begin{equation} 
\label{EfA}  
\left. f''_\avg \right|_{\xx=1/m} \Sim A, \qquad 
\left. \tfrac{\partial}{\partial \sss} f''_\avg \right|_{\xx=1/m} \,         \Sim \, m^{1/2} A. 
\end{equation} 
Similarly we can prove that 
$$
\left. f''_\avg \right|_{\xx=\pi/2 - 1/m} \Sim B, \qquad 
\left. \tfrac{\partial}{\partial \sss} f''_\avg \right|_{\xx=\pi/2 - 1/m} \,         \Sim \, - m^{-1/2} B, 
$$
where 
\quad $B:= \left. f''_\avg \right|_{\partial  \Kbr_{\pp^0}}$, \quad  
that is the average of the eigenfunction at the boundary of a polar bridge.

Let 
$\lambda := \lambda_1( \, V_{0+}^{++},{\LcalhD} , \Sbreve[\Sigma ,\sss_0, \sss_2]  \, ) $, \    
and \ $\stil_0:= \left. \sss \right|_{\xx=1/m} = \tfrac1m+O(1/m^2)$. \ 
We fix \  $\stil_1$  \   large enough so applying \ref{LphiDN} and using \ref{EfA} we have for $m$ large enough 
\begin{multline*} 
\left. f''_\avg \right|_{\xx=1/m} \, \phiN[\lambda,\stil_0 ] \, + \, 
\left. \tfrac{\partial}{\partial \sss} f''_\avg \right|_{\xx=1/m} \,         \phiD[\lambda,\stil_0 ] 
\\ = \, 
A_1 \, \phien[\lambda,\stil_1 ] \, + \, A_2  \, \phisn[\lambda,\stil_1], 
\end{multline*} 
with \ $A_1\Sim A\sqrt{m}$ \  and \ $A_2\Sim        -A$. \    
Hence we also have 
(recall \ref{DSMbr}) 
\quad $B< \cbar A \sqrt{m}$, \quad       
$\|f'':L^2(\Sbreve_{L_0})\|^2\Sim m A^2/m = A^2$,  \quad  
$\|f'':L^2(\Sbreve_{\Sigma})\|^2\Sim m A^2$,  \quad  
and \quad 
$\|f'':L^2(\Sbreve_{L_0})\|^2\, < \, \cbar'  A^2 $.  \quad  
By the normalization of \ $f''$ \ we conclude that \  $A\Sim1/\sqrt{m}$ \  and this implies the result. 
\end{proof}

\begin{cor}[Derivatives of first eigenvalues]
\label{L:3ed} 
\quad 
There is an absolute 
\linebreak 
constant $\cbar'>0$ independent of $m$ such that 
for \ $\sss_0,\sss_2\in [\sroot/3,\sroot] $ \  and $j=0,2$ \   we have the following.     ` 
\begin{enumerate}[label=\emph{(\roman*)}]
\item 
\label{L:3ed1} 
$\frac{d}{d\sss_0}\lambda_1( \, V_{0+}^{++},{\LcalhD} , \Sbreve[L_j,\sss_j ] \, ) \Sim^\cbarp -1  $.    
\item 
\label{L:3ed2} 
$ \frac{\partial}{\partial\sss_0}\lambda_1( \, V_{0+}^{++},{\Lcal}^D, \Sbreve[\Sigma ,\sss_0, \sss_2]  \, ) \Sim^\cbarp 1/m   $. 
\item 
\label{L:3ed4} 
$0 \,  <   \, \frac{\partial}{\partial\sss_j}\lambda_1( \, V_{0+}^{++},{\LcalhD} , \Sbreve[\Sigma ,\sss_0, \sss_2]   \, ) \, < \,   \cbarp/m   $.   
\item 
\label{L:3ed3} 
$ \big|  \, \lambda_1( \, V_{0+}^{++},{\LcalhD} , \Sbreve[\Sigma ,\sss_0, \sss_2]   \, ) \, \big| \, < \,   \cbarp/m   $.   
\end{enumerate} 
\end{cor} 

\begin{proof} 
By the classical Hadamard formula (see for example \cite{grinfeld}),  
for a simple eigenvalue $\lambda$ with corresponding eigenfunction $f$ on a domain $\Omega$ depending smoothly on $s$, we have     
$$
\frac{\partial\lambda}{\partial s} 
\, = \, 
\int_{\partial\Omega} \big({\tfrac\partial{\partial s}\partial\Omega}\big) \, |\vec\eta f |^2 \, d\sigma_h \  \bigg/ \int_\Omega f^2 dV_h , 
$$ 
where ${{\frac\partial{\partial s}\partial\Omega}}$ denotes the normal velocity of $\partial\Omega$, 
$\vec{\eta}$ is the unit normal to $\partial\Omega$, and $d\sigma_h$, $dV_h$ denote the induced measures from the metric $h$. 
Notice that the numerator of the right hand side is conformally invariant. 
It is straightforward then to conclude the proof by referring to \ref{L:3e0} and \ref{L:3ef}. 
\end{proof}

\begin{prop}[The third eigenfunction in $V_{0+}^{+\pm}$]    
\label{P3rd} 
There is a constant $\cbar_0>0$ independent of $m$ such that 
\qquad 
$$ 
\lambda_3(V_{0+}^{+-},\Lcalh) > \lambda_3(V_{0+}^{++},\Lcalh) \Sim^{\cbar_0} 1/m .  
$$ 
\end{prop} 

\begin{proof} 
We define for \ $\sss_0,\sss_2\in [\sroot/3,\sroot] $ \ 
\begin{equation*}
\begin{aligned} 
\Lambda_0( \sss_0, \sss_2 ) 
\, &:= \, \lambda_1( \, V_{0+}^{++},{\LcalhD} , \Sbreve[\Sigma , \sss_0, \sss_2 ]  \, ) \, - \, \lambda_1( \, V_{0+}^{++},{\LcalhD} , \Sbreve[L_0,\sss_0 ] \, ),  
\\ 
\Lambda_2( \sss_0, \sss_2 ) 
\, &:= \, \lambda_1( \, V_{0+}^{++},{\LcalhD} , \Sbreve[\Sigma ,\sss_0, \sss_2]  \, ) \, - \, \lambda_1( \, V_{0+}^{++},{\LcalhD} , \Sbreve[L_2,\sss_2 ] \, ).  
\end{aligned} 
\end{equation*} 
From \ref{L:3e0} and \ref{L:3ed} we conclude that 
$\lambda_1( \, V_{0+}^{++},{\LcalhD} , \Sbreve[\Sigma ,\sroot, \sroot]   \, ) \Sim 1/m   $, 
and hence 
\begin{equation*}
\begin{aligned} 
\Lambda_0( \sroot , \sroot  ) \Sim^\cbar 1/m ,  
\qquad  \quad 
\Lambda_2( \sroot , \sroot  ) \Sim^\cbar 1/m .  
\end{aligned} 
\end{equation*} 
By \ref{L:3ed} we have 
\begin{equation*}
\begin{aligned} 
&\frac{\partial \Lambda_0 }{\partial \sss_0}( \sss_0, \sss_2 ) \, \Sim -1 ,  
\qquad \qquad \quad 
&&0  <  \frac{\partial \Lambda_0 }{\partial \sss_2}( \sss_0, \sss_2 ) \,  <  \frac Cm , & 
\\
&0  <  \frac{\partial \Lambda_2 }{\partial \sss_0}( \sss_0, \sss_2 ) \,  <  \frac Cm ,  
\qquad \qquad 
&& \frac{\partial \Lambda_2 }{\partial \sss_2}( \sss_0, \sss_2 ) \, \Sim -1 . & 
\end{aligned} 
\end{equation*} 
We conclude then that there is $(\sss''_0,\sss''_2)$ such that 
$$ 
\Lambda_0( \sss''_0, \sss''_2 ) \, = \, \Lambda_2( \sss''_0, \sss''_2 ) \, = \, 0,   
$$ 
and hence using \ref{L:3ed} again we conclude that 
\begin{multline}
\label{E3rd-pr2} 
\lambda_1(V_{0+}^{++},{\LcalhD} , \Sbreve[L_0,\sss''_0] )  \, = \, 
\lambda_1(V_{0+}^{++},{\LcalhD} , \Sbreve[\Sigma ,\sss''_0, \sss''_2] ) \\ = \, 
\lambda_1(V_{0+}^{++},{\LcalhD} , \Sbreve[L_2,\sss''_2] ) \Sim^{\cbar-1} 1/m . 
\end{multline} 
Using cutoff functions we remove the oscillatory part close to the boundaries matching this way the eigenfunctions at the boundaries 
and constructing this way a smooth approximate eigenfunction of $\Lcal$ on $\Mbreve$. 
We can control the error using \ref{L:3ef} and conclude that there exists an eigenfunction of ${\Lcal}$ in $V_{0+}^{++}$ on $\Mbreve$ whose eigenvalue is $\Sim^{\cbar} 1/m$. 
By \ref{Lapp}\ref{Lapp1} we conclude that this    eigenvalue has to be the third eigenvalue, and hence \ $\lambda_3(V_{0+}^{++},\Lcalh) \Sim^\cbar 1/m $.  

The inequality \  $\lambda_3(V_{0+}^{+-},\Lcalh) > \lambda_3(V_{0+}^{++},\Lcalh)$ \  follows by the variational characterization of the eigenvalues on the fundamental domain corresponding to  
$V_{0+}^{+\pm}$ 
and this completes the proof. 
\end{proof}

\subsection*{Index and nullity in $V^{++}_{\mu}$ ($\mu\ne0$) of $\xibreveomo$ } 
\nopagebreak

\begin{prop}[Eigenvalue signs in $W_{\mi+}^{++}$ ($\mi\in {[2, m/2]} $) ]    
\label{Pmu++} 
$\forall \mi\in\mtwo$ we have 
\qquad 
$\lambda_2^\C(W_{\mi}^{++},\Lcalh) = \lambda_2(V_{\mi\pm}^{++},\Lcalh) < 0$,  
\qquad 
and if $m\in2\Z$,  then 
\qquad 
$\lambda_1(V_{\mh-}^{++},\Lcalh) < 0       $. 
\end{prop} 

\begin{proof} 
The proof is by constructing two test functions in $W_{\mi}^{++}$ for each $\mi\in\mtwo$ with $\mi\ne m/2$,     
or just one for $\mi  = m/2$.     
In this proof and the next all functions we consider are even      with respect to $\refl_{\Sigma^0}$ and $\refl_{\Sigma^{\pi/2}}$. 
Similarly to \ref{Lapp}\ref{Lapp4}, 
the first test function when $\mi\ne m/2$    
is defined to be 
\qquad 
$\fo := \Ecal_{\UY_0,\mu} \foY$,  
\qquad 
where $\foY$ is a truncated constant supported on 
\qquad 
$\Kbr_{\pp_0}\subset \UY_0:=\xiomo  \cap \YY_0$ 
\qquad 
(recall \ref{NYY} and \ref{LEmu}). 
We define the second test function, or the only one when $\mi  = m/2$,      
to be 
\qquad 
$\ft  :=\Ecal_{\UY_0,\mu} \ftY$, 
\qquad 
where \ $\ftY=J^\CC$ \ on \ $\UY_0$ \ and so vanishes on \ $\partial \UY_0$ \ by the symmetries and satisfies \ref{LEmu}\ref{LEmu3}  trivially. 
\ $\ft  $ \ has Rayleigh quotient \ $0$ \ but because of its derivative jumps at \ $\partial_\pm \YY_0$,\  
it can be improved to give a negative Rayleigh quotient, while retaining its symmetry properties and its orthogonality to the first test function. 
This proof can be viewed as analogous to the use in \cite{LindexI} of a Montiel-Ros argument \cite{montiel-ros91}. 
\end{proof} 

\begin{prop}[Third eigenvalue in $W_{1}^{++}$ and $ V_{1\pm}^{++}$ ]    
\label{Pcos3} 
$\lambda_3^\C(W_{1}^{++},\Lcalh) = \lambda_3(V_{1\pm}^{++},\Lcalh) < 0          $. 
\end{prop} 

\begin{proof} 
The idea of this proof is similar to the previous one complicated by the presence of an extra sliding at the polar bridge. 
In this spirit we introduce the decomposition 
\quad  $\xibreveomo      = U_0 \cup U_2$  \quad   
into two compact domains defined by (recall \ref{CJnodal}) 
\quad $U_0:= \Sph^3_{|\xx|\le   \xx_-} \bigcap \xibreveomo $  \quad      
and 
\quad $U_2:= \Sph^3_{|\xx|\ge   \xx_-} \bigcap \xibreveomo $  \quad      
so that 
\quad $\Pi_\Sigma ( \partial U_0 ) = \CC^\parallel_{\xx_- } \cup \CC^\parallel_{-\xx_- } $.       
By \ref{CJnodal} we have 
(recall \ref{D:eig-res}) 
\qquad 
$\lambda_1(V_{1-}^{++},{\LcalhD} , U_2 ) <0 $, 
\qquad 
and therefore it is enough to prove that 
\qquad 
$\lambda_2^\C( W_{1}^{++} ,\LcalhD, U_0 ) <0 $. 
\qquad 
We achieve this by constructing two test functions in \quad $W_{1+}^{++}[U_0]$ \quad as follows. 

Let 
\quad 
$\UY_0:=\xiomo  \cap \YY_0$,  
\quad 
and 
\quad     $\foY,\ftY$, \quad 
be as in the proof of \ref{Pmu++},     
and in analogy we define 
\quad 
$\fo   := \Ecal_{\UY_0, 1 } \foY$,  
\quad 
and 
\quad 
$\ft   := \Ecal_{\UY_0, 1 } \ftY $.   
We do not need to modify \   $\fo  $ \   because it already vanishes on \   $\partial_2U_0$.\    
The second test function is defined to be 
\qquad 
$\fhat'_2 := \fhat_2 \cdot \psicut[\sss_-, \sss_- -1 ] \circ \sss $, 
\qquad 
where 
\qquad 
$\fhat_2 := \Ecal_{\UY_0, 1 } \fhattY  $  
\qquad 
is a modification of \   $\ft   $ \   defined below. 
Note that the truncation is needed to ensure that the Dirichlet condition is satisfied at \      $\partial_2U_0$.   \       
The reason for the modification is to lower the Rayleigh quotient in order to compensate for any increase due to the truncation. 

We recall first that the nodal lines of \ $J^\CC$ \ on \ $\xiomo$ \, are given by \quad $\xiomo\cap\Sph^3_{\yy= j \pi/m}$ \      ($j\in Z$). \quad 
In between these nodal lines \ $J^\CC$ \ does not change sign by \ref{Lapp}\ref{Lapp2}. 
Using standard arguments based on approximating by harmonic functions in the $\chiK$ metric, 
separation of variables, the symmetries, and \ref{dgtSS}\ref{dgtSS1},    
we obtain the following estimate 
(recall \ref{D:norm-g}, \ref{Ecyl}, \ref{Ncyl} and \ref{Dcyl}   
and note that $\Usq$ does not intersect the bridges or the nodal lines). 
\begin{equation} 
\label{EJest} 
\begin{gathered} 
\| \, J^\CC \, : \, C^{3}( \xiomo \setminus \Pi_\Sigma^{-1} D_{L}^\Sigma(\pi/7m) \, , 1 , m^2 \chiK , e^{- m \ssst /2 } ) \, \| \, \le  \, C \, A, 
\\ 
\text{ where } \qquad A:= \inf_{\Usq  } |J^\CC|, \qquad \Usq  := \Pi_\Sigma^{-1} \Sigma^0_{|\xx|\le 1/m, \, |\yy- \pi/2m| \le \pi/3m }.  
\end{gathered} 
\end{equation} 

We define now a neighborhood 
\quad $\UY_0' := \xiomo  \cap \Sph^3_{\yy\in [ -\pi/2m , \pi/2m ] } \supset \Kbr_{\pp_0}$ \quad  
of the nodal line \quad $\xiomo\cap\Sph^3_{\yy= 0      }$, \quad 
a neighborhood 
(not intersecting the equatorial bridges)  
\quad $\UY'_1 := \xiomo  \cap \Sph^3_{\yy\in [ \pi/2m , 3\pi/2m ] } $ \quad 
of the nodal line \quad $\xiomo\cap\Sph^3_{\yy= \pi/m}$, \quad  
and boundary components 
\quad $\partial_- \UY'_1  :=  \xiomo  \cap \Sph^3_{\yy = \pi/2m } $,   \quad 
\quad $\partial_+ \UY'_1  :=  \xiomo  \cap \Sph^3_{\yy = 3\pi/2m } = \refl_{\Sigma_{\pi/m}} \partial_- \UY'_1  $,   
\quad so   that \quad 
$\partial \UY'_1  = \partial_- \UY'_1  \cup \partial_+ \UY'_1 $.  \quad 
Note that \quad $\UY_0'\cup\UY_1'$ \quad is a fundamental domain for the action of \ $\rot^\CC_{2\pi/m}$  \  which contains two nodal lines, 
unlike \ $\UY_0$ \ which contains three with two on the boundary. 
Moreover since \quad $\UY'_1 \subset \Mbreve_{\gr}$ \quad is a narrow strip, 
it is easy to see \quad $\lambda_1( \, W_{1}^{++}, \LcalhD, \UY'_1 \,) > 0$  \quad  (recall \ref{D:eig-res}).
We can define then 
\   $\fhat_2$ \   uniquely by requiring 
\begin{equation*} 
\begin{gathered} 
\fhat_2=\fhattY  \, := \, \ft  =\ftY = J^\CC   \quad \text{ on } \quad \UY'_0  \subset \UY_0 , 
\\ 
\LcalM \fhat_2 = 0 \quad \text{ on } \quad \UY'_1, 
\qquad  
\fhat_2= \ft   \quad \text{ on } \quad  \partial \UY'_1, 
\end{gathered} 
\end{equation*} 
so that 
\   $\left. \fhat_2 \right|_{\UY'_1} $ \   is a smooth solution to a Dirichlet problem and \   $\fhat_2 $ \   is continuous and piecewise smooth on \    $\xibreveomo$.

We define now 
\quad $\ftilde_{2}  \, := \,  e^{-\pi \ii/m} \fhat_2$ \quad {on} \      $\xibreveomo$ .       
By \ref{Lderot}\ref{Lderot4},  the definitions, 
the symmetries of $J^\CC$, and the uniqueness of $\fhat_2$,
we have then 
\begin{equation*} 
\begin{gathered} 
\ftilde_2 \, = \, e^{-\pi \ii/m} \fhat_2 \, = \, e^{\pm\pi \ii/m} \, J^\CC \qquad \text {on} \qquad \partial_\pm  \UY'_1, 
\\ 
J_\CC   \, = \, - J_\CC  \circ \refl_{\Sigma_{\pi/m}}  
\quad \text{and} \quad  
\overline{\ftilde_2 }   \, = \, - \ftilde_2 \circ \refl_{\Sigma_{\pi/m}}  
\quad \text{on} \quad  \UY'_1. 
\end{gathered} 
\end{equation*} 
We conclude then that there is $\R$-valued \quad $\ftildep\in C^\infty(\UY'_1)$ \quad satisfying 
\begin{equation*} 
\begin{gathered} 
\ftilde_2 \, = \, 
\cos \tfrac \pi m  \, J^\CC \, + \, \sin \tfrac \pi m  \, \ftildep  \, \ii \quad \text {on} \quad  \UY'_1, 
\\ 
\text{where} \quad \ftildep =\ftildep \circ \refl_{\Sigma_{\pi/m}} \quad \text{and} \quad \Lcalh \ftildep =0 \quad \text {on} \quad  \UY'_1, 
\\ 
\ftildep  = J^\CC \quad \text {on} \quad  \partial_- \UY'_1.  
\end{gathered} 
\end{equation*} 
We have then that there is an absolute constant \ $C'>0$ \ such that 
\begin{equation} 
\label{Emodif} 
\begin{gathered} 
\int_{\UY'_1} \, |\ft|^2 \, dV_\hM 
\, - \,
\int_{\UY'_1} \, |\fhat_2|^2 \, dV_\hM 
\hfill \phantom{kkkkkkkkkkkk} 
\\ 
\phantom{kkkkkkkkkkkk} \hfill 
\, = \, 
\sin^2 \tfrac \pi m  \, \int_{\UY'_1} \left(  (J^\CC)^2 - \ftilde_+^2     \right) dV_\hM  \, < \, - C' \, A^2/m^4  , 
\\ 
\int_{\UY'_1} \, |\nabla \ft|^2 \, dV_\hM 
\, - \,
\int_{\UY'_1} \, |\nabla \fhat_2|^2 \, dV_\hM 
\hfill \phantom{k} 
\\ 
\phantom{k} \hfill 
\, = \, 
\sin^2 \tfrac \pi m  \, \int_{\UY'_1} \left(  |\nabla J^\CC|^2 - | \nabla \ftilde_+ |^2     \right) dV_\hM  \, > \, C' \, A^2/m^3  , 
\end{gathered} 
\end{equation} 
where the equalities follow from the definitions and the inequalities follow from the definition of \  $A$ \    by standard arguments as for \ref{EJest} using also that 
by the definition of \    $J^\CC$ \    we have 
\quad 
$ {\int |\nabla \fhat_2 |^2}= 2{\int |\fhat_2 |^2} \Sim1$. \quad  
We conclude that 
\begin{equation} 
\label{ERayl2} 
\, \frac{\int |\nabla \fhat_2 |^2}{\int |\fhat_2 |^2}  
\, < \, 
2 \,-\,  C' \, A^2 /m^2 . 
\end{equation} 
The exponential decay in \ref{EJest} and a similar estimate for $\ftilde_2$ imply that 
\begin{equation} 
\label{ERayl1} 
\left| 
\, \frac{\int |\nabla \fhat'_2 |^2}{\int |\fhat'_2 |^2}  
\, - \, 
\frac{\int |\nabla \fhat_2 |^2}{\int |\fhat_2 |^2}  \, 
\right| 
< \taubreve^{100} \, A^2  . 
\end{equation} 
Combining with \ref{ERayl2} we complete the proof. 
\end{proof} 

\begin{prop}[Fourth eigenvalue in $W_{1}^{++}$ and $ V_{1\pm}^{++}$ ]    
\label{Pcos4} 
For $ \epsilon$ and $m$ as in \ref{Lapp} we have 
$\epsilon < \lambda_4^\C(W_{1}^{++},\Lcalh) = \lambda_4(V_{1\pm}^{++},\Lcalh) $. 
\end{prop} 

\begin{proof} 
Following the proof of \ref{Lapp}\ref{Lapp3} 
we use now the slightly different decomposition with disjoint interiors $\Mbreve= \Sbreve_{L_2} \bigcup \Sbreve'_{\Sigma}$ 
where $\Sbreve'_{\Sigma}\, := \, \SbreveS \bigcup  \Sbreve_{L_0} $. 
It is enough to prove 
\quad 
$  \epsilon < \lambda_3^\C(W_{1}^{++},\LcalhN, \SbreveS'  ) = \lambda_3(V_{1\pm}^{++},\LcalhN, \SbreveS'  ) $ 
\quad 
by arguing as in the proof of \ref{P:appr}. 
Using the usual arguments and redefining $\epsilon$ it is then enough to prove
$$    
\epsilon < \lambda_3^\C(W_{1}^{++},\Lcalh, \Mbreve''   ) = \lambda_3(V_{1\pm}^{++},\Lcalh, \Mbreve''   ) , 
$$    
where 
\ $\Mbreve''$ \ is \ $\Mbreve$ with the polar bridges removed and the holes left filled in: 
the precise definition is 
\quad 
$\Mbreve'' \, := \, \left( \SbreveS' \setminus \Pi_\Sigma^{-1}( \Sigma_{|\xx|\ge \pi/4 } ) \right) \cup \hat{S}$, 
\quad 
where \  $\hat{S}$ \  is the union of the ghaphs of \  $\pm\hat{u}$ \  over \  $\Sigma_{|\xx|\ge \pi/4 }$ \  with \  $\hat{u}=\ubreve$ \  
on \  $\Sigma_{\pi/3\ge|\xx|\ge \pi/4 }$, \  
\   $\hat{u} =0$ \ on \  $\Sigma_{|\xx|\ge 8\pi/9 }$, 
\  and \  $\hat{u}$ \ is the product of \   $\ubreve$ \    
with a cutoff function on \  $\Sigma_{8\pi/9\ge|\xx|\ge \pi/3 }$. \  
Let $\Mbreve''':= \Mbreve \cap \Sph^3_{|\xx|\le1/10} \subset \Mbreve''$. 
Using the decompositions with disjoint interiors 
\    $\Mbreve''  = \Sbreve_{L_0} \cup \Sbreve''_\Sigma $ \     
and 
\    $\Mbreve'''  = \Sbreve_{L_0} \cup \Sbreve'''_\Sigma $ \     
(recall $\Sbreve_{L_0}  \subset \Mbreve \cap \Mbreve''$)  
we have as before (using appropriate modifications of \ref{LapprS}\ref{LapprS2}) that 
\begin{equation*} 
\begin{gathered} 
-2< \lambda_1^\C(W_{1}^{++},\LcalhN, \Mbreve''  ) = \lambda_1(V_{1\pm}^{++},\LcalhN, \Mbreve''  )<   -2+\epsilon,  
\\ 
-\epsilon < \lambda_2^\C(W_{1}^{++},\LcalhN, \Mbreve''  ) = \lambda_2(V_{1\pm}^{++},\LcalhN, \Mbreve''  ),   
\\ 
-2< \lambda_1^\C(W_{1}^{++},\LcalhD, \Mbreve'''  ) = \lambda_1(V_{1\pm}^{++},\LcalhD, \Mbreve'''  )<   -2+\epsilon,  
\\ 
-\epsilon < \lambda_2^\C(W_{1}^{++},\LcalhD, \Mbreve'''  ) = \lambda_2(V_{1\pm}^{++},\LcalhD, \Mbreve'''  )<   \epsilon,  
\\ 
3 < \lambda_3^\C(W_{1}^{++},\LcalhD, \Mbreve'''  ) = \lambda_3(V_{1\pm}^{++},\LcalhD, \Mbreve'''  ).   
\end{gathered} 
\end{equation*} 

Let \   $f=f_+ + \ii f_- \in        W_{1+}^{++} $ \ with \ $f_\pm\in V_{1\pm}^{++}$ \      
be an \       $\Lcalh$ \  eigenfunction  
of eigenvalue \   $ \lambda \in  ( -\epsilon , \epsilon ) $ \   
on \      $\Mbreve''  $ \     
normalized by  (recall \ref{ELh}) \quad $\|f:L^2(\Mbreve'', h_{\Mbreve''} )\,\| \, = \, 1$. 
We will prove that if $m$ is large enough in terms of given $\epsilon'>0$, 
then $f$ satisfies 
\begin{equation} 
\label{Eiii} 
\|f:C^3(\Mbreve'' \cap \Sph^3_{|\xx|\ge1/20}, h_{\Mbreve''} )\,\| \, < \, \epsilon'.  
\end{equation} 
This allows us to modify   $f$ to a smooth function \ $f'\in C^\infty (\Mbreve''') $ \   satisfying \ $f'=0$ on $\partial\Mbreve'''$,  \   
with 
\     $  f-f'                 $ \   and $\LhM f' + \lambda f' $  supported on \quad 
$       \Mbreve'' \cap \Sph^3_{|\xx|\ge1/20} $, \quad 
and satisfying 
\quad 
$\|f- f':C^3(\Mbreve''' , h_{\Mbreve''} )\,\| \, < \, C \, \epsilon'$   
\quad 
and 
\quad 
$\|\LhM f' + \lambda f' :C^1(\Mbreve''' , h_{\Mbreve''} )\,\| \, < \, C \, \epsilon'$.  
\quad 
If $f_1,f_2$ are then two orthogonal such eigenfunctions,  
the projections of their truncations to the one-dimensional approximate kernel of $\LcalhD$ on $\Mbreve'''$ will form an approximately orthonormal set, a contradiction. 
It only remains then to prove \ref{Eiii}. 

Using interior estimates in the spherical regions and the usual analysis on the rest of $\Mbreve''$ we conclude that 
\begin{equation} 
\label{E3C0} 
\|f:C^3(\Mbreve'' \cap \Sph^3_{|\xx|\ge1/m }, h_{\Mbreve''} )\,\| \, < \, C \, m^3, 
\qquad  
\|f  :C^0(\Mbreve'')\|  \, < \, C . 
\end{equation} 
Using \ref{dgtSS}\ref{dgtSS1}\ref{dgtSS1b}, the symmetries, and separation of variables for $\Lcal_\Sigma$ we conclude that (recall \ref{Ecyl}) 
\begin{equation} 
\label{Eiii1} 
\|f_{(1\perp)} :C^3(\Mbreve'' \cap \Sph^3_{|\xx|\ge1/20}, h_{\Mbreve''} )\,\| \, < \, \taubreve^{7/9}.  
\end{equation} 
We define the operator \   
$\Lcal := \Delta_\Sigma+2+\lambda$ \ 
which can be applied to either functions on domains of $\Sigma$ or functions on domains in the graphical region of \  $\Mbreve''$ \   in analogy with convention  \ref{dgtSS}\ref{dgtSS1}\ref{dgtSS1z}. 
By an analogous statement to \ref{LSh}, the bound on $\lambda$, and interior Holder estimates for $f$ using \ref{E3C0}, 
we have 
\   $\| \Lcal f : C^1( \Mbreve'' \cap \Sph^3_{|\xx|\ge1/m } ) \| < C m^3 \taubreve^{7/9} $. \     
By averaging to obtain the ODE  \   $\Lcal f_{(1\parallel)} \approx 0$, \ and using the implied bound for  \    $\Lcal f_{(1\parallel)}$,  \   
we obtain 
\begin{equation} 
\label{Eiii2} 
\begin{gathered} 
\|f:C^3(\Mbreve'' \cap \Sph^3_{|\xx|\ge1/20}, h_{\Mbreve''} )\,\| \, < \, C\, ( \, |f_{\parallel}(1/m)| + \taubreve^{7/9} \, ), 
\\      
\left| \tfrac{\partial}{\partial\xx}f_{\parallel}(1/m ) \right|  \, < \, C\, ( \, \epsilon |f_{\parallel}(1/m)| + \taubreve^{7/9} \, ), 
\end{gathered} 
\end{equation} 
where the function \ $f_{\parallel} :[\tau_0^{2\alpha}, \pi/2]\to\R$ \ is defined by \ $f=   f_{\parallel} (\xx) e^{\ii\yy}$.

Let $p\in L_0$ and following \ref{dnorm}\ref{dnormS} we define 
\ $\Mbreve_{(p)} := \Mbreve_{\rr=\dbold^{\Sigma}_p \circ \Pi_\Sigma} \supset \Kbr_p$, 
so that \     $\Pi_\Sigma ( \, \partial \Mbreve_{(p)} \, ) \, = \, \partial \Sigma_{(p)} $. \   
Let \    $\Kbr'''_p := \Kbr_{p,\rho\in[\tau_0^{2\alpha} /2  , 2 \tau_0^{2\alpha} ]}  $. \   
Using cutoff functions and $\chibar$ we extend the $\chiK$ metric from $\Kbr_p$ to $\Mbreve_{(p)}$,  
and using then separation of variables we conclude that  
\begin{equation} 
\label{Eiii3} 
\begin{gathered} 
\|f_\osc \, : \, C^3(     \,  \Kbr'''_p   , \chiK       \,  )\,\| \, < \, C\, m^3 \tau_0^{2\alpha} + C\, \tau_0^{1-2\alpha} \, < \, C\, m^3 \tau_0^{2\alpha} ,        
\\      
\|\, \tfrac{\partial}{\partial\sss}f_\avg -a_2 f_\avg( \sbar_2 ) \, : \, C^3(     \,  \Kbr'''_p                   , \chiK       \,  )\,\| \, < \, C\, m^3 \tau_0^{2\alpha},       
\end{gathered} 
\end{equation} 
where \ $\rho(\sbar_2)=\tau_0^{2\alpha}$  \     (using the identification in \ref{dgtSS}\ref{dgtSS2}\ref{dgtSS2b}), and \     $a_2\Sim^5 1/\sqrt{m}$  \    
(by calculating the cylindrical length in the $\chiK$ metric).  

We recall now Green's second identity for $\C$-valued functions in the form 
\begin{equation} 
\label{Egreen2} 
\left\langle  \Lcal u , \phi \right\rangle_{U_0} - 
\left\langle  u , \Lcal \phi \right\rangle_{U_0} = 
\left\langle  \eta  u , \phi \right\rangle_{\partial U_0} - 
\left\langle  u , \eta  \phi \right\rangle_{\partial U_0}.  
\end{equation} 
where $\eta$ is the outward unit normal  of a compact smooth Riemannian domain $U_0$,  
we use 
the standard definition    
\    $ \left\langle  u_1 , u_2  \right\rangle_{\Om} := \Re \int_\Om u_1 \overline{u_2} $ \   and similarly for $\partial U_0$ in place of $U_0$, 
$u,\phi$ are $\C$-valued $C^2$ functions on $U_0 $,  
and $u_1,u_2$ are $\C$-valued $C^0$ functions on $U_0 $.  
We define \   $\UC  :=\Sigma_{|\xx|\le1/m}$ \   and \   $U_0:= \UC   \setminus D_L^\Sigma(\tau_0^{2\alpha})$ \    (recall \ref{tubular}). 
We choose \ $\phi:\UC  \to\C$ \ to be the solution to the Dirichlet problem \ $\Lcal \phi=0$ \  on $\UC  $ \ and \ $\phi= e^{\ii\yy}$ \ on \ $\partial \UC  $; \ 
note that clearly $\phi$ is in $W_{1+}^{++}$.  
By \ref{Egreen2} we have 
\begin{equation} 
\label{Euflux} 
| F_u - F'_u |   \, \le \, \tfrac7m \, \| \Lcal u : C^0 (U_0)\| ,            
\end{equation} 
where \      $F_u, F'_u$ \      can be thought of as \emph{fluxes} of $u$ through \        $\partial \UC \subset \partial U_0 $
and \   $\partial' U_0 := \partial U_0 \setminus \partial \UC$ \ defined by 
\begin{equation} 
\label{EufluxD} 
\begin{gathered} 
F_u  \, := \, 
\left\langle  \eta  u , \phi \right\rangle_{\partial \UC} - 
\left\langle  u , \eta  \phi \right\rangle_{\partial \UC},  
\\       
F'_u \, := \, 
-\left\langle  \eta  u , \phi \right\rangle_{\partial' U_0} + 
\left\langle  u , \eta  \phi \right\rangle_{\partial' U_0}.  
\end{gathered} 
\end{equation} 

We choose now 
\ $ u  :U_0  \to\C$ \ to be the solution to the Dirichlet problem \ $\Lcal  u  =0$ \  on $U_0  $, \ $ u  = 0 = f-\fC $ \ on \ $\partial \UC  $, \ and \ $u= f- \fC $ \  
on \ $\partial' U_0 $, \ 
where 
\ $ \fC :\UC  \to\C$ \ is the solution to the Dirichlet problem \ $\Lcal  \fC =0$ \  on $\UC  $ \ and \ $ \fC = f         $ \ on \ $\partial \UC  $; \ 
note that $u,\fC$ are in   $W_{1+}^{++}$.  
Applying \ref{Euflux} we conclude that \ $F_u=F'_u$. 
In order to obtain estimates on $u$ we reconstruct $u$ by solving a Dirichlet problem using separation of variables on disjoint congruent maximal annuli contained in $\UC  $ 
and centered at each $p\in L$, 
then combining the solutions to get an approximate solution to the Dirichlet problem on $U_0$, and finally correcting. 
The annuli are conformal to standard cylinders of length \  $\alpha\sqrt{2m\,} + O(\log m)$ \ and so we conclude by \ref{Eiii3}, 
and observing that the symmetries imply that the contributions from the connected components of $\partial' U_0$ are all the same,          
$$   
F_u= F'_u \ = \   \tfrac {2\pi m}{\alpha \sqrt{2m\,} + O(\log m) }  (f_\avg (\sbar_2) -\fC(\pp_0)) + O(\tau^{2\alpha} ) , 
$$  
where the identification is through $\Xbreve_{\pp_0}$ and $f_\avg(\sbar_2)\in\R$ by the symmetries. 

On the other hand by \ref{Eiii2} \   
\begin{equation} 
\label{Efluxf} 
|F_f| \, \le \, C\, ( \, \epsilon |f_{\parallel}(1/m)| + \taubreve^{7/9} \, ), 
\end{equation} 
and by the definition of $\fC$ we easily prove that 
$$  
|F_{\fC}| \, \le \, C\, ( \, \frac1m |f_{\parallel}(1/m)| + \taubreve^{7/9} \, ). 
$$ 
Moreover 
\   $u-f+\fC=0$ \  on \    $\partial U_0$ \    
and \ $\Lcal (u-f+\fC) = \Lcal f$ \     
which can be estimated by the eigenfunction equation and 
\ref{dgtSS}\ref{dgtSS1}.   
Estimating the Dirichlet problem we obtain that 
$$  
|F_u-F_f+F_{\fC} | < \taubreve^{7/9}. 
$$ 

Combining we conclude that 
$$
| f_\avg ( \sbar_2) -\fC(\pp_0) + O(\tau^{2\alpha}_0) | \, \le \, C\, ( \, \epsilon |f_{\parallel}(1/m)| + \taubreve^{7/9} \, ) / \sqrt{m} .  
$$
Clearly $|\fC(\pp_0) - f_{\parallel}(1/m) |< 1/m$ and hence  
\begin{equation} 
\label{Efavg} 
\left| f_\avg ( \sbar_2) -  f_{\parallel}(1/m)  \right|      \, \le \, m^{-1} \, + \, m^{-1/2} \,   \left| f_{\parallel}(1/m) \right|.  
\end{equation} 

Combining \ref{EufluxD}, \ref{Eiii3}, and \ref{Efavg}, we conclude that 
$$ 
\left|\, F'_f - 2 \pi a_2 m f_{\parallel}(1/m)  \right| \, < \, C      \, + \, C |f_{\parallel}(1/m)| + C\, m^4 \tau_0^{2\alpha},       
$$ 
which implies that there is \quad $a_3 \Sim^6 2\pi\sqrt{m}$ \quad such that 
$$ 
\left|\, F'_f - a_3 f_{\parallel}(1/m)  \right| \, < \, C.       
$$ 
We apply now \ref{Euflux} a second time choosing $f$ in place of $u$, to obtain 
$$ 
\left| F_f - F'_f \right|   \, \le \, \tfrac7m \, \| \Lcal f : C^0 (U_0)\| \, \le \, \taubreve^{7/9}.  
$$ 
Combining we obtain \quad $f_{\parallel}(1/m) < Cm^{-1/2}$, \quad  
which combined with \ref{Eiii2} confirms \ref{Eiii} and the proof is complete. 
\end{proof}

\subsection*{Index and nullity in $V^{+-}_{\mu}$ ($\mu\ne0$) of $\xibreveomo$ } 
\nopagebreak

\begin{prop}[Eigenvalue signs in $W_{\mi+}^{+-}$ ($\mi\in {[2, m/2]} $) ]    
\label{Pmu+-} 
$\phantom{kk}$ 
$\forall \mi\in\mtwo$ we have 
\quad 
$0 < \lambda_1^\C(W_{\mi}^{+-},\Lcalh) = \lambda_1(V_{\mi\pm}^{+-},\Lcalh) $, 
\quad 
and if $m\in2\Z$,  \quad then 
\quad 
$0 < \lambda_1(V_{\mh+}^{+-},\Lcalh) $. 
\end{prop} 

\begin{proof} 
Let $\mi\in {[2, m/2]} $
and \   $f:= f_+ + \ii f_- $ \   be the ground state in $W_{\mu+}^{+-}$, 
where \   $f_\pm\in V_{\mi\pm}^{+-}$ \  are ground states 
of the same eigenvalue (recall \ref{Lderot}\ref{Lderot4+}),  
and \    $f$ \    is normalized uniquely by 
\begin{equation} 
\label{Efnor} 
\quad \| f : L^2 (\xiomo , \hM) \|^2 \, = \, 4m    \quad     
\end{equation} 
and \quad 
$f_+ \ge 0$  \quad     
on \         $\xiomo\cap\Sph^{3++}_{\yy=0}$ \    in the vicinity of the waist of \ $\Kbr_{\pp_0}$. \      
By the symmetries we have 
\quad on \quad $\xiomo \cap \Sph^3_{\yy=0}$ \quad $\forall j\in \Z$ \quad (recall \ref{N:xx}) 
\begin{equation} 
\label{Efsin} 
f_+\circ\rot_{2j\pi/m}^\CC = \cos {2j\mi{\textstyle\frac\pi m}} \, f_+ ,  
\qquad 
f_-\circ\rot_{2j\pi/m}^\CC = \sin {2j\mi{\textstyle\frac\pi m}} \, f_+ . 
\end{equation} 
By considering then the intersection of the nodal sets of $f_-$ with 
\ $\xiomo \cap \Sph^3_{\yy=2j\pi/m}$  \ 
and 
\ $\xiomo \cap \Sph^3_{\yy=2(j+1) \pi/m}$  \ 
when \      $\sin {2j\mi{\textstyle\frac\pi m}} $ \  
and  \      $\sin {2(j+1) \mi{\textstyle\frac\pi m}} $ \  have different signs, 
it is easy to see that \   $f_-$ \   has at least \   $\mu$ \   nodal domains on \   $\xiomo\cap\Sph^{3++}_+$, \     
and strictly more than \   $\mu$ \   if \   $f$ \   changes sign on 
\         $\xiomo\cap\Sph^{3++}_{\yy=0}$. \     

By Courant's nodal theorem therefore, 
there are at least two eigenvalues (corresponding to eigenfunctions in $V_{-}^{+-}$) 
which are strictly smaller than 
the smallest eigenvalue 
\     $ \lambda_{\min}$  \  
among the ones for $\mu\ge3$. 
Clearly these two eigenvalues will have to be eigenvalues for $\mu=1,2$.      
Therefore if \  $\lambda_{\min} \le 0$,  \ 
then there are at least two negative eigenvalues for $\mu=1,2$.  
Since there is only one negative eigenvalue for $\mu=1$ by \ref{Lapp}\ref{Lapp9}, 
we need to prove the Proposition only for $\mu=2$. 
In the rest of the proof then we study the normalized ground state \  
$f$ \  as above for $\mu=2$; 
we will denote the corresponding eigenvalue by \  $\lambda$. 

If \     $\lambda\le0$, \ then since there is only one negative eigenvalue for \    $\mu=1$, 
we conclude by Courant's nodal theorem that \ $f_-$  \     has       at most two nodal domains on \   $\xiomo\cap\Sph^{3++}_+$, \     
and therefore by an earlier assertion \ $f+$ \ cannot change sign on 
\         $\xiomo\cap\Sph^{3++}_{\yy=0}$. \     
Using this and \ref{Efsin}, we conclude that for \  $j\in [1, m/4]\cap\N$ \          
we have         
\         $\xiomo\cap\Sph^{3++}_{\yy\in [0, 2 j  \pi/m  ]  \ }$ \  
contained in a nodal domain of \  $f_-$. \quad  
By monotonicity of domain this implies that 
\begin{equation*} 
\begin{aligned} 
\lambda_1 \big(V^{+-}, \LcalhD , \xiomo\cap\Sph^{3++}_{\yy\in [0, 2 j  \pi/m  } \big) \, >& \, \lambda,  
\\ 
\text{and so} \quad  
\lambda_1 \big(V^{+-}, \LcalhD , \xiomo\cap\Sph^{3++}_{|\yy|  \, < \,    j  \pi/m  } \big) \, >& \, \lambda . 
\end{aligned} 
\end{equation*} 
We conclude then 
that \ $f_+$ \   does not change sign on (recall \ref{NYY}) 
$$ 
\xiomo\cap  \Sph^{3++}_{|\yy|  \, < \,    j  \pi/m  } \supset \UY^{++}_0 := \xiomo\cap\Sph^{3++}\cap\YY_{0} . \ 
$$
Note that by the symmetries it is enough to estimate \    $f$ \     on \  $\UY^{++}_0$    \      
where we claim the following.

\begin{lemma}[Ground state in $W_{ 2 +}^{+-}$ ($\mi=2$) ]    
\label{Lmu+-} 
For \      $f=f_+ + \ii f_-$ \      as above 
we have the following,  where 
\quad 
$\partial_+ \UY^{++}_0 := \UY^{++}_0 \cap\partial_+ \YY_{0} $    
\quad 
and 
\quad 
$\UY_{\parallel} := \UY^{++}_0 \cap   \Pi_\Sigma^{-1} {\CC^\parallel_{1/m }} $  
\quad 
(recall \ref{NYY} and \ref{N:xx}).  
\begin{enumerate}[label=\emph{(\roman*)}]
\item 
\label{Lmu+-1} 
$f_+\ge0$  \    on \     $\UY^{++}_0 $ \ and \, 
$f_-\ge0$  \    on \     $\UY^{++}_0 \cap \Sph^{3}_+ $.   
\item 
\label{Lmu+-3} 
$f_- \, = \, f_+ \, {\tan 2\pi/m} $      
\    on \     $\partial_+ \UY^{++}_0 $.    
\item 
\label{Lmu+-5} 
$A:= \inf_{\UY_\parallel} f_+ \Sim^{m^3} \tau_0  $. \ 
\item 
\label{Lmu+-2} 
$f_+ \Sim^\cbar \, e^{- 2 \ssst } A $  \    on \     
$\UY^{++}_0 \setminus \Pi_\Sigma^{-1} D^{\Sigma^0}_{\CC}(1/m)   $ \    where \       
$A:= \inf_{\UY_\parallel} f_+$. \ 
\item 
\label{Lmu+-4} 
$\|\, f \,  : \, C^{3}(     \, \UY^{++}_0 \setminus \Pi_\Sigma^{-1} D^{\Sigma^0}_{\pp_0}(1/m)  \,  , 1 , \chiK , e^{- 2 \ssst } ) \, \| \, \le  \, C A    $. 
\end{enumerate} 
\end{lemma} 

We postpone the proof of the lemma and complete the proof of the Proposition assuming the lemma. 
By \ref{CJnodal} we have 
$$ 
0 < \lambda_1^\C(W_{1}^{+-},\LcalhD,U_-) = \lambda_1(V_{1\pm}^{+-},\LcalhD,U_-)  , 
$$ 
where 
\quad 
$U_-:=\xiomo\cap \Sph^3_{\xx\in[0,\xx_-]}$. 
\quad 
It is enough then to prove 
\begin{equation}
\label{Emu+-} 
\lambda_1^\C(W_{1}^{+-},\LcalhD,U_-) \, < \, 
\lambda_1^\C(W_{ 2 }^{+-},\Lcalh).
\end{equation} 
To achieve this we convert the ground state \    $f \in W_{ 2 +}^{+-}$ \   defined above,    
to a test function \     $f_1 \in W_{1}^{+-} [U_-]$   \     defined as follows (recall \ref{LEmu}). 
\begin{equation} 
\label{Ef1} 
\begin{gathered} 
f_1 := \Ecal_{U_-,1} f'_1 \in \left. W_{1}^{+-} \right|_{U_-} ,   
\quad \text{ where } \quad 
f'_1 \in \left. W_{1}^{+-} \right|_{U_-\cap\YY_0},  
\\ 
f'_1  \, := \,  f'_+ \, + \, a'  f'_- \, \ii, 
\qquad 
a' \, := \, \tfrac{\tan \pi/m}{\tan 2\pi/m} = \tfrac12 + O(1/m) , 
\\
f'_\pm \, := \,  f_\pm \, \psi_-, 
\qquad 
\psi_-:= \psicut[\sss_-, \sss_- -1 ] \circ \sss . 
\end{gathered} 
\end{equation} 
The conversion amounts to a truncation using $\psi_-$, followed by a change of frequency from $\mu=2$ to $\mu=1$  (recall \ref{LEmu}). 

To complete the proof it is enough to prove  
\begin{equation} 
\label{Erayl} 
\frac {\disint |\nabla f'_+|^2 + {a'}^2              \disint |\nabla f'_-|^2 } {\disint | f'_+|^2 + {a'}^2              \disint | f'_-|^2 }     
\, < \, \frac {\disint |\nabla f_+|^2 + \disint |\nabla f_-|^2 } {\disint | f_+|^2 + \disint | f_-|^2 },    
\end{equation} 
where all integrals are taken with respect to the $h$ metric  
with the ones on the left Rayleigh quotient over \  $\UY^{++}_0 \cap U_- $ \    and the ones on the right over \  $\UY^{++}_0$. \    
Clearly by \ref{CJnodal} the intersection of   \   $\UY^{++}_0$ \   with the region where the cutoff function changes value has area \   $\Sim^4 1/m$ \   in the $\chiK$ metric.   
In this region  we have   \quad      $e^{-2\sss}\Sim^2 1/\sqrt{m}$  \quad   
and so by \ref{Lmu+-}\ref{Lmu+-4} the \   $C^3(\chiK)$ \   norm of \   $f$ \, is \   $\le C A /\sqrt{m}$. \ 
We conclude that
the Rayleigh quotient for \    $f$ \    increases by at most 
\quad    $O(A^2/{m}^2)$ \quad    by the truncation. 
On the other hand 
by \ref{Lmu+-}\ref{Lmu+-3},\ref{Lmu+-2},\ref{Lmu+-4}      \ $f'_-$ \ 
changes from \ $0$ \ to \ $\Sim A/m$ \  over a distance \ $\Sim 1/m$, \ and hence by standard arguments  
we have \quad $\int | \nabla f'_-|^2 >    C \cbar^2 A^2  /m$.    \quad  
Similarly \quad $\int | f'_-|^2  <   C  A^2 /m^3$.    \quad  
This implies that the change of frequency reduces   the Rayleigh quotient by at least  \quad $C \cbar^2 A^2 /m $  \quad                   
which is \ $> O(A^2/{m}^2)$ \ for large $m$ and so implies \ref{Erayl} as we want. 
It only remains then to prove the lemma. 

\emph{Proof of lemma \ref{Lmu+-}.} 
\ref{Lmu+-1} follows from the more general statements earlier in the proof. \ref{Lmu+-3} follows from \ref{LEmu}. 
Applying standard arguments as in the proof of \ref{Lapp}\ref{Lapp10} following \cite[Appendix B]{kapouleas:1990} 
we conclude that there is \   $\sbars\in [10,20]$ \   
such that for $m$ large enough depending on given \      $\epsilon>0$   \    we have 
\begin{equation} 
\label{Esbars}  
\| (f_+)_{(1\perp)} \, : \, C^3(   \Kbr_{\pp_0,\sss=\sbars}, \chiK) \| \, < \, \epsilon. 
\end{equation} 
(To prove this we first find two such $\sbars$'s satisfying $L^2$ estimates and then upgrade for $\sbars$ inbetween  them using separation of variables and approximation). 
We define then \quad $\Ustar:= \xiomo \setminus (\Kbr_{L_0})_{\sss<\sbars}$. \quad 
By standard arguments we have 
$$
3 < \lambda_1^\C(W_{ 2 }^{+-},\LcalhD, \Ustar ) = \lambda_1(V_{\mi\pm}^{+-},\LcalhD,\Ustar) , 
$$  
and so the Dirichlet problem for \   $\Lcal:= \Lcalh +\lambda $ \   on  \  $\Ustar$  \    has a unique solution for given boundary data. 
We prove then \ref{Lmu+-5}-\ref{Lmu+-4} by reconstructing \    $f$ \   from its boundary data and obtaining estimates as follows. 

Let \   $\Ustarp$ \ be a disjoint union of annuli with \quad $\partial \Ustarp = \partial \Ustar \disjun \partial' \Ustarp $ \quad 
where \quad $\partial' \Ustarp \, := \, \{ q\in \xiomo \, : \, \rr                                    (q) = 1/5m \} $  \quad and \quad  
$\rr \, := \, \dbold^{\Sigma}_{L_0} \circ \Pi_\Sigma$. \quad  
We define \ $\fstarp$ \ to be the unique solution to the Dirichlet problem 
$$ 
\Lcal \fstarp = 0 
\quad \text{on} \quad \Ustarp, 
\qquad \fstarp =f 
\quad \text{on} \quad \partial \Ustar, 
\qquad \fstarp =0 
\quad \text{on} \quad \partial' \Ustarp. 
$$ 
We define now (recall \ref{Epsiab}) \quad 
$
\psistar \, := \psicut[ 1/5m, 1/9m ] \circ \rr 
$ 
and \ $\fstar$ \ to be the unique solution to the Dirichlet problem 
$$ 
\Lcal \fstar =  -\Lcal ( \, \psistar \cdot \fstarp \, ) 
\quad \text{on} \quad \Ustar, 
\qquad \fstar = 0 
\quad \text{on} \quad \partial \Ustar. 
$$ 
Clearly then \quad $f= \psistar \cdot \fstarp + \fstar $ \quad on \quad $\Ustar$ \quad 
and we can produce the required estimates by estimating \ $\fstarp$ \ and \ $\fstar$. \   
By using standard arguments, separation of variables, comparison with the catenmoidal bridges under the \ $\chiK$ \ metric, 
and comparison with \ $\Sigma$ \ it is straightforward (but lengthy) to complete the proof. 
\end{proof}

\subsection*{Conclusion on index and nullity of $\xibreveomo$ } 
\nopagebreak

\begin{prop}[Eigenvalue signs in $V^{-+}$ ]    
\label{P-+} 
\quad 
$\lambda_1(V^{-+},\Lcal ) = -2$, 
\\    
$\lambda_2(V^{-+},\Lcal ) = \lambda_3(V^{-+},\Lcal ) = 0$, 
\quad 
and 
\quad 
$0 < \lambda_4(V^{-+},\Lcal ) $. 
\end{prop} 

\begin{proof} 
We prove first that 
$$
\lambda_1^\C(W^{-+}_\mu,\Lcal ) \ge 0 
\quad 
\forall \mi\in {[1, m/2]} \cap\Z,   
$$
with equality only if $\mu=1$ in which case 
$\lambda_2^\C(W^{-+}_1  ,\Lcal ) > 0$.  
Clearly the imaginary part of an eigenfunction in 
$W^{-+}_{\mu+}$ is in 
$V^{-+}_{\mu-} \subset V^{-+}_{-} $. 
By \ref{Lderot}\ref{Lderot4+}, \ref{LJpp} and Courant's nodal theorem the claim follows. 

By \ref{Lderot} it remains only to prove that 
\begin{equation} 
\label{E-+} 
\lambda_2(V^{-+}_{0},\Lcal ) >0.  
\end{equation} 
Let $f$ be the corresponding eigenfunction. 
We assume without loss of generality that either $f\in V^{-+}_{0+}$ or $f\in V^{-+}_{0-}$.   
We recall that 
$\xiomo^{++}$ is an annulus topologically and its boundary $\partial \xiomo^{++}$ has two connected components: 
$\partial_2 \xiomo^{++}\subset\Sigma^0$ which is the waist of the catenoidal bridge at $\pp^0$, 
and $\partial_0 \xiomo^{++}\subset\Sigma^0 \cup\Sigma^{\pi/2}$     
which lies in the vicinity of $\CC$. 
If $f$ has a nodal domain whose interior does not intersect $\partial_0 \xiomo^{++}$    
we conclude by monotonicity of domain that 
$\lambda_2(V^{-+}_{0+},\Lcal ) > \lambda_1(V^{--},\Lcal )$  
and we apply \ref{Lapp}\ref{Lapp--}. 
If not, it is clear by the topology and the symmetries that $f$ has a nodal domain whose intersection with $\Sph^{3++}$ is contained in 
$$  
\Sph^3_{\yy\in[-2\pi/m , 2\pi/m ]} \, \subset \,  
\Sph^3_{\yy\in[-2\pi/m , \pi-2\pi/m ]} \, = \, \rot^\CC_{-2\pi/m} \Sph^3_+. 
$$
The proof is then completed by using monotonicity of domain, \ref{LJpp}, and Courant's nodal theorem. 
\end{proof} 

\begin{theorem}[Main Theorem] 
\label{Tmain} 
For $m$ large enough the index of $\xibreveomo$ is $2\gamma+5=2m+7$ and its nullity $6$, 
so it has no exceptional Jacobi fields and is $C^1$ isolated. 
\end{theorem} 

\begin{proof} 
We have already seen that on $\xibreveomo$ all eigenfunctions with $0$ eigenvalue are induced by Killing fields. 
We calculate now the index as follows. 
\\  
$\ind(V^{++}) 
\, = \, \ind(V^{++}_{0+}) + \ind(V^{++}_{0-}) + \ind(V^{++}_1) + \ind( \, \Oplus_{\mi\in\mo} V^{++}_\mu \, )  
\, = \, 2 + 0 + 6 + 2(m-3) = 2m+2$ 
by \ref{Lapp}\ref{Lapp1}, \ref{P3rd}; \ref{Lapp}\ref{Lapp2}; \ref{Lapp}\ref{Lapp3}, \ref{Pcos3}, \ref{Pcos4};    \ref{Lapp}\ref{Lapp4}-\ref{Lapp5}, \ref{Pmu++}. 
\\
$\ind(V^{+-}) 
\, = \, \ind(V^{+-}_{0+}) + \ind(V^{+-}_{0-}) + \ind(V^{+-}_1) + \ind( \, \Oplus_{\mi\in\mo} V^{+-}_\mu \, )  
\, = \, 2 + 0 + 2 + 0      = 4$
by \ref{Lapp}\ref{Lapp7}, \ref{P3rd}; \ref{Lapp}\ref{Lapp8}; \ref{Lapp}\ref{Lapp9}; \ref{Lapp}\ref{Lapp10}-\ref{Lapp11}, \ref{Pmu+-}. 
\\
$\ind(V^{-+}) = 1 $ by \ref{Lapp}\ref{Lapp-+}, \ref{P-+}. 
\\
$\ind(V^{--}) = 0 $ by \ref{Lapp}\ref{Lapp-+}. 
\\
By adding and since the genus is \, $\gamma=m+1$ \,  
we conclude that $\ind(V) = 2m+7=2\gamma+5 $.   
\end{proof}

\bibliographystyle{amsplain}
\bibliography{paper}
\end{document}